\documentclass[leqno,11pt]{amsart}
\usepackage{amsmath,amsthm,amssymb,enumerate}
\usepackage{epsfig}
\usepackage{amssymb}
\usepackage{amsmath,xspace}
\usepackage{amssymb}
\usepackage{amsmath,amsthm}
\usepackage[latin1]{inputenc}
\usepackage[T1]{fontenc}
\usepackage{path}
\usepackage{ae,aecompl}
\usepackage{amsfonts}
\usepackage{amsxtra}
\usepackage{bbm,euscript,mathrsfs}
\usepackage{color}
\usepackage{todonotes}

\allowdisplaybreaks

\usepackage{amsfonts}
\usepackage{amsxtra}

\usepackage[notcite,notref]{showkeys}

\relax

\usepackage{enumitem}
\setenumerate{label={\rm (\alph{*})}}

\makeatletter
\long\def\unmarkedfootnote#1{{\long\def\@makefntext##1{##1}\footnotetext{#1}}}
\makeatother

\newcommand{\dz}{\,\mathrm{d}z}
\newcommand{\dy}{\,\mathrm{d}y}
\newcommand{\dx}{\,\mathrm{d}x}
\newcommand{\ds}{\,\mathrm{d}s}
\newcommand{\dt}{\,\mathrm{d}t}
\newcommand{\dr}{\,\mathrm{d}r}

\newcommand{\Div}{\divergence}
\newcommand{\eps}{\bfvarepsilon}
\newcommand{\ep}{\bfvarepsilon}
\newcommand{\R}{\mathbb R}
\newcommand{\N}{\mathbb N}

\newcommand{\rn}{\mathbb R^n}

\newtheorem{defs}{Definition}[section]
\newtheorem{theorem}[defs]{Theorem}

\newtheorem{example}[defs]{Example}

\newtheorem{lemma}[defs]{Lemma}

\newtheorem{corollary}[defs]{Corollary}

\newtheorem{remark}[defs]{Remark}

\title{
Symmetric gradient Sobolev spaces  endowed with rearrangement-invariant norms
}

\author {Dominic Breit, Andrea Cianchi}

\address{Dominic Breit, 
  Heriot-Watt University, Department of Mathematics\\
EH14 4AS Riccarton, Edinburgh, UK} \email{db13@hw.ac.uk}

\address{Andrea Cianchi, Dipartimento di Matematica e Informatica \lq\lq U. Dini"\\
Universit\`a di Firenze,
Viale Morgagni 67/a,
50134 Firenze,
Italy} \email{andrea.cianchi@unifi.it}
\urladdr{}

\date{}

\numberwithin{equation}{section}

\begin{document}
\maketitle

%


%
%
%
%
%
%
\date{}
%
%
%

\begin{abstract} A unified approach to embedding theorems for Sobolev type spaces of vector-valued  functions, defined via their symmetric gradient, is proposed. The Sobolev spaces in question are built upon general rearrangement-invariant norms. Optimal target spaces in the relevant embeddings are determined  within the class of all rearrangement-invariant  spaces.  In particular, all  symmetric gradient Sobolev embeddings into rearrangement-invariant target spaces are shown to be equivalent to the corresponding embeddings for the full gradient  built upon the same   spaces.
A sharp condition for embeddings into spaces of uniformly continuous functions, and their optimal targets, are also exhibited. By contrast, these embeddings may be weaker than the corresponding ones for the full gradient.
Related results, of independent interest in the theory of symmetric gradient Sobolev spaces, are   established. They include  global approximation and extension theorems under minimal assumptions on the domain. A formula for the $K$-functional, which is pivotal for our  method based on reduction to one-dimensional inequalities, is provided as well. The case of symmetric gradient Orlicz-Sobolev spaces, of use in mathematical models in continuum mechanics driven by nonlinearities of non-power type, is especially focused.
\end{abstract}



\setcounter{tocdepth}{1}
\tableofcontents

\unmarkedfootnote{
}
%

\section{Introduction}\label{intro}

In the present paper we offer a theory of Sobolev type spaces, defined in terms of the symmetric gradient of vector-valued functions, and equipped with arbitrary rearrangement-invariant norms on open sets $\Omega \subset \rn$, with $n\geq 2$. Recall that the symmetric gradient $\ep(\bfu)$ of a function $\bfu :\Omega \to \rn$  is defined as the symmetric part of the gradient $\nabla \bfu$ of $\bfu$. Namely,
\begin{equation}\label{symmgrad}
\ep(\bfu)=\frac{1}{2}\big(\nabla\bfu+\nabla\bfu^T\big),
\end{equation}
where $\nabla\bfu^T$ stands for the transpose of $\nabla \bfu$.
\par Mathematical models for a number of physical systems are shaped by the symmetric gradient of a function that describes the   system under consideration. Relevant instances include 
the theory of (generalized)
Newtonian fluids, and the classical theories of
plasticity and nonlinear elasticity. 
\par 
Sobolev type spaces built upon the symmetric gradient provide a natural functional framework for the analysis of systems of partial differential equations which govern these models. These spaces will be called symmetric gradient Sobolev spaces, and consist of those functions $\bfu$ in some  space $X(\Omega)$ such that $\ep(\bfu) \in X(\Omega)$ as well. A space of this kind will be denoted by $E^1X(\Omega)$ in what follows. A corresponding space 
of functions which vanish, in a suitable sense, on $\partial \Omega$, is denoted by $E^1_0X(\Omega)$.
\par Diverse aspects and properties of the spaces $E^1X(\Omega)$ and $E^1_0X(\Omega)$ have been investigated in the literature, especially in the classical case when $X(\Omega)=L^p(\Omega)$ for some $p \in [1, \infty]$, see e.g. \cite{AG,Ba,Fu1,FuS,GmRa2,Haj,Ne,Resh,ts,St,Temam}.
A weakness
 of the existing theory is, however,
that certain results of the same nature rest upon distinct proofs, depending on the space $X(\Omega)$ at hand.
%
Hence, they are not well suited for extensions  encompassing less customary spaces $X(\Omega)$.
\par A paradigmatic example of this situation is supplied by the Sobolev inequality in $E^1_0L^p(\Omega)$. This inequality is well known to involve the same target norm as the  classical Sobolev inequality. For instance, if $\Omega$ is bounded and $p\in [1,n)$, then there exists a constant $C$ such that
\begin{equation}\label{intro1}
\|\bfu\|_{L^{p^*}(\Omega)} \leq C \|\ep(\bfu)\|_{L^p(\Omega)}
\end{equation}
for every $\bfu \in E^1_0L^p(\Omega)$. Here, $p^*=\frac {np}{n-p}$, the usual Sobolev conjugate of $p$. Yet, the available proofs of inequality \eqref{intro1} for $p\in (1,n)$ and for $p=1$ are substantially different. The former consists in reducing inequality \eqref{intro1} to its counterpart for the space $W^{1,p}_0(\Omega)$. \textcolor{black}{This is possible thanks to Korn's inequality, which implies the equivalence of the  norms of $\ep(\bfu)$ and $\nabla \bfu$ in $L^p(\Omega)$ for $p\in (1,\infty)$.} Alternatively, 
representation formulas   for functions in terms of integral operators applied to their symmetric gradient can be exploited.
Neither of these approaches applies to the case when $p=1$, which requires 
 a direct argument in the spirit of the proofs by Gagliardo and by Nirenberg of the Sobolev inequality for $W^{1,1}_0(\Omega)$. 
\par Our original motivation for this reasearch was an optimal Sobolev inequality, in the spirit of  \eqref{intro1}, for symmetric gradient Orlicz-Sobolev spaces $E^1_0L^A(\Omega)$ built upon an arbitrary Young function $A$. These spaces come into play in certain models in continuum mechanics ruled by non-polynomial type nonlinearities. For instance,   the
nonlinearities appearing in the Prandt-Eyring model for non-Newtonian fluids
\cite{BrDF,E,FuS}, and in models for plastic materials with
logarithmic hardening \cite{FrS}  are described  by a Young function
$A(t)$ that grows like $t\log(1+t)$ near infinity. Young functions with exponential growth are well suited to account for the behavior of fluids in certain liquid body armors
\cite{HRJ, SMB, W}, 
and also play a role in the study of functionals with linear growth in the symmetric gradient \cite{Gm}. 
\textcolor{black}{Further results on non-Newtonian fluids, whose mathematical models are patterned via Young functions, can be found in \cite{AcMi, buli1,buli2}.}
\par
Inequalities with an optimal target for the standard Orlicz-Sobolev space $W^1_0L^A(\Omega)$, defined through the full gradient, are available \cite{Ci1, Cianchi_CPDE, Ci3, cianchi-randolfi IUMJ}. Still,   none of the methods  mentioned above for $E^1_0L^p(\Omega)$ can be adapted to deal with the whole family of spaces $E^1_0L^A(\Omega)$, which, in particular, embraces the entire scale of spaces $E^1_0L^p(\Omega)$ with $p \in [1, \infty]$. In this connection, an obstruction is the failure of a Korn type inequality in $E^1_0L^A(\Omega)$, which affects not only the space $E^1_0L^1(\Omega)$, but also any other space associated with a Young function $A$ that does not simultaneously satisfy both the $\Delta_2$-condition and the $\nabla_2$-condition \cite{BrD, DRS,Fu2}. When these conditions are dropped, sharp versions of the Korn inequality, with possibly slightly different Orlicz norms for $\ep(\bfu)$ and $\nabla \bfu$, are still at our disposal \cite{BCD, Ckorn}. However, in general they do not enable one to deduce optimal Sobolev inequalities for the space $E^1_0L^A(\Omega)$ from those in $W^1_0L^A(\Omega)$. Optimal inequalities do not even follow  via representation formulas, since the integral operators coming into play are only of weak type if the Young function $A$ does not satisfy proper additional assumptions.
\par One main contribution of this paper is a comprehensive treatment of Sobolev inequalities for the spaces $E^1_0X(\Omega)$ and $E^1X(\Omega)$, where $X(\Omega)$ is an arbitrary rearrangement-invariant space, and not just an Orlicz space. {\color{black} Loosely speaking, a rearrangement-invariant space is a Banach space of measurable functions endowed with a norm depending only their   integrability properties.}
 As will be clear, the accomplishment of this project has  required sheding light on diverse aspects and developing alternate methods,  of independent interest, in the theory of these Sobolev type spaces. 
\par
A sharp criterion for the validity of Sobolev inequalities of the form
\begin{equation}\label{intro2}
\|\bfu\|_{Y(\Omega)} \leq C \|\ep(\bfu)\|_{X(\Omega)}
\end{equation}
for every $\bfu \in E^1_0X(\Omega)$ is presented, when $\Omega$ is a domain of finite Lebesgue measure $|\Omega|$, and $X(\Omega)$ and $Y(\Omega)$ are rearrangement-invariant spaces. {\color{black} A remarkable property  of this class of spaces is that, unlike its subclass of Lebesgue spaces,  it is closed under the operation of associating an optimal target $Y(\Omega)$ with a given domain $E^1_0X(\Omega)$  in the Sobolev inequality. Namely, for each space $E^1_0X(\Omega)$, there exists a smallest possible rearrangement-invariant space  $Y(\Omega)$ which renders inequality \eqref{intro2} true. }
As a consequence of our criterion, the optimal target space $Y(\Omega)$ in \eqref{intro2} is characterized for any given domain space $X(\Omega)$. In particular, a threshold condition on $X(\Omega)$ for inequality \eqref{intro2} to hold with $Y(\Omega)=L^\infty(\Omega)$ is determined. Under this condition, functions  in $E^1_0X(\Omega)$ turn out to be also continuous. On the other hand, the existence of an uniform bound for their modulus of continuity, depending only on $E^1_0X(\Omega)$, is shown to be equivalent  to  a slightly stronger condition on $X(\Omega)$.
For spaces $X(\Omega)$  fulfilling this strengthened  condition, we identify   the optimal modulus of continuity $\sigma (\cdot)$ in the inequality
\begin{equation}\label{intro3}
\|\bfu\|_{C^\sigma (\Omega)} \leq C \|\ep(\bfu)\|_{X(\Omega)}
\end{equation}
for every $\bfu \in E^1_0X(\Omega)$. Here, $C^\sigma (\Omega)$ denotes the space of functions in $\Omega$  whose modulus of continuity does not exceed the function $\sigma$.
\par An overall trait of our results on inequalities \eqref{intro2} and \eqref{intro3} is that, whereas the former holds if and only if an inequality with  the same spaces $X(\Omega)$ and $Y(\Omega)$ holds with $\ep(\bfu)$ replaced by $\nabla \bfu$, this is not always the case for the latter. Indeed, there exist rearrangement-invariant spaces $X(\Omega)$ for which the optimal space $C^\sigma (\Omega)$ in a counterpart of inequality \eqref{intro3}, with $\ep(\bfu)$ replaced by $\nabla \bfu$, is strictly smaller. In other words:
\begin{itemize}
\item Symmetric gradient Sobolev embeddings    into \emph{rearrangement-invariant} target spaces are \emph{always equivalent} to the corresponding embeddings for the full gradient  built upon the same domain spaces.
\item
 Symmetric gradient Sobolev embeddings  into target  spaces of \emph{uniformly continuous} functions \emph{may be  weaker}  than the corresponding embeddings for the full gradient built upon the same domain spaces.
\end{itemize}
%
 \par Sobolev  inequalities for functions with unrestricted boundary values, namely inequalities of the form
\begin{equation}\label{intro4}
\|\bfu\|_{Y(\Omega)} \leq C \|\bfu\|_{E^1X(\Omega)}
\end{equation}
for every $\bfu \in E^1X(\Omega)$,  as well as  their complement
\begin{equation}\label{intro5}
\|\bfu\|_{C^\sigma (\Omega)} \leq \|\bfu\|_{E^1X(\Omega)}
\end{equation}
for every $\bfu \in E^1X(\Omega)$, are also discussed. 
Their characterizations, and hence the optimal target spaces, are the same as for $E^1_0X(\Omega)$, but, as in the case of standard Sobolev inequalities, some regularity of the domain $\Omega$ is now required. With this regard, we are able to allow for minimally regular domains, introduced by Jones \cite{Jo}, and called $(\varepsilon, \delta)$-domains in the literature. A fundamental property of these domains, established in \cite{Jo},  is that they admit a bounded extension operator in classical Sobolev spaces.
\par A result of ours, which is critical in view of our Sobolev inequalities in $E^1X(\Omega)$,  tells us that the extension property of $(\varepsilon, \delta)$-domains carries over to symmetric gradient Sobolev spaces. We emphasize that this property is new even for the customary spaces $E^1L^p(\Omega)$. The proof of the existence of a linear bounded extension operator in this class of domains makes use, in its turn, of the density of the space $C^\infty(\overline \Omega)$ in  $E^1L^1(\Omega)$. This is a further novel result that will be proved.
\par The case when $\Omega = \rn$ in inequalities \eqref{intro2} and \eqref{intro4} is also dealt with, and presents distinct features in the two cases. 
Indeed, the criteria and the optimal target spaces  for the inequality 
\begin{equation}\label{intro6}
\|\bfu\|_{Y(\rn)} \leq C \|\ep(\bfu)\|_{X(\rn)}
\end{equation}
for every $\bfu \in E^1_0X(\rn)$, and for 
\begin{equation}\label{intro7}
\|\bfu\|_{Y(\rn)} \leq C \|\bfu\|_{E^1X(\rn)}
\end{equation}
for every $\bfu \in E^1X(\rn)$, take a different form. For instance, if $X(\rn)=L^p(\rn)$ for some $p> n$, then inequality \eqref{intro6} fails whatever $Y(\rn)$ is (even if $\ep(\bfu)$ is replaced by $\nabla \bfu$), whereas inequality \eqref{intro7} classically holds with $Y(\rn)=L^\infty(\rn)$.
\par Our characterizations of the Sobolev inequalities with rearrangement-invariant target norms amount to their equivalence to one-dimensional Hardy type inequalities involving the same kind of norms. The relevant Hardy inequalities agree with those characterizing parallel Sobolev inequalities for the full gradient \cite{Ci1, Ci3, EKP}. For inequalities of the latter type, the reduction to one-dimensional inequalities is classically performed via Schwarz symmetrization, at least in the case of functions vanishing on the boundary of their domain, and relies upon the classical isoperimetric inequality in $\mathbb{R}^{n}$. The equivalence of
 isoperimetric and Sobolev inequalities  in a general framework was discovered some sixty years ago in seminal researches by Maz'ya, that were initiated in the paper  \cite{Ma1960}, and whose full developments can be found in the monograph \cite{Mazya}. Special cases are also contained in the work of Federer-Fleming \cite{FF}.
Schwarz symmetrization has proved to be of crucial use in detecting sharp constants in classical Sobolev type inequalities, starting with the papers by 
 Moser \cite{Moser}, Aubin
\cite{Aubin} and Talenti \cite{Ta}, and has been successfully exploited
in the solution to a number of related optimization  problems.
Remarkably, it is also critical in the analysis  of affine invariant Sobolev inequalities introduced in the frames of the Brunn-Minkowski  convexity theory by Zhang \cite{Z} and Lutwak-Yang-Zhang \cite{LYZ1, LYZ2}.
The approach to   Sobolev type inequalities by symmetrization can be adjusted to deal with functions with unrestricted boundary values on regular bounded domains, via a preliminary extension argument to compactly supported functions in $\rn$. Although information on sharp constants  is lost under the action of an extension operator, that on optimal spaces is preserved.
\par When symmetric gradient Sobolev spaces are in question, symmetrization methods do not apply. We have instead to resort to an approach based on interpolation inequalities, which are derived via optimal endpoint embeddings and $K$-functional inequalities. Such an approach has also been exploited in the proof of reduction principles for Sobolev inequalities for the full gradient in situations where symmetrization fails. This is the case, for instance, when higher-order derivatives \cite{CPS_advances, KermanPick} and/or traces of functions \cite{CP_TAMS, CPS_Frostman} come into play. An adaptation of this method to the symmetric gradient realm is however not straightforward.
\par A first issue in this connection concerns the endpoint embeddings.  Whereas one of them can be   deduced via quite standard techniques based on representation formulas, the other one rests upon 
a Sobolev inequality in $E^1L^1(\rn)$, with optimal Lorentz target norm, which has been an open problem for some time  and has only recently been settled in \cite{SV} (see also \cite{VSh} for earlier partial results in this direction). A second point regards a formula for the $K$-functional  for symmetric gradient Sobolev spaces, namely a parallel of the celebrated result for ordinary Sobolev spaces of \cite{DeSc}. This seems to be missing in the literature, and constitutes another main  achievement  of this paper.
%
In particular, in its proof we recourse to  ad hoc   truncation operators which are reminiscent of those exploited in the context of Lipschitz truncation. The latter has been  successfully employed  in the existence and regularity theory of nonlinear systems of partial differential equations \cite{BuSc,DMS,DRW,DuMi}, as well as in the analysis of semicontinuity problems in the calculus of variations \cite{AcFu1,23}. An overview of the use of this tool can be found in \cite[Chapter 1.3]{Br}.
\par Different techniques are employed in the proof of Sobolev inequalities with target norms in spaces of uniformly continuous functions. This proof relies on the representation formulas mentioned above.  As is well known from the case of ordinary Sobolev spaces, this approach need not yield optimal Sobolev embeddings with rearrangement-invariant target spaces. Interestingly, our results demonstrate that this limitation does not occur when target spaces of uniformly continuous functions enter the game. 
\par The results for general rearrangement-invariant spaces outlined above are implemented for the family of Orlicz-Sobolev spaces $E^1_0L^A(\Omega)$ and $E^1L^A(\Omega)$. The optimal target spaces in the resultant Sobolev inequalities take an even more explicit form in this ambient. {\color{black} Importantly, an optimal (smallest possible) Orlicz target space exists for any Young function $A$. This means that also the class of Orlicz spaces is rich enough to be  closed under   the operation  of associating an optimal target  in symmetric gradient Sobolev embeddings.}
\par The paper is organized as follows. Section \ref{sec2} contains  definitions and background material on rearrangement-invariant spaces and their associated Sobolev type spaces. Suitable representation formulas  and ensuing Sobolev-Poincar\'e type inequalities of special form and on specific domains are provided in Section \ref{poincareE1}.
Truncation operators in symmetric gradient Sobolev spaces are introduced in Section \ref{truncation}, where their properties are also established. Section \ref{smoothapprox} is devoted to the approximation of functions in symmetric gradient Sobolev spaces on $(\varepsilon, \delta)$-domains $\Omega$ by smooth functions on $\overline \Omega$. An extension operator for functions in these spaces, for the same class of domains, is constructed in Section \ref{extension}. $K$-functionals for diverse couples of symmetric gradient Sobolev spaces are computed in Section \ref{Kfunct}. The analysis of Sobolev inequalities begins with Section \ref{sobemb}, where rearrangement-invariant target spaces are considered. Sobolev inequalities into target spaces of uniformly continuous functions are the subject of Section \ref{cont}. In the final Section \ref{secorlicz},  inequalities for symmetric gradient Orlicz-Sobolev spaces are derived via the general results of the preceding sections.

\section{Function spaces}\label{sec2}
  In this section we collect the necessary background on function spaces of measurable functions and of weakly differentiable functions. Some notations and definitions about symmetric gradient Sobolev type spaces are also introduced.
\par
Throughout the paper, the relation $\lq\lq \lesssim "$ between two positive expressions means that the former is bounded by the latter, up to a multiplicative constant depending on quantities to be specified.
The relations  $\lq\lq \gtrsim "$ and $\lq\lq \approx "$ are defined accordingly. 

\subsection{Rearrangement invariant spaces}\label{ri}
We denote by $|E|$ the Lebsegue measure of a set $E\subset \rn$. Assume that $\Omega $ is a measurable subset of $\rn$, with $n \in \mathbb N$, and let $m \in \mathbb N$.  The notation 
$\M(\Omega)$ is adopted for
 the space of all Lebesgue-measurable functions
$\bfu : \Omega \to \mathbb R^m$. If $m>1$, functions in  $\M(\Omega)$ will usually denoted in bold fonts, whereas normal fonts will be used if $m=1$. We denote by $u_i$, $i=1, \dots, m$, the $i$-th component of a function $\bfu : \Omega \to \mathbb R^m$; namely $\bfu =(u_1, \dots , u_m)$.  Similar notations are employed for matrix-valued functions, that will usually be denoted in bold fonts and by upper case letters. 
Also, if $m=1$ we  define $\Mpl(\Omega)=\{u\in\M(\Omega)\colon u\geq 0 \
\textup{a.e. in}\ \Omega\}$.
\\
The \emph{decreasing rearrangement}  $\bfu^{\ast}:[0, \infty )\to
[0,\infty]$ of a function $\bfu \in \M(\Omega)$ is   defined as
\begin{equation*}
\bfu^{\ast}(s) = \inf \{t\geq 0: |\{x\in \Omega : |\bfu(x)|>t \}|\leq s \}
\qquad \hbox{for $s \in [0,\infty)$}.
\end{equation*}
The function  $\bfu^{**}: (0, \infty ) \to [0, \infty )$ is given by
\begin{equation}\label{c15}
\bfu^{**}(s)=\frac{1}{s}\int_0^s \bfu^*(r)\dr \qquad \hbox{for $s>0$}.
\end{equation}
One has that
\begin{equation}\label{c15bis}
\bfu^{**}(s) = \sup_{|E|=s} \frac 1{s} \int _E |\bfu(x)|\dx \qquad
\hbox{for $s
>0$,}
\end{equation}
where the supremum is taken over all measurable sets $E\subset\Omega$.
Moreover,
\begin{equation}\label{subadd}
(|\bfu|+|\bfv|)^{**}(s) \leq \bfu^{**}(s) + \bfv^{**}(s) \qquad \hbox{for $s>0$,}
\end{equation}
 for every  $\bfu, \bfv \in  \M(\Omega)$.
A basic  property of the operation of decreasing rearrangement is
the \emph{Hardy-Littlewood inequality}, which states that
\begin{equation}\label{B.0}
\int_{\Omega }|\bfu(x) \cdot \bfv(x)| \dx \leq \int_{0}^{\infty
}\bfu^{\ast}(s)\bfv^{\ast}(s)\ds
\end{equation}
for every  $\bfu, \bfv \in  \M(\Omega)$. Here, the dot $\lq\lq \cdot \lq\lq$ stands for scalar product in $\rn$.
\\ Another property ensures that,    given functions $\bfu , \bfv \in \M(\Omega)$, 
\begin{equation}\label{ALT}
 \bfu^{**}(t )\leq \bfv^{**}(t)\,\,\text{for $t>0$}  \quad  \text{if and only  if}\quad \int _\Omega A(|\bfu|) \dx \leq  \int _\Omega A(|\bfv|) \dx \,\, \text{for every Young function $A$,}
\end{equation}
see e.g. \cite[Proof of Proposition 2.1]{ALT}. Recall that $A: [0, \infty ) \to [0, \infty ]$ is called a Young
function if it is convex (non trivial), left-continuous and
vanishes at $0$. 
\\
Let $L\in (0,\infty]$. A functional
 $\|\cdot\|_{X(0,L)}{:\Mpl(0,L)\to[0,\infty]}$ is called a
\textit{function norm} if, for all functions $f, g\in \Mpl(0,L)$, all sequences
$\{f_k\} \subset {\Mpl(0,L)}$, and every $\lambda {\in[0,\infty)}$:
\begin{itemize}
\item[(P1)]\quad $\|f\|_{X(0,L)}=0$ if and only if $f=0$ a.e.;
$\|\lambda f\|_{X(0,L)}= \lambda \|f\|_{X(0,L)}$; \par\noindent \quad
$\|f+g\|_{X(0,L)}\leq \|f\|_{X(0,L)}+ \|g\|_{X(0,L)}$;
\item[(P2)]\quad $ f \le g$ a.e.\  implies $\|f\|_{X(0,L)}
\le \|g\|_{X(0,L)}$;
\item[(P3)]\quad $f_k \nearrow f$ a.e.\
implies $\|f_k\|_{X(0,L)} \nearrow \|f\|_{X(0,L)}$;
\item[(P4)]\quad $\|\chi _E\|_{X(0,L)}<\infty$ if $|E| < \infty$;
\item[(P5)]\quad if $|E|< \infty$, there exists a constant
 $C$ depending on $E$  such that $\int_E f(t)\,dt \le C
\|f\|_{X(0,L)}$.
\end{itemize}
Here, $E$ denotes a measurable set in $(0,L)$, and  $\chi_E$ stands for its characteristic function.
If, in addition,
\begin{itemize}
\item[(P6)]\quad $\|f\|_{X(0,L)} = \|g\|_{X(0,L)}$ whenever $f\sp* = g\sp *$,
\end{itemize}
we say that $\|\cdot\|_{X(0,L)}$ is a
\textit{rearrangement-invariant function norm}.
\\
The \textit{associate function norm}  $\|\cdot\|_{X'(0,L)}$ of a function norm $\|\cdot\|_{X(0,L)}$ is  defined as
$$
\|f\|_{X'(0,L)}=\sup_{\begin{tiny}
                        \begin{array}{c}
                       {g\in{\Mpl(0,L)}}\\
                        \|g\|_{X(0,L)}\leq 1
                        \end{array}
                      \end{tiny}}
\int_0\sp{L}f(s)g(s)\ds
$$
for $ f\in\Mpl(0,L)$.
Note that
\begin{equation}\label{X''}
\|\cdot \|_{(X')'(0,L)} = \|\cdot \|_{X(0,L)}.
\end{equation}
\par\noindent
A property of function norms tells us that if $f , g \in \Mpl(0,L)$, then
\begin{equation}\label{hardy} f^{**}(s) \leq g^{**}(s) \quad \hbox{for $s>0$ implies that}
\quad \|f\|_{X(0,L)} \leq \|g\|_{X(0,L)}.
\end{equation}
Let $\Omega$ be a measurable set in $\rn$, and let
 $\|\cdot\|_{X(0,|\Omega|)}$ be   a rearrangement-invariant function norm.  Then the space $X(\Omega)$ is
defined as the collection of all  functions  $\bfu \in\M(\Omega)$
such that the quantity
\begin{equation}\label{norm}
\|\bfu\|_{X(\Omega)}= \|\bfu^*\|_{X(0,|\Omega|)} 
%
\end{equation}
is finite. The space $X(\Omega)$ is a Banach space, endowed
with the norm given by \eqref{norm}. 
 The space $X(0,|\Omega|)$ is called the
\textit{representation space} of $X(\Omega)$.
\\
The \textit{associate
space}  $X'(\Omega)$
of   a~rearrangement-invariant~space $X(\Omega)$ is
 the
rearrangement-invariant space   built upon the
function norm $\|\cdot\|_{X'(0,|\Omega|)}$. By property \eqref{X''},
$X''(\Omega)=X(\Omega)$.  
\\
The \textit{H\"older type inequality}
\begin{equation}\label{holder}
\int_{\Omega}|\bfu(x)\cdot \bfv(x)|\dx\leq\|\bfu\|_{X(\Omega)}\|\bfv\|_{X'(\Omega)}
\end{equation}
holds  for every $\bfu, \bfv \in \M(\Omega)$.
\\
If  $|\Omega|< \infty$, then
\begin{equation}\label{l1linf}
L^\infty (\Omega) \to X(\Omega) \to L^1(\Omega)
\end{equation}
for every rearrangement-invariant space
$X(\Omega)$.
\\ The fundamental function $\varphi_X : [0, \infty) \to [0, \infty)$ of a rearrangement-invariant space $X(\Omega)$ is defined as
\begin{equation}\label{fund}
\varphi_X(s) = \|\chi_E\|_{X(\Omega)} \quad \text{for $s \geq 0$,}
\end{equation}
where $E$ is any measurable subset of $\Omega$ such that $|E|=s$. The function $\varphi_X(s)/s$ is non-increasing. Furthermore, 
\begin{equation}\label{fund1}
\varphi_X(s)  \varphi_{X'}(s) =s \quad \text{for $s \geq 0$,}
\end{equation}
for every rearrangement-invariant space $X(\Omega)$.
 \\We shall need to extend
rearrangement-invariant spaces defined on a set $\Omega$ to rearrangement-invariant spaces
defined on the whole of $\rn$, and, conversely, to restrict rearrangement-invariant
spaces defined on the whole of $\rn$ to rearrangement-invariant spaces
defined on a subset $\Omega$, in some canonical way.  
\\  If $\bfu : \Omega \to \R^m$, we call $\bfu_e: \R^n \to \R^m$ its extension to the whole of $\R^n$ defined by $0$ outside $\Omega$. Moreover, if $\bfu : \R^n \to \R^m$, we call $\bfu_r: \Omega \to \R^m$ its restriction to $\Omega$.
\\ Given $L\in (0, \infty)$ and a function norm $\|\cdot\|_{X(0, \infty)}$, we define the  function norm 
 $\|\cdot\|_{X_r(0, L)}$ as 
\begin{equation}\label{feb80}
\|f\|_{X_r(0, L)} = \|f^*\|_{X(0, \infty)}
\end{equation}
for $f \in \mathcal M_+(0,L)$. An analogous definition holds if $L<L'$ and a function norm  $\|\cdot\|_{X(0, L')}$ is given.  If $\Omega$ is a measurable set in $\R^n$, we denote by $X_r(\Omega)$ the rearrangement-invariant space built upon the function norm $\|\cdot\|_{X_r(0, |\Omega|)}$. Notice that
\begin{equation}\label{feb81}
\|\bfu\|_{X_r(\Omega)}=  \|\bfu^*\|_{X(0, \infty)}= \|\bfu _e\|_{X(\R^n)}
\end{equation}
for every $\bfu \in \mathcal M(\Omega)$.
\\ On the other hand, given $L\in (0, \infty)$ and a function norm $\|\cdot\|_{X(0, L)}$, we define the  function norm 
 $\|\cdot\|_{X_e(0, \infty)}$ as 
\begin{equation}\label{feb82}
\|f\|_{X_e(0, \infty)} = \|f^*\|_{X(0, L)}
\end{equation}
for $f \in \mathcal M_+(0,\infty)$. We denote by $X_e(\R^n)$ the rearrangement-invariant space associated with the function norm $\|\cdot\|_{X_e(0, \infty)}$. One has that 
\begin{equation}\label{feb83}
\|\bfu\|_{X_e(\R^n)}=  \|\bfu^*\|_{X(0, |\Omega|)}
\end{equation}
for every $\bfu \in \mathcal M(\R^n)$. In particular, if $\bfu =0$ a.e. in $\rn \setminus \Omega$, then 
\begin{equation}\label{extbis}
\|\bfu\|_{X_e(\R^n)}=  \|\bfu_r\|_{X(\Omega)}.
\end{equation}
 More generally, a function norm  $\|\cdot\|_{\widehat X(0, \infty)}$ will be called an extension of  the function norm $\|\cdot\|_{X(0, L)}$ if 
\begin{equation}\label{mar170}
\|f\|_{\widehat X(0, \infty)}= \|f_r\|_{X(0, L)}
\end{equation}
for every $f\in \Mpl (0,\infty)$ such that $f=0$ a.e. in $(L, \infty)$. Clearly,  
\begin{equation}\label{mar171}
\|\bfu\|_{\widehat X(\rn)}= \|\bfu _r\|_{X(\Omega)}
\end{equation}
for every $\bfu\in \mathcal M(\rn)$ such that $\bfu = 0$ a.e. in $\rn \setminus \Omega$.
\subsection{Orlicz, Lorentz, Lorentz-Zygmund and Orlicz-Lorentz spaces}  \label{orlicz&lorentz}

 Assume that $0<p,q\le\infty$. Let $L\in (0, \infty]$. We define the functional
$\|\cdot\|_{L\sp{p,q}(0,L)}$  by
$$
\|f\|_{L\sp{p,q}(0,L)}=
\left\|s\sp{\frac{1}{p}-\frac{1}{q}}f^*(s)\right\|_{L\sp q(0,L)}
$$
for  $f \in {\Mpl(0,L)}$. Here, and in what follows,  we use the convention that $\frac1{\infty}=0$.
If either $1<p<\infty$
and $1\leq q\leq\infty$, or $p=q=1$, or $p=q=\infty$,
 then $\|\cdot\|_{L\sp{p,q}(0,L)}$ is equivalent to a~rearrangement-invariant function norm.
The norm
$\|\cdot\|_{L\sp{p,q}(0,L)}$
is called  \textit{Lorentz function norm}, and the corresponding space
$L\sp{p,q}(\Omega)$ 
is called 
\textit{Lorentz space}.
\\
Suppose now that $0<p,q\le\infty$ and $\alpha \in\R$. Let $L\in (0, \infty)$. We define the
functional $\|\cdot\|_{L\sp{p,q;\alpha}(0,L)}$ by
\begin{equation}\label{E:1.18}
\|f\|_{L\sp{p,q;\alpha}(0,L)}=
\left\|s\sp{\frac{1}{p}-\frac{1}{q}}\log \sp
\alpha(\tfrac{eL}{s}) f^*(s)\right\|_{L\sp q(0,L)}
\end{equation}
for  $f \in {\Mpl(0,L)}$. For suitable choices of the parameters  $p,q, \alpha$, the functional
$\|\cdot\|_{L\sp{p,q;\alpha}(0,L)}$ is equivalent to a~rearrangement-invariant function
norm. If this is the  case, $\|\cdot\|_{L\sp{p,q;\alpha}(0,L)}$  is
called  \textit{Lorentz--Zygmund function norm}, and the
corresponding space $L\sp{p,q;\alpha}(\Omega)$ is called
\textit{Lorentz--Zygmund space}.  A four-parameter space $L^{\infty,p; -\frac 1p, -1}(\Omega)$, built upon the function norm $\|\cdot\|_{L^{\infty,p; -\frac 1p, -1}(0,L)}$ will also play a role in our applications. The latter is defined, for $p\in (1, \infty)$, as
\begin{equation}\label{GLZ}
\|f\|_{L^{\infty,p;-\frac{1}p,-1}(0,L)}  =
\left\|s^{-\frac{1}{p}}\log^{-\frac{1}{p}}\left(\tfrac{eL}{s}\right) (\log (1+\log \left(\tfrac{eL}{s}\right)))^{-1}f^{*}(s) \right\|_{L^{p}(0, L)}
\end{equation}
for  $f \in {\Mpl(0,L)}$.
A detailed study of
Lorentz-Zygmund spaces can be found in \cite{EOP} and \cite[Chapter~9]{PKJF}.
\par
The notion of Orlicz space relies upon that of Young function.
Any such function takes the form
\begin{equation}\label{young}
A(t) = \int _0^t a(\tau ) \,\mathrm{d}\tau \qquad \hbox{for $t \geq
0$},
\end{equation}
for some non-decreasing, left-continuous function $a: [0, \infty )
\to [0, \infty ]$ which is neither identically equal to $0$ nor to
$\infty$. 
 \\ Given $L\in (0, \infty]$, the
\textit{Luxemburg function norm} $\|\cdot \|_{L^A (0,L)}$ is defined
by
\begin{equation}\label{lux}
\|f\|_{L^A(0,L)}= \inf \left\{ \lambda >0 :  \int_{0}^LA \left(
\frac{f(s)}{\lambda} \right) \dt \leq 1 \right\} 
\end{equation}
for  $f \in {\Mpl(0,L)}$.
The corresponding
 rearrangement-invariant  space
$L\sp A(\Omega)$ is called \textit{Orlicz space}. 
In particular,
$L^A (0,L)= L^p (0,L)$ if $A(t)= t^p$ for some $p \in [1, \infty )$,
and $L^A (0,L)= L^\infty (0,L)$ if $A(t)=0$ for $t\in [0, 1]$ and
$A(t) = \infty$ for $t>1$.
\\
Let  $A$ and $B$ be Young functions. If $L\in (0, \infty)$, the function norms $\|\cdot \|_{L^A
(0,L)}$ and $\|\cdot \|_{L^B (0,L)}$ are equivalent if and only if
$A$ and $B$ are equivalent near infinity, in the sense that there
exist constants $c\in[1,\infty)$ and $t_0\in[0,\infty)$ such that
\begin{equation}\label{equivorlicz}
A(t/c) \leq B(t) \leq A(ct) \quad \text{for $t\in[t_0,\infty)$.}
\end{equation}
The function norms $\|\cdot \|_{L^A
(0,\infty)}$ and $\|\cdot \|_{L^B (0,\infty)}$ are equivalent if and only if
$A$ and $B$ are globally equivalent, in the sense that  inequality 
\eqref{equivorlicz} holds for some constant $c>1$ and every $t\geq 0$.
\\ When $|\Omega|<\infty$, the alternate notation $A(L)(\Omega)$  will also be adopted, when convenient, to denote the Orlicz space built upon a Young function (equivalent near infinity to) $A$.
\\
The associate function norm  $\|\cdot \|_{(L^A)^{\prime}
(0,L)}$ of the function norm $\|\cdot \|_{L^A
(0,L)}$  satisfies the inequalities
\begin{equation}\label{assorlicz}
 \|f \|_{L^{\tilde A}(0,L)} \leq \|f\|_{(L^A)^{\prime}(0,L)}\leq 2  \|f \|_{L^{\tilde A}(0,L)}
\end{equation}
for $f \in {\Mpl(0,L)}$, 
where
%
$\widetilde{A}$ is the
\emph{Young conjugate} of $A$  defined by
$$\widetilde{A}(t) = \sup \{\tau t-A(\tau ):\,\tau \geq 0\} \qquad \text{for $t\geq 0.$}$$
The following observation will be of use in our applications. Assume that $f,g \in {\Mpl(0,L)}$. Then
\begin{align}\label{new20}
&\text{$\|f\|_{L^A(0,L)}\leq  \|g\|_{L^A(0,L))}$ for every Young function $A$}\\ \nonumber & \text{if and only if  $\int_0^LA(f(s))\ds \leq 
\int_0^LA(g(s))\ds$ for every Young function $A$.}
\end{align}
The latter property follows from the former, applied with $A(t)$ replaced by the Young function $A(t)/M$, where $M=\int_0^LA(g(s))\ds$. The former follows from the latter, applied with $A(t)$ replaced by the Young function $A(t/\lambda)$, where $\lambda =  \|g\|_{L^A(0,L))}$.
\\ A comprehensive treatment of Orlicz spaces can be found in \cite{RR1, RR2}.
\\
An extension of the  class of Orlicz spaces, that also includes 
 various instances of Lorentz and Lorentz-Zygmund spaces,
 is provided by the
family of \emph{Orlicz-Lorentz spaces}.
Let $L\in (0,\infty]$. Given  a Young function $A$ and a number $q\in \R$,  we denote
 by $\|\cdot \|_{L(A,q)(0,L)}$ the \emph{Orlicz-Lorentz
functional} defined as
\begin{equation}\label{016}
\|f\|_{L(A,q)(0,L)} = \left\|s^{-\frac 1q} f^* (s)\right\|_{L^A(0, L)}  
\end{equation}
for $f\in \Mpl(0,L)$.
 Under suitable assumptions on $A$ and $q$, this quantity is (equivalent to) a rearrangement-invariant function norm.   This is certainly the case when $q>1$ and
\begin{equation}\label{aug310}
\int^\infty \frac{A(t)}{t^{1+q}}\dt < \infty\,,
\end{equation}
see \cite[Proposition 2.1]{Ci3}.
\\ The Lorentz  functional $\|\cdot\|_{L^{p,q}(0,L)}$, the Lorentz-Zygmund functional $\|\cdot\|_{L^{p,q;\alpha}(0,L)}$ and the Orlicz-Lorentz functional  $\|\cdot\|_{L(A,q)(0,L)}$ are in their turn special instances of the functionals $\|\cdot\|_{\Lambda ^A(\nu)}$, associated with a Young function $A$ and a locally integrable function $\nu: (0, L) \to [0, \infty)$, and defined  as 
\begin{equation}\label{feb85}
\|f\|_{\Lambda ^A(\nu)}= \|\nu (s) f^*(s)\|_{L^A(0, L)}
\end{equation}
for $f \in \mathcal \Mpl (0,L)$. A complete characterization of those Young
functions $A$ and weights $\nu$ for which the functional
$\|\cdot\|_{\Lambda\sp{A}(\nu)(0,L)}$ is (equivalent to) a
rearrangement-invariant function norm seems not to be available in
the literature. However, this will always be the case in all our applications.
%
%
For such a choice of weights $\nu$ and Young functions $A$, we shall denote by
$\Lambda\sp{A}(\nu)(\Omega)$ the corresponding
rearrangement-invariant space on a measurable set
$\Omega\subset \rn$.

\subsection{Sobolev-type spaces}\label{sobolevtype}
Assume that $\Omega\subset\R^n$ is an open set and $X(\Omega)$ is a rearrangement-invariant space.  
Given a   weakly differentiable function $\bfu :\Omega \to \R^m$, we denote by  $\nabla \bfu : \Omega \to \R^{m\times n}$ its gradient, and set 
\begin{equation}\label{Du}
\mathcal D \bfu = (\bfu, \nabla \bfu) \in \rn\times \setR^{m\times n}.
\end{equation}
The  \emph{Sobolev space}
$W^{1}X(\Omega)$ is defined as 
$$W^{1}X(\Omega) = \{\bfu \in X(\Omega) : \nabla \bfu \in X(\Omega)\}.$$
%
It is a Banach space endowed with the norm
\begin{align*}
\|\bfu\|_{W^{1}X(\Omega)}=\|\bfu\|_{X(\Omega)}+\|\nabla
\bfu\|_{X(\Omega)}.
\end{align*}
In the special case when $X(\Omega)=L^p(\Omega)$, with $1\leq p\leq \infty$, or, more generally, $X(\Omega)=L^A(\Omega)$, where $A$ is a Young function, this definition recovers the standard Sobolev space $W^{1,p}(\Omega)$ and the Orlicz-Sobolev space $W^{1}L^A(\Omega)$,  respectively.
We also define the subspace
\begin{equation}\label{apr100}
W^1_bX(\Omega) = \{\bfu \in
\mathcal M(\Omega) :\,\, \{|\bfu|>0\}\, \text{is bounded and}\, \, \bfu_e \in E^{1}X(\rn) \},
\end{equation}
equipped with the norm
\begin{equation}\label{feb44}
\|\nabla \bfu\|_{X(\Omega)}.
\end{equation}
Moreover, we set
\begin{align}\label{apr101}
W^1_0X(\Omega) = \text{closure of $W^1_bX(\Omega)$ with respect to the norm \eqref{feb44}}.
%
\end{align}
The space $W^1_0X(\Omega)$ is a subspace of $W^1X(\Omega)$ of those functions which vanish on $\partial \Omega$ and near infinity in the most general   suitable sense in the context of our paper.
\\
Plainly, if $\Omega$ is bounded, then
$$W^1_bX(\Omega) = W^1_0X(\Omega).$$
%
%
%
 Our main concern is about function spaces involving the weak symmetric gradient $\ep (\bfu)$ of a function 
 $\bfu : \Omega \to \rn$ defined by \eqref{symmgrad}.
%
%
 We shall also employ the notation
\begin{equation}\label{new23} 
\mathcal E(\bfu) = (\bfu,  \ep (\bfu))\in \R^{n}\times \R^{n\times n}.
\end{equation}
The kernel of the operator $\ep$ is denoted by $\mathcal R$, and will be referred to as the \emph{space of rigid displacements} on $\rn$. Recall that
\begin{equation}\label{ker1}
\mathcal R = \{\bfb + \bfQ x: \bfb \in \rn, \bfQ\in {\mathbb R}^{n \times n}_{\rm skew}\},
\end{equation}
where ${\mathbb R}^{n \times n}_{\rm skew}$ stands for the space of skew symmetric $n\times n$ matrices.
\\
Notice that $\mathcal R$ is a finite-dimensional space.  As a consequence, if $\Omega$ is either a cube or a ball, then 
\begin{equation}\label{norm-equiv3}
  \quad \|\bfv\|_{L^\infty(\Omega)} \approx \frac 1{|\Omega_1|}\|\bfv\|_{L^1(\Omega)},
\end{equation} 
for every $\bfv \in \mathcal R$, with equivalence constants depending on $n$. 
\\ Moreover,  if $\Omega_1$ and $\Omega_2$ are either cubes or balls,  whose diameters are bounded by each other up to multiplicative constants $c_1$ and $c_2$ and whose distance is bounded by their diameters up to a multiplicative constant $c_3$, then
\begin{equation}\label{norm-equiv4}
\|\bfv\|_{X(\Omega_1)} \approx \|\bfv\|_{X(\Omega_2)},
\end{equation}
 where $X(\Omega_1)$ and $X(\Omega_2)$ are  rearrangement-invariant spaces built upon the same function norm (suitably restricted, if $|\Omega_1|\neq |\Omega_2|$, as described in Subsection \ref{ri}). After verifying equation \eqref{norm-equiv4} when $X=L^1$, the general case follows, on making  also use of  equations \eqref{norm-equiv3}, \eqref{norm-equiv4}, \eqref{holder} and \eqref{fund1}, via the chain
\begin{align}\label{new40}
\|\bfv\|_{X(\Omega_1)} &\leq \|\bfv\|_{L^\infty(\Omega_1)} \|1\|_{X(\Omega_1)} \lesssim \frac {\|1\|_{X(\Omega_1)} }{|\Omega_1|}\int_{\Omega_1}|\bfv|\dx 	\\ \nonumber &	 \lesssim \frac {\|1\|_{X(\Omega_1)}}{|\Omega_1|}\int_{\Omega_2}|\bfv|\dx  
\lesssim \frac {\|1\|_{X(\Omega_1)} \|1\|_{X'(\Omega_2)} }{|\Omega_1|} \|\bfv\|_{X(\Omega_2)}\lesssim  \|\bfv\|_{X(\Omega_2)},
\end{align}
up to mutiplicative constants depending on $c_1$, $c_2$, $c_3$ and $n$, and via an analogous chain  with the roles of $\Omega_1$ and $\Omega_2$ exchanged.
\\
Given a rearrangement-invariant space $X(\Omega)$ we define the \emph{symmetric gradient Sobolev space} $E^1X(\Omega)$ as
\begin{align*}
E^1X(\Omega) =\set{\bfu \in X(\Omega):  \ep(\bfu) \in X(\Omega)}.
\end{align*}
One has that  $E^1X(\Omega)$ is a Banach space equipped with the norm
\begin{equation}\label{normE1X}
\|\bfu\|_{E^1X(\Omega)} = \|\bfu\|_{X(\Omega)} + \|\ep(\bfu)\|_{X(\Omega)}.
\end{equation}
By $E^1_{\rm loc}X(\Omega)$ we denote the space of all functions $\bfu : \Omega \to \rn$ such that $\bfu \in E^1X(G)$ for every bounded open set $G$ such that $\overline  G \subset \Omega$.
%
\\ 
In analogy with \eqref{apr100} and \eqref{apr101},
we   denote by $E^1_bX(\Omega)$ the space 
\begin{equation}\label{apr102}
E^1_bX(\Omega) = \{\bfu \in
\mathcal M(\Omega) :\,\, \{|\bfu|>0\}\, \text{is bounded and}\, \, \bfu_e \in E^{1}X(\rn) \},
\end{equation}
equipped with the norm
\begin{equation}\label{feb44bis}
\|\ep(\bfu)\|_{X(\Omega)},
\end{equation}
and
define
\begin{align}\label{apr103}
E^1_0X(\Omega) =  \text{closure of $E^1_bX(\Omega)$ with respect to the norm \eqref{feb44bis}}.
%
\end{align}
{\color{black} Owing to \cite[Propositions 1.1 and 1.3, Chapter II, Section 1]{Temam}, 
\begin{equation}\label{feb87}
\text{if $\bfu\in E^1_0X(\Omega)$, then $\bfu\in L^1_{\rm loc}(\Omega)$
and $\ep(\bfu)\in X(\Omega)$.}
\end{equation}}
If $\Omega$ is bounded, then
$$E^1_bX(\Omega) = E^1_0X(\Omega).$$
It is well known that the space $C^\infty_0(\rn)$ is dense in $E^1L^p(\rn)$ for every $p \in [1,\infty)$. A parallel density result, with analogous proof, holds for the space  $E^1L^{p,q}(\rn)$ for every $p,q \in [1,\infty)$. Similarly,  $C^\infty_0(\rn)$ is dense in $E^1_bL^p(\rn)$, and, more generally, in $E^1_bL^{p,q}(\rn)$, for every $p,q \in [1,\infty)$.

\section{Representation formulas and basic Poincar\'e type inequalities}\label{poincareE1}

Here we exhibit a few representation formulas for functions in open sets $\Omega \subset \rn$, with $n \geq 2$, in terms of their symmetric gradient. Ensuing Poincar\'e type inequalities in $E^1L^1(\Omega)$ are also presented. The emphasis is on the dependence of the constants for sets $\Omega$ of special form.
\\ We begin by considering functions defined on the whole of $\rn$.

 \begin{lemma}\label{reprrn} Let $\bfu \in E^1_bL^1(\rn)$. Then 
\begin{align}\label{reprrn1}
\bfu(x)= \frac{1}{n\omega_n}\int_{\R^n}\mathcal A\ep(\bfu)(y)\frac{x-y}{|x-y|^{n}}\dy \quad \text{for a.e. $x\in \rn$,}
\end{align}
where $\omega_n$ denotes the Lebesgue measure of the unit ball in $\rn$,  $\mathcal A:\R^{n\times n}\rightarrow \R^{n\times n}$ denotes the linear map defined as
$\mathcal A \bfQ=2\bfQ-\Big(\frac{1}{2}+\frac{1}{n}\Big)\mathrm{tr}\bfQ\,\bfI$ for $\bfQ \in  \R^{n\times n}$, $\mathrm{tr}\bfQ$ stands for the trace of $\bfQ$ and $\bfI \in \R^{n\times n}$  for the identity matrix.
\end{lemma}
\begin{proof}
Assume that $\bfu \in C^\infty_0(\rn)$. 
 A direct computation shows that
$\Delta\bfu =\Div \big(\mathcal A\ep(\bfu)\big)$ --
see e.g. \cite{Ne}.
A classical representation formula via either the Newtonian or the logarithmic potential, depending on whether $n=2$ or  $n\geq 3$, tells us that
\begin{align}\label{mar20}
\bfu(x)=\begin{cases}
\displaystyle -\frac{1}{2\pi}\int_{\R^2}\log \frac 1{|x-y|} \Div \big(\mathcal A\ep(\bfu)\big)(y)\dy & \text{if $n=2$,}
\\ \\
\displaystyle  -\frac{1}{n(n-2)\omega_n}\int_{\R^n}\frac{1}{|x-y|^{n-2}}\Div \big(\mathcal A\ep(\bfu)\big)(y)\dy & \text{if $n \geq 3$,}
\end{cases}
\end{align}
for $x \in \rn$. Equation \eqref{reprrn1} hence follows, via an integration by parts. Suppose now that $\bfu \in E^1_bL^1(\rn)$. Then $\bfu$ can be approximated in $E^1_bL^1(\rn)$, and a.e.,  by a sequence of functions  $\{\bfu_k\} \subset C^\infty_0(\rn)$ with uniformly bounded supports contained in some ball $B\subset \rn$. Since  
there exists a constant $c=c(n)$ such
\begin{equation}\label{mar21}
\bigg|\mathcal A\bfQ\frac{x-y}{|x-y|^{n}}\bigg| \leq \  \frac{c|\bfQ|}{ |x-y|^{n-1}} \quad \hbox{for $x \neq y$,}
\end{equation}
the integral operator on the right-hand side of equation \eqref{reprrn1} is bounded from $L^1(B)$ into $L^1(B)$.  Thus, making use of equation \eqref{reprrn1} for $\bfu_k$ and passing to the limit as $k \to \infty$ shows that the equation  also  holds for $\bfu$. 
%
%
%
%
%
\end{proof}

Recall that the Riesz potential operator $I_1$, of order $1$, of a function $\bfu \in L^1(\rn)$ with bounded support is defined as 
$$I_1\bfu (x) = \int _{\rn} \frac{\bfu(y)}{|x-y|^{n-1}}\dy \quad \hbox {for $x\in \rn$.}$$
As a consequence of O'Neil's rearrangement inequality for convolutions \cite[Lemma 1.8.8]{Ziemer}, there exists a constant $C=C(n)$ such that
\begin{equation}\label{potrearr}
(I_1\bfu)^*(s) \leq C\bigg(s^{-\frac 1{n'}}\int _0^s\bfu^*(r)\dr + \int_s^\infty \bfu^*(r) r^{-\frac 1{n'}}\dr\bigg) \quad \hbox{for $s>0$.}
\end{equation}
Here, $n'=\tfrac n{n-1}$ is the H\"older conjugate of $n$.
From equations \eqref{reprrn1}, \eqref{mar21} and \eqref{potrearr}, one can deduce the following pointwise rearrangement inequality for $\bfu$ in terms of $\ep(\bfu)$.

\begin{corollary}\label{rearrepu} There exists a constant $C=C(n)$ such that
\begin{equation}\label{rearrepu1}
\bfu^*(s) \leq C\bigg(s^{-\frac 1{n'}}\int _0^s\ep(\bfu)^*(r)\dr + \int_s^\infty \ep(\bfu)^*(r) r^{-\frac 1{n'}}\dr\bigg) \quad \hbox{for $s>0$,}
\end{equation}
for every $\bfu \in E^1_bL^1(\rn)$.
\end{corollary}

The next lemma provides us with a  a representation formula for  functions on  balls and cubes
in terms of a projection operator on   $\mathcal R$  and of a Riesz type operator of their symmetric gradient. The punctum in the lemma is the dependence of the norms of the relevant operators.
\\ In what follows, the notation $T: Z \to W$ is employed for   a linear bounded operator $T$ between the normed spaces $Z$ and $W$.
%
%

\begin{lemma}\label{lem:repr}
Let $\Omega\subset\R^n$  be either a cube or ball  centered at $x_0$.
Then there exists a linear bounded operator
$$\mathcal R_\Omega : L^1(\Omega) \to L^\infty(\Omega)$$
of the form 
\begin{equation}\label{mar95}\mathcal R_\Omega (\bfu)(x)=\bfb_\Omega (\bfu) +\bfR_\Omega (\bfu)(x-x_0) \quad \hbox{for $x \in \Omega$,}
\end{equation}
for $\bfu \in L^1(\Omega)$,
where $\bfb_\Omega$ and $\bfR_\Omega$ are linear bounded operators
\begin{equation}\label{mar92} \bfb_\Omega : L^1(\Omega) \to \rn \quad \text{and} \quad \bfR_\Omega : L^1(\Omega) \to \mathbb R^{n\times n}_{\rm skew},
\end{equation}
and there exists a linear bounded operator 
\begin{equation}\label{mar93}
\mathscr L_\Omega : X(\Omega) \to X(\Omega)
\end{equation}
for every rearrangement-invariant space $ X(\Omega)$,
 such that
\begin{align}\label{eq:repr0}
\bfu(x)=\mathcal R_\Omega(\bfu)(x)+\mathscr L_\Omega(\ep(\bfu))(x) \quad \hbox{for a.e. $x \in \Omega$,}
\end{align}
for every $\bfu\in E^1L^1(\Omega)$.
%
\\
In particular,
 there exists a constant $c=c(n)$ such that
 \begin{align}\label{eq:extE1}
 \|\mathcal R_\Omega   (\bfu)\|_{L^\infty(\Omega)}  \leq \frac{c}{|\Omega|}
\|\bfu\|_{L^1(\Omega)}
 \end{align}
for every  $\bfu\in L^1(\Omega)$, and 
 \begin{align}\label{eq:extgrad}
|\bfR_\Omega   (\bfu)|  \leq c
\|\nabla \bfu\|_{L^\infty(\Omega)}
 \end{align}
for every $\bfu \in W^{1,\infty}(\Omega)$.
Moreover, 
there exists a constant $c=c(n)$ such that, 
 for every  rearrangement-invariant space $X(\Omega)$,
 \begin{align}\label{eq:extE_A}
\|\mathcal R_\Omega   (\bfu)\|_{X(\Omega)} \leq c \|\bfu\|_{X(\Omega)}
 \end{align}
for every   $\bfu\in X(\Omega)$, and
%
  \begin{align}\label{eq:extE2}
\| \mathscr L_\Omega (\bfE)\|_{X(\Omega)}\leq c|\Omega|^{\frac{1}{n}}\|\bfE\|_{X(\Omega)}
 \end{align}
for every  $\bfE \in X(\Omega)$.
%
%
%
%
 \end{lemma}

 \begin{proof}[Proof of Lemma \ref{lem:repr}] Suppose, for the time being, that $x_0=0$. 
By \cite[Lemma 4.3]{Ckorn} -- a consequence of
\cite[Theorem 4]{Ka} -- if  $\omega$ is any function such that
$\omega \in C^\infty _0(\Omega)$ and $\int _{\Omega} \omega
(x)\dx =1$, 
then, for every $\bfu \in E^1L^1(\Omega)$, 
\begin{align}\label{repr2}
u_i(x) &= v_i (x) + \sum _{j=1}^n x_j\int _\Omega \omega(y) \ep_{ij}(\bfu)(y) \dy - \sum _{k,l=1}^n \int _\Omega \frac{\partial
H_{kl}(x,y)}{\partial y_l} \ep_{ik}(\bfu)(y) \dy \\ \nonumber &
\,\,+ \sum _{k,l=1}^n \int _\Omega \frac{\partial
H_{kl}(x,y)}{\partial y_i}\ep_{kl}(\bfu)(y) \dy - \sum _{k,l=1}^n
\int _\Omega \frac{\partial H_{kl}(x,y)}{\partial y_k} \ep_{il}(\bfu)(y) \dy \quad\hbox{for a.e. $x \in \Omega$,}
\end{align}
 $i=1, \dots , n$, where
\begin{align}\label{repr10}
v_i(x)= \int _\Omega [(n+1)\omega (y) + y \cdot \nabla \omega
(y)]u_i(y) \dy - \frac 12 \sum _{j=1}^n x_j
 \int _\Omega \Big( \frac{\partial \omega}{\partial y_i}u_j(y)
- \frac{\partial \omega}{\partial y_j} u_i(y)\Big)\dy \quad
\hbox{for $x \in \rn$,}
\end{align}
$i=1, \dots , n$, and
\begin{equation}\label{repr5}
H_{kl}(x,y) = 2 \frac{(y_k - x_k)(y_l - x_l)}{|y-x|^n} \int
_{|y-x|}^\infty \omega \Big(x + t \frac {y-x}{|y-x|}\Big) t^{n-1}\dt
\quad \hbox{for $x \neq y$,}
\end{equation}
$k, l =1, \dots n$. 
Moreover, the function $\bfv :\rn \to \rn$ given by $\bfv = (v_1, \dots v_n)$ is such that
$\bfv  \in \mathcal R$.
\\
Now, consider the vector $\bfb_\Omega(\bfu) \in \rn$ and the matrix $\bfR_\Omega(\bfu) \in {\mathbb R}^{n \times n}_{\rm skew}$ whose components are given by
\begin{equation}\label{mar90}
\bfb_\Omega(\bfu)_i = \int _\Omega [(n+1)\omega (y) + y \cdot \nabla \omega
(y)]u_i(y) \dy,
\end{equation}
and
\begin{equation}\label{mar91}
\bfR_\Omega(\bfu)_{ij} = 
- \frac 12 
 \int _\Omega \Big( \frac{\partial \omega}{\partial y_i}u_j(y)
- \frac{\partial \omega}{\partial y_j} u_i(y)\Big)\dy,
%
\end{equation}
for $i=1, \dots , n$ and $j=1, \dots , n$. Also, define the operator $\mathscr L_\Omega $, mapping  matrix-valued functions  into vector-valued functions, as 
$$ \mathscr L_\Omega (\bfE)= \mathscr L_\Omega ^1(\bfE) + \mathscr L_\Omega ^2(\bfE)$$
for $\bfE\in L^1(\Omega)$,
where
\begin{equation}\label{mar96} \mathscr L_\Omega ^1(\bfE)_i(x) = 
\sum _{j=1}^n x_j\int _\Omega \omega(y) 
E_{ij}(y) \dy \quad \hbox{for $x \in \Omega$,}
\end{equation}
and 
\begin{align}\label{aug100} \mathscr L_\Omega ^2(\bfE)_i (x)& =
- \sum _{k,l=1}^n \int _\Omega \frac{\partial
H_{kl}(x,y)}{\partial y_l}  E_{ik}(y) \dy \\ \nonumber &
\,\,+ \sum _{k,l=1}^n \int _\Omega \frac{\partial
H_{kl}(x,y)}{\partial y_i} E_{kl}(y) \dy - \sum _{k,l=1}^n
\int _\Omega \frac{\partial H_{kl}(x,y)}{\partial y_k} 
E_{il}(y) \dy  \quad \hbox{for $x \in \Omega$,}
\end{align}
for $i=1, \dots , n$. Then equation \eqref{eq:repr0} holds owing to \eqref{repr2}.
Property \eqref{mar95} is an easy consequence of equations \eqref{mar90} and \eqref{mar91}. As for property \eqref{mar93}, one can verify via equations \eqref{repr5}, \eqref{mar96} and \eqref{aug100} that
there exists a constant $C=C(\Omega)$ such that
\begin{align}\label{aug101'}
| \mathscr L_\Omega (\bfE) (x)| \leq C  \int
_\Omega \frac{|\bfE (y)|}{|x-y|^{n-1}}\dy \quad
\hbox{for a.e. $x \in \Omega$.}
\end{align}
Since $|\Omega|<\infty$, the Riesz potential operator $I_1$ is bounded on $L^1(\Omega)$ and on $L^\infty(\Omega)$. Hence, the operator $ \mathscr L_\Omega$ enjoys the same boundedness properties. Property \eqref{mar93} hence follows, 
via an interpolation theorem by Calder\'on \cite[Chapter 3, Theorem 2.12]{BS}.
%
\\
Now, set  $\lambda= |\Omega|^{\frac 1n}$. Hence, $\Omega= \lambda \Omega_1$, where $\Omega_1$ is either the  cube or the ball, centered at $0$ with $|\Omega_1|=1$.
If $\omega$ is a function as in \eqref{repr2}, with $\Omega$ replaced by $\Omega_1$, then the function $\omega_\lambda$ given by $\omega_\lambda (x) = \lambda ^{-n}\omega(\lambda ^{-1} x)$ 
is such that $\omega _\lambda \in C^\infty_0(\Omega)$ and $\int_\Omega \omega _\lambda (x)\dx=1$.
%
Given a function $\bfu \in L^1(\Omega)$, define $\bfu^1 \in L^1(\Omega_1)$ as $\bfu^1 (z) = \bfu (\lambda z)$ for $z \in \Omega_1$. Thus, via a change of variables in equations \eqref{mar90} and \eqref{mar91} one obtains that
\begin{align}\label{aug101}
\mathcal R_{\Omega}(\bfu)_i (x)& =  \int _{\Omega} [(n+1)\omega_\lambda (y) + y \cdot \nabla \omega_\lambda
(y)]u_i(y) \dy   - \frac 12 \sum _{j=1}^n x_j
 \int _{\Omega} \Big( \frac{\partial \omega_\lambda }{\partial y_i}u_j(y)
- \frac{\partial \omega_\lambda }{\partial y_j} u_i(y)\Big)\dy
\\ \nonumber
&  =
\int _{\Omega_1} [(n+1)\omega (z) + y \cdot \nabla \omega
(z)]u_i^1(z) \dz   - \frac 12 \sum _{j=1}^n \frac{x_j}\lambda 
 \int _{\Omega_1} \Big( \frac{\partial \omega }{\partial z_i}u_j^1(z)
- \frac{\partial \omega}{\partial z_j} u_i^1(z)\Big)\dz
\\ \nonumber &  =  \mathcal R_{\Omega_1}(\bfu^1)_i (x/\lambda)\quad \text{for $x\in \Omega$,}
\end{align}
for $i=1, \dots , n$. Therefore, 
\begin{align}\label{aug102}
\sup_{x \in \Omega} |\mathcal R_{\Omega}(\bfu)(x) | = \sup_{x \in \Omega} |\mathcal R_{\Omega_1}(\bfu^1)(x/\lambda)|  = \sup_{z\in  \Omega_1} |\mathcal R_{\Omega_1}(\bfu^1)(z)|,
\end{align}
whence, by property \eqref{mar92} for $\Omega_1$, there exists a constant $c=c(n)$ such that
\begin{align}\label{aug103}
\|\mathcal R_{\Omega}(\bfu)\|_{L^\infty(\Omega)} & = \|\mathcal R_{ \Omega_1}(\bfu^1)\|_{L^\infty(\Omega_1)}
\leq c \int _{\Omega_1} |\bfu^1(z)|\dz \\ \nonumber & = \frac {c}{\lambda ^n} \int  _{\Omega} |\bfu(x)|\dx =  \frac {c}{|\Omega|} \int  _{\Omega} |\bfu(x)|\dx.
\end{align}
Inequality \eqref{eq:extE1} is thus established. 
\\ Inequality \eqref{eq:extgrad} follows from the fact that, if $\bfu \in W^{1,\infty}(\Omega)$, then
\begin{equation}\label{new25}
\bfR_\Omega(\bfu)_{ij} = 
\frac 12 
 \int _\Omega \Big( \frac{\partial u_j}{\partial y_i}(y)
- \frac{\partial u_i}{\partial y_j} (y)\Big)\omega(y)\dy
\end{equation}
for $i=1, \dots , n$ and $j=1, \dots , n$,
as shown via  an integration by parts in the integral on the right-hand side of \eqref{mar91}.
\\  Inequality \eqref{eq:extE_A} is a consequence of the chain
\begin{align}
\| \mathcal R_\Omega   (\bfu)\|_{X(\Omega)}& \leq \|\mathcal R_{\Omega}(\bfu)\|_{L^\infty(\Omega)} \| 1\|_{X(\Omega)} \leq \frac c{|\Omega|}   \| 1\|_{X(\Omega)} \int  _{\Omega} |\bfu(x)|\dx \\ \nonumber & \leq \frac c{|\Omega|} \| 1\|_{X(\Omega)}  \|1\|_{X'(\Omega)} \|\bfu\|_{X(\Omega)} = c \|\bfu\|_{X(\Omega)},
\end{align}
where the second inequality holds owing to \eqref{eq:extE1} and the last equality is a consequence of identity \eqref{fund1}.
%
%
\\  Consider next   inequality \eqref{eq:extE2}. If $\bfE \in L^1(\Omega)$, then
\begin{align}\label{mar98}
\|\mathscr L_{\Omega} ^1(\bfE)\|_{L^1(\Omega)} & \leq \int_{\Omega}  |x|\int _{\Omega}\omega_\lambda (y) |\bfE (y)|\dy \,\dx  =
\int_{\Omega} \omega_\lambda (y) |\bfE (y)|\dy \,\int_{\Omega} |x|\dx  
\\ \nonumber & = \int_{\Omega} \lambda  \omega(y/\lambda) |\bfE (y)|\dy  \int_{\Omega_1} |z|\dz 
 \leq c \|\lambda  \omega(\cdot/\lambda)\|_{L^\infty(\Omega)} \|\bfE\|_{L^1(\Omega)}
\\ \nonumber & = c \lambda  \|\omega\|_{L^\infty(\Omega_1)} \|\bfE\|_{L^1(\Omega)}
  =  c' |\Omega|^{\frac 1n}  \|\bfE\|_{L^1( \Omega)}.
\end{align}
A similar chain yields 
\begin{align}\label{mar98inf}
\|\mathscr L_{\Omega} ^1(\bfE)\|_{L^\infty(\Omega)}  \leq 
   c |\Omega|^{\frac 1n}  \|\bfE\|_{L^\infty( \Omega)}.
\end{align}
for some constant $c=c(n)$ and for every $\bfE \in L^\infty( \Omega)$.
\\ On the other hand, observe that any derivative $\frac{\partial H_{kl}}{\partial y_m}$ admits an estimate of the form
\begin{align}\label{aug107}
\bigg|\frac{\partial H_{kl}}{\partial y_m}(x,y)\bigg| & \leq \frac{c}{|x-y|^{n-1}} \int
_{|y-x|}^\infty \omega \Big(x + t \frac {y-x}{|y-x|}\Big) t^{n-1}\dt \\
\nonumber & \quad + c\, \omega (y) |x-y| + \frac{c}{|x-y|^{n-1}}\int
_{|y-x|}^\infty \Big|\omega ' \Big(x + t \frac {y-x}{|y-x|}\Big)\Big|t^{n}\dt \quad \text{for $x\neq y$,}
\end{align}
for some constant $c=c(n)$.
Let us choose $\omega$ such that it is radially symmetric with respect to $0$ and vanishes outside a ball of radius $R= c|\Omega|^{\frac 1n}$ for some  constant $c=c(n)$. Consequently, we can assume that $\omega \leq c_1  |\Omega|^{-1}$, and $|\omega '|\leq c_1  |\Omega|^{-1-\frac 1n}$ for some constant $c_1=c_1(n)$. Thus, there exists a constant $c_2=c_2(n)$ such that
$\omega \Big(x + t \frac {y-x}{|y-x|}\Big)= \omega ' \Big(x + t \frac {y-x}{|y-x|}\Big) =0$ if $t>c_2|\Omega|^{\frac 1n}$ and $x,y \in \Omega$, $x\neq y$.
Hence, 
\begin{align}\label{aug108}
\bigg|\frac{\partial H_{kl}}{\partial y_m}(x,y)\bigg|  &\leq \frac{c}{|x-y|^{n-1}}\frac 1{ |\Omega|} \int
_0^{ c_2|\Omega|^{\frac 1n}}    t^{n-1}\dt 
+ c \frac{ |x-y|}{ |\Omega|} + \frac{c}{|x-y|^{n-1}}\frac{1}{ |\Omega|^{1+\frac 1n}}\int
_0^{ c_2|\Omega|^{\frac 1n}}   t^{n}\dt
\\ \nonumber & \leq 
\frac{c'}{|x-y|^{n-1}}   + c' \frac{ |x-y|}{ |\Omega|} \leq \frac{c''}{|x-y|^{n-1}} \quad \text{$x,y \in \Omega$, $x\neq y$,}
\end{align}
for suitable constants $c,c',c''$ depending on $n$.
Notice that the last inequality holds since $|x-y|^n \leq  c |\Omega|$ for some constant $c=c(n)$ and for $x, y \in \Omega$. 
From equations \eqref{aug100}, \eqref{aug108} and \eqref{potrearr} we deduce that there exists a constant $c=c(n)$ such that
\begin{equation}\label{mar99}
\mathscr L_{\Omega}^2(\bfE)^*(s) \leq c\bigg(s^{-\frac 1{n'}}\int _0^s\bfE^*(r)\dr + \int_s^{|\Omega|} \bfE^*(r) r^{-\frac 1{n'}}\dr\bigg) \quad \hbox{for $s>0$.}
\end{equation}
From inequality \eqref{mar99} one can deduce that there exists constants $c=c(n)$, $c'=c'(n)$  such that
\begin{align}\label{mar100}
\|\mathscr L_{ \Omega}^2(\bfE)\|_{L^1(\Omega)} & = \|\mathscr L_{\Omega}^2(\bfE)^*\|_{L^1(0,|\Omega|)} \leq  c \bigg\|s^{-\frac 1{n'}}\int _0^s\bfE^*(r)\dr\bigg\|_{L^1(0,|\Omega|)} +
c \bigg\|\int_s^{|\Omega|} \bfE^*(r) r^{-\frac 1{n'}}\dr\bigg\|_{L^1(0,|\Omega|)}
\\ \nonumber & \leq c' |\Omega|^{\frac 1n} \|\bfE^*\|_{L^1(0,|\Omega|)} = c' |\Omega|^{\frac 1n} \|\bfE\|_{L^1(\Omega)},
\end{align}
for every $\bfE\in L^1(\Omega)$, and
\begin{align}\label{mar100bis}
\|\mathscr L_{ \Omega}^2(\bfE)\|_{L^\infty(\Omega)} & = \|\mathscr L_{\Omega}^2(\bfE)^*\|_{L^\infty(0,|\Omega|)} \leq  c \bigg\|s^{-\frac 1{n'}}\int _0^s\bfE^*(r)\dr\bigg\|_{L^\infty(0,|\Omega|)} +
c \bigg\|\int_s^{|\Omega|} \bfE^*(r) r^{-\frac 1{n'}}\dr\bigg\|_{L^\infty(0,|\Omega|)}
\\ \nonumber & \leq c' |\Omega|^{\frac 1n} \|\bfE^*\|_{L^\infty(0,|\Omega|)} = c' |\Omega|^{\frac 1n} \|\bfE\|_{L^\infty(\Omega)},
\end{align}
%
for every $\bfE\in L^\infty(\Omega)$.
\\ As a consequence of inequalities \eqref{mar98}, \eqref{mar98inf}, \eqref{mar100} and \eqref{mar100bis}, the linear operator $\mathscr L_{\Omega}$ fulfills:
\begin{equation}\label{new1}
\mathscr L_{\Omega} : L^1(\Omega) \to L^1(\Omega), \quad \text{with norm not exceeding $c |\Omega|^{\frac 1n}$,}
\end{equation}
and
\begin{equation}\label{new2}
\mathscr L_{\Omega} : L^\infty(\Omega) \to L^\infty(\Omega), \quad \text{with  norm not exceeding $c |\Omega|^{\frac 1n}$,}
\end{equation}
for some constant $c=c(n)$.  Inequality  \eqref{eq:extE2}  follows from \eqref{new1} and \eqref{new2}, since every rearrangement-invariant space $X(\Omega)$ is an exact interpolation space between $L^1(\Omega)$ and $L^\infty(\Omega)$, see
 \cite[Chapter 3, Theorem 2.12]{BS}.
\\
The proof of the lemma is complete in the case when $x_0=0$. The general case follows from this one applied to  the set $\Omega_0=\Omega - x_0$, and to the function $\bfu_0 : \Omega_0 \to \R$ defined as $\bfu_0(x)= \bfu(x+x_0)$ for $x \in \Omega_0$. Indeed, on setting
\begin{align*}
\mathcal R_\Omega(\bfu)(x) =\mathcal R_{\Omega_0}(\bfu_0)(x-x_0)\quad \text{and} \quad \mathscr L_\Omega(\bfu)(x)=\mathscr L_{\Omega_0}(\bfu_0)(x-x_0) \quad \text{ for  $x\in\Omega$,}
\end{align*}
one has that
$$ \bfu(x)=\mathcal R_\Omega(\bfu)(x)+\mathscr L_\Omega(\bfu)(x) \quad \text{ for a.e. $x\in\Omega$.}$$ 
Moreover, a change of varibales shows that properties \eqref{mar92} and \eqref{mar93}, and inequalities \eqref{eq:extE1}--\eqref{eq:extE2} continue to hold in this case.
  \end{proof}

 \begin{corollary}\label{equivnorm}
 Let $\Omega_1$ and $\Omega_2$ be either cubes or balls,  whose diameters are bounded by each other up to multiplicative constants $c_1$ and $c_2$ and whose distance is bounded by their diameters up to a multiplicative constant $c_3$. Then there exists a constant $c=c(c_1,c_2,c_3,n)$ such that
\begin{equation}\label{new35}
\|\mathcal R_{\Omega_2}(\bfu)\|_{X(\Omega_1)}\leq c \|\bfu\|_{X(\Omega_2)}
\end{equation}
for every $\bfu \in X(\Omega_2)$. Here, $\mathcal R_{\Omega_2}(\bfu)$ is defined as in \eqref{mar95}, and $X(\Omega_1)$ and $X(\Omega_2)$ are  rearrangement-invariant spaces built upon the same function norm (suitably restricted, if $|\Omega_1|\neq |\Omega_2|$, as described in Subsection \ref{ri}).
  \end{corollary}
\begin{proof}
Inequality \eqref{new35} is a consequence of the following chain:
\begin{align}\label{new36}
\|\mathcal R_{\Omega_2}(\bfu)\|_{X(\Omega_1)} &\leq \|\mathcal R_{\Omega_2}(\bfu)\|_{L^\infty(\Omega_1)} \|1\|_{X(\Omega_1)} \leq c \|\mathcal R_{\Omega_2}(\bfu)\|_{L^\infty(\Omega_2)}  \|1\|_{X(\Omega_1)} \leq c' \frac { \|1\|_{X(\Omega_1)}} {|\Omega_2|} \int_{\Omega_2}|\bfu|\dx 
\\ \nonumber &  \leq \frac { \|1\|_{X(\Omega_1)} \|1\|_{X'(\Omega_2)}} {|\Omega_2|}\|\bfu\|_{X(\Omega_2)}
\leq c'' \|\bfu\|_{X(\Omega_2)},
\end{align}
where the constants $c,c',c''$ depend only on   $c_1,c_2,c_3,n$. Notice that the second inequality holds owing to \eqref{norm-equiv4}, the third one by \eqref{eq:extE1} and the last follows from \eqref{fund1}.
\end{proof}

  As a straightforward consequence of Lemma  \ref{lem:repr} we obtain a Poincar\'e type inequality in $E^1L^1(\Omega)$, when $\Omega$ is either  a ball or a  cube in $\rn$, with a dependence of the constant just on their measure and on $n$.

  \begin{corollary}\label{Poinc}
 Assume that $\Omega\subset\R^n$ is either a cube or a ball. Then  there exists a constant $c=c(n)$ such that 
     \begin{align}\label{itm:meanw1}
      \|\bfu-\mathcal{R}_{\Omega}(\bfu)\|_{L^1(\Omega)}&\leq
      c\,|\Omega|^{\frac{1}{n}} \|\ep(\bfu)\|_{L^1(\Omega)}
    \end{align} 
for every $\bfu\in E^1L^1(\Omega)$. Here,  $\mathcal{R}_{\Omega}(\bfu)$ is defined as  in \eqref{mar95}.
  \end{corollary}

 A Poincar\'e type inequality in $E^1L^1(\Omega)$, with a different normalization condition, on a Lipschitz domain $\Omega$ is the subject of the next lemma.

\begin{lemma}\label{poinc-meas0} Let $\Omega$ be a bounded   connected Lipschitz domain in $\R^n$ and let $\gamma  \in (0, |\Omega|)$. Then there exists a constant $C=C(\Omega, \gamma)$ such that
\begin{equation}\label{jan200}
\|\bfu \|_{L^1(\Omega)} \leq C \|\ep (\bfu)\|_{L^1(\Omega)}
\end{equation}
for every function $\bfu \in E^1L^1(\Omega)$ such that $|\{\bfu =0\}|\geq \gamma$.
\end{lemma}
\begin{proof} We follow the outline of the proof of \cite[Lemma 4.1.3]{Ziemer}.
Assume, by contradiction, that inequality \eqref{jan200} fails. Then   there exists a sequence of  functions $\{\bfu_k\}\subset  E^1L^1(\Omega)$ and a sequence of measurable sets $\{E_k\}$ such that $E_k\subset \Omega$, $|E_k| \geq \gamma$, $\bfu_k=0$ in $E_k$, and 
\begin{equation}\label{jan201}
\|\bfu_k \|_{L^1(\Omega)}\geq k \|\ep (\bfu_k)\|_{L^1(\Omega)}
\end{equation}
for every $k \in \setN$. On replacing $\bfu_k$ by $\bfu_k/\|\bfu_k\|_{L^1(\Omega)}$, we may also assume that 
\begin{equation}\label{jan202}
\|\bfu_k\|_{L^1(\Omega)}=1
\end{equation}
 for $k \in \setN$. The sequence $\{\bfu_k\}$ is thus bounded in $E^1L^1(\Omega)$ and  hence, by the compactness of the embedding $E^1L^1(\Omega) \to L^1(\Omega)$ \cite[Chapter 2, Proposition 1.4]{Temam}, there exist a subsequence, still denoted by  $\{\bfu_k\}$,  and a function $\bfu \in L^1(\Omega)$, such that
$\bfu_k \to  \bfu$ in  $L^1(\Omega)$, and $\bfu_k \to \bfu$ a.e. in $\Omega$. Equations \eqref{jan201} and \eqref{jan202} tell us that  $\ep (\bfu_k) \to 0$ in $L^1(\Omega)$. Thus, in fact, $\bfu \in E^1L^1(\Omega)$, and $\ep(\bfu)=0$. As a consequence, $\bfu = \bfb + \bfB x$ for some vector $\bfb \in \R^n$ and some matrix $\bfB \in \R^{n\times n}_{\rm skew}$. On the other hand, owing to equation \eqref{jan202}, $\|\bfu\|_{L^1(\Omega)}=1$. Therefore, 
\begin{equation}\label{jan203}
\text{either 
$\bfb \neq 0$ or $\bfB \neq 0$.}
\end{equation}
However,  since $\bfu_k \to \bfu$ a.e. in $\Omega$, we have that $\bfu =0$ in the set $E=\cap_{k=1}^\infty \cup_{h=k}^\infty E_h$. Since $|E|\geq \gamma >0$, this contradicts \eqref{jan203}.
\end{proof}

In what follows, the term \emph{annulus} denotes a subset of $\rn$ obtained as the difference between an open  ball and a closed concentric ball with smaller radius.

\begin{lemma}\label{poinc-cubes} Let $Q_r$ be an open cube in $\R^n$ with sidelength $r$, let $\lambda >1$ and let $Q_{\lambda r}$ denote the cube, concentric with $Q_r$,  whose sidelength is $\lambda r$.  Let 
$G\subset \R^n$  be any annulus satisfying $\overline {Q_r} \setminus G \neq \emptyset$.
Then there exists a constant $C=C(n, \lambda)$ such that
\begin{equation}\label{jan210}
\|\bfu \|_{L^1(Q_{\lambda r})} \leq r C \|\ep (\bfu)\|_{L^1(Q_{\lambda r})}
\end{equation}
for every function $\bfu \in E^1L^1(Q_{\lambda r})$ such that $\bfu =0$ in $Q_{\lambda r}\setminus G$.
\end{lemma}
\begin{proof} Without loss of generality, we may assume that $Q_r$ is centered at $0$. Suppose, for the time being, that 
\begin{equation}\label{r=1}
r=1.
\end{equation}
Owing to Lemma \ref{poinc-meas0}, inequality \eqref{jan210} will follow if we show that 
\begin{equation}\label{jan220}
\inf_{G \in \mathcal G}|Q_{\lambda} \setminus G| >0,
\end{equation}
where $\mathcal G = \{ G \text{ is an annulus s.t. } \overline {Q_{1}} \setminus G \neq \emptyset\}$.
In order to prove property \eqref{jan220}, consider a sequence of annuli $\{G_k\}$ such that
\begin{equation}\label{jan230}
\overline {Q_1} \setminus G_k \neq \emptyset \quad \text{for $k \in \N$,}
\end{equation}
and 
\begin{equation}\label{jan240}
\lim_{k\to \infty} |Q_{\lambda} \setminus G_k| = \inf_{G \in \mathcal G}|Q_{\lambda } \setminus G|.
\end{equation}
Let $x_k\in \R^n$ and  $0<r_k < R_k$ be such that $G_k=B_{R_k}(x_k)\setminus \overline {B_{r_k}(x_k)}$ for $k \in \setN$.
\\ If there exist subsequences (still indexed by $k$) such that $\{x_k\}$ and $\{R_k\}$ are bounded, then there exists a  further subsequence of $\{G_k\}$, still denoted by  $\{G_k\}$, and a (possibly empty) annulus $G_\infty$, such that 
\begin{equation}\label{jan260}
\chi_{G_k} \to \chi_{G_\infty},
\end{equation}
in $L^1(\rn)$, and  
\begin{equation}\label{jan270}
 \overline {Q_1} \setminus G_\infty \neq \emptyset,
\end{equation}
 whence 
\begin{equation}\label{jan250}
\inf_{G \in \mathcal G}|Q_{\lambda } \setminus G| =  \lim_{k\to \infty} |Q_{\lambda} \setminus G_k| =  |Q_{\lambda} \setminus G_\infty|>0.
\end{equation}
Inequality \eqref{jan220} is thus established in this case. 
\\ If there exist subsequences such that $\{R_k\}$ is bounded and  $\{x_k\}$  is unbounded, then there   exists a  subsequence of $\{G_k\}$, such that
$Q_{\lambda} \setminus G_k = Q_{\lambda}$ for large $k$. This contradicts equation \eqref{jan240}, since 
\begin{equation}\label{jan280}
\inf_{G \in \mathcal G}|Q_{\lambda } \setminus G|< |Q_{\lambda }|.
\end{equation}
\\ If there exist subsequences such that $\{x_k\}$ and $\{R_k\}$ are unbounded, and  $\{r_k\}$ is bounded, 
then there exists a  subsequence $\{G_k\}$ such that equations \eqref{jan260}--\eqref{jan250} hold, where $G_\infty$ is either the whole of $\R^n$, or a half-space, or $\emptyset$. Only the second alternative is admissible, since
the first one is excluded, being not consistent with \eqref{jan270}, whereas the last one is excluded, owing to \eqref{jan280}. Inequality \eqref{jan220} thus follows also in this case. 
\\Finally, assume that there exist subsequences such that $|x_k|\to \infty$, $r_k \to \infty$ and  $R_k\to \infty$. Equation \eqref{jan230} ensures that, for each $k\in \N$, 
\begin{equation}\label{jan290}
\text{either}\,\, \overline {Q_1} \setminus B_{R_k}(x_k) \neq \emptyset \quad \text{or} \quad  \overline {Q_1} \cap \overline {B_{r_k}(x_k)} \neq \emptyset.
\end{equation}
Thus, there exists a further subsequence satisfying either the first or the second condition 
 in \eqref{jan290} for every $k$. If $\overline {Q_1} \setminus B_{R_k}(x_k) \neq \emptyset $ for every $k$, then there exists a half-space  $H_\infty$ such that $\overline {Q_1}\setminus H_\infty \neq \emptyset$, and (up to  subsequences) $B_{R_k}(x_k) \to H_\infty$  and 
\begin{equation}\label{jan300}
\inf_{G \in \mathcal G}|Q_{\lambda } \setminus G| =  \lim_{k\to \infty} |Q_{\lambda} \setminus G_k| \geq    \lim_{k\to \infty} |Q_{\lambda} \setminus B_{R_k}(x_k)|= |Q_{\lambda} \setminus H_\infty|>0.
\end{equation}
If, instead, $\ \overline {Q_1} \cap \overline {B_{r_k}(x_k)} \neq \emptyset$ for every $k$, then there exists again a half-space $H_\infty$ such that $\overline {Q_1}\cap  \overline{H_\infty} \neq \emptyset$, and (up to  subsequences) $B_{r_k}(x_k) \to H_\infty$ and 
\begin{equation}\label{jan310}
\inf_{G \in \mathcal G}|Q_{\lambda } \setminus G| =  \lim_{k\to \infty} |Q_{\lambda} \setminus G_k| \geq    \lim_{k\to \infty} |Q_{\lambda} \cap  \overline {B_{r_k}(x_k)}| =|Q_{\lambda} \cap  \overline{H_\infty}|>0.
\end{equation}
Equation \eqref{jan220} follows from either equation \eqref{jan300} or \eqref{jan310}.
\\ It remains to remove assumption \eqref{r=1}. This can be accomplished via a scaling argument. Given any annulus $G$ as in the statement and any function $\bfu$ as in the statement, consider the function $\bfu_1: Q_1 \to \R^n$, defined as
$$\bfu_1(x) = \bfu (rx) \quad 	\text{for $x \in Q_1$.}$$
Therefore, $\bfu_1 \in E^1L^1(Q_1)$, and $\bfu_1 =0$ in $Q_1 \setminus G_{1/r}$. Here, $G_{1/r}$ denotes the annulus $\{x: r x \in G\}$. Clearly, $\overline {Q_1} \setminus G_{1/r} \neq \emptyset$, since $\overline {Q_r} \setminus G \neq \emptyset$ . Inequality \eqref{jan210} thus follows via an application  to $\bfu_1$ of the same inequality, with $r=1$, and via  the change of variables $y=rx$.
\end{proof}

\section{Truncation operators}\label{truncation}
Our purpose in this section is to construct  truncation operators in symmetric gradient Sobolev spaces. 
%
 The following lemma from \cite{BrDF} and \cite{DRW} provides us with a Whitney type decomposition of an open set in $\rn$, with $n \geq 2$, into dyadic cubes. Such a decomposition
enjoys specific properties needed in the definition of the truncation operators in question.
%
 \\ In what follows, we denote by $r(Q)$ the side-length of a cube $Q \subset \rn$. 

\begin{lemma}\label{lems:whitneysteady}{\rm{\bf \cite{BrDF, DRW}}}
Let $\mathcal O$ be an open set in $\R^n$.
There exists  a   covering  of $\mathcal O$  by closed dyadic cubes $\set{Q_j}_{j\in\N}$  with the following properties:
\begin{enumerate}[label={(4.\arabic{*})}]
\item\label{itm:D1} $\bigcup_{j\in\N} Q_j = \mathcal{O}$ and  $\overset{\circ}{Q}_j \cap \overset{\circ}{Q}_k= \emptyset$  for $j\neq k$.
\item \label{itm:Wbnd} $8 \sqrt{n} r(Q_j) \leq \distance(Q_j,
  \partial\mathcal{O}) \leq 32\,\sqrt{n} r(Q_j)$.
 In
  particular, on setting $c_n = 2+32 \sqrt{n}$, one has that  $(c_n Q_j) \cap
  (\setR^n \setminus \mathcal{O}) \not= \emptyset$.
\item\label{itm:W3'}  If $\partial Q_j \cap \partial Q_k\neq \emptyset$, then
  \begin{align*}
    \frac{1}{2} \leq \frac{r(Q_j)}{r(Q_k)} \leq 2.
  \end{align*}
\item\label{itm:D4} For each $j$, there exist at most $(3^n-1)2^n$ cubes $Q_k$ such that $\partial Q_j\cap \partial Q_k\neq \emptyset$.
\end{enumerate}
\end{lemma}
Let $\set{Q_j}$ be the covering of $\mathcal O$  provided by Lemma \ref{lems:whitneysteady}, and set
\setcounter{equation}{5}
\begin{equation}\label{Q*}
Q_j^* = \tfrac{9}{8} Q_j \quad  \text{and} \quad r_j = r(Q_j).
\end{equation}
The following
properties of the family of cubes $\{Q_j^*\}$ are a consequence of Lemma \ref{lems:whitneysteady}.

\begin{corollary}\label{cor:whitney}
Under the assumptions of Lemma \ref{lems:whitneysteady} the following properties hold:
\begin{enumerate}[label={(4.\arabic{*})},start=6]
\item \label{itm:W1}  $\bigcup_{j\in\N} {Q_j^\ast} = \mathcal{O}$
\item \label{itm:W2} If $Q_j^*\cap Q_k^*\neq \emptyset$, then $\partial Q_j \cap \partial Q_k\neq \emptyset$, and $Q_j^* \subset 5 Q_k^*$.
  Moreover, $r_j\approx r_k$ and $\abs{Q_j^* \cap
    Q_k^*} \approx \abs{Q_j^*} \approx \abs{Q_k^*}$, up to multiplicative constants depending only on $n$.
\item \label{itm:W3} The family $Q_j^*$ is  locally  finite, with a maximum number of overlaps depending only on $n$.
\item \label{itm:W4} There exists a constant $c=c(n)$ such that $\sum_{j\in\N} |Q_j^\ast|\leq
  c|\mathcal{O}|$.
\end{enumerate}
\end{corollary}

The construction of  a partition of unity with respect to the covering  $\{Q_j^*\}_{j\in\N}$, which satisfies the properties stated in the next lemma, is standard.

\begin{lemma}\label{lemma:partition}
Let $\mathcal O$ be open set in $\rn$ and let $\{Q_j\}$   and $\{Q_j^*\}_{j\in\N}$ be families of cubes as in Lemma \ref{lems:whitneysteady} and Corollary \ref{cor:whitney}. Then there exists a partition of unity $\set{\phi_j}_{j\in\N}$ with respect to the covering $\{Q_j^*\}_{j\in\N}$ and  \,\,a constant $c=c(n)$ such that:
\begin{enumerate}[label={(4.\arabic{*})},start=10]
\item \label{itm:U2} $\phi_j \in C^\infty_0(\R^n)$, $\support \phi_j \subset Q_j^*$ and $\sum_{j\in\N}\phi_j=1$ in $\mathcal O$.
\item \label{itm:U3} $\chi_{\frac{7}{9} Q_j^*} = \chi_{\frac{7}{8}
    Q_j} \leq \phi_j \leq \chi_{\frac{9}{8} Q_j} = \chi_{Q_j^*}$.
\item \label{itm:U4} $\abs{\nabla \phi_j} \leq \displaystyle{ \frac{c\,
      \chi_{Q_j^*}}{r_j}}$.
\end{enumerate}
\end{lemma}
\setcounter{equation}{12}
 
 \textcolor{black}{The   truncation operator to be introduced  acts on a function through the level sets of the Hardy-Littlewood maximal operator applied to the relevant function and to its symmetric gradient.}
%
%
Recall that the Hardy-Littlewood maximal operator $M$  is defined at any function $\bfu \in L^1_{\rm loc}(\rn)$ as 
  \begin{equation}\label{Max} M \bfu (x) = \sup _{Q\ni x} \dashint_{Q}|\bfu|\dy \qquad \hbox{for $x \in \rn$,}
   \end{equation}
where $\dashint_{Q}= \tfrac 1{|Q|}\int_Q\dots \dy$,  the averaged integral over $Q$. The operator $M$ is of weak type in any rearrangement-invariant space $X(\rn)$, in the sense that 
\begin{equation}\label{weakX}
t \varphi_X (|\{M\bfu > t\}|)  \leq 
 C_M \|\bfu\| _{X(\R^n)}   \quad \text{for $t>0$,}
\end{equation}
for some constant $C_M=C_M(n)$, and for every $\bfu \in X(\rn)$. 
This is a consequence of (the upper estimate in) a two-sided rearrangement inequality for the operator $M$, which tells us that 
\begin{equation}\label{riesz}
 c_M \bfu^{**} (s) \leq  (M\bfu)^*(s) \leq C_M \bfu^{**} (s) \quad \text{for $s>0$,}
\end{equation}
for some positive constants $c_M=c_M(n)$ and  $C_M=C_M(n)$  \cite[Chapter 3, Theorem 3.8]{BS}. To verify assertion \eqref{weakX}, notice that, thanks to property \eqref{riesz}, to  the H\"older type inequality \eqref{holder}, and to equation \eqref{fund1},
\begin{equation}\label{riesz1}
(M\bfu)^*(s) \leq \frac{C_M}s \|\bfu\| _{ X(0, \infty)} \|\chi_{(0,s)}\| _{X'(0, \infty)} =
 \frac{C_M}s \|\bfu\| _{X(\R^n)}\varphi_{X'}(s) = \frac{C_M}{\varphi_{X}(s)}\|\bfu\| _{X(\R^n)}\ \quad \text{for $s>0$.}
\end{equation}
By the definition of the decreasing rearrangement, inequality \eqref{riesz1} implies \eqref{weakX}.
\par\noindent
Given $\bfu \in E^1L^1(\R^n)$,   we define for $\theta,\lambda>0$ the  set
\begin{align}\label{mar106}
\mathcal O_{\theta,\lambda}&=\set{x\in\R^n:\,M(\bfu)>\theta}\cup\set{x\in\R^n:\,M(\ep(\bfu))>\lambda}.
\end{align}
Note that $\mathcal O_{\theta,\lambda}$ is an open set in $\rn$, thanks to the continuity of the maximal function.
It follows directly from the weak-type $L^1$-estimate for the maximal operator, namely from inequality \eqref{weakX} with $X(\rn)=L^1(\rn)$, that
\begin{align*}
    |\set{x\in\R^n:\,M(\bfu)>\theta}|\leq\frac{C_M}{\theta}\|\bfu\|_{L^1(\R^n)},\quad    |\set{x\in\R^n:\,M(\ep(\bfu))>\lambda}|\leq\frac{C_M}{\lambda}\|\ep(\bfu)\|_{L^1(\R^n)}.
\end{align*}
In particular, $|\mathcal O_{\theta,\lambda}| <     \infty$. More generally, if  $\bfu \in E^1X(\R^n)$ for some rearrangement-invariant space $X(\rn)$, then, by inequality \eqref{weakX} again,
\begin{align}\label{Mar101}
    \varphi_X(|\set{x\in\R^n:\,M(\bfu)>\theta}|)\leq\frac{C_M}{\theta}\|\bfu\|_{X(\R^n)}, \quad   \varphi_X(|\set{x\in\R^n:\,M(\ep(\bfu))>\lambda}|)\leq\frac{C_M}{\lambda}\|\ep(\bfu)\|_{X(\R^n)}.
\end{align}
Hence,  $|\mathcal O_{\theta,\lambda}| <     \infty$ if $\lim _{t \to \infty} \varphi_X(t)>\max\{\frac{C_M}{\theta}\|\bfu\|_{X(\R^n)}, \frac{C_M}{\lambda}\|\ep(\bfu)\|_{X(\R^n)}\}$, and, in particular, if $\lim _{t \to \infty} \varphi_X(t)=\infty$.
%
%
%
\\ Let $\{Q_j\} $, $\{Q_j^*\} $ and $\{\phi_j\} $ be as in Lemmas  \ref{lems:whitneysteady}--\ref{lemma:partition}, with $\mathcal{O}= \mathcal{O}_{\theta,\lambda}$.
\\
Given a rearrangement-invariant space $X(\rn)$ and a function $\bfu \in E^1L^1(\R^n)$, 
 set   
\begin{equation}\label{mar105}
\bfu_j = \mathcal{R}_{Q_j^*}(\bfu) \quad \hbox{for $j \in \setN$,}
\end{equation}
where $\mathcal{R}_{Q_j^*}(\bfu)$  is defined as  in \eqref{mar95}, and 
the
function $T^{\theta,\lambda}\bfu : \rn \to \rn$ is given by
\begin{align}\label{eq:wlambda}
  T^{\theta,\lambda} \bfu &=
  \begin{cases}
    \bfu &\quad\text{in $\R^n \setminus \mathcal{O}_{\theta,\lambda}$}
    \\[2mm]
    \displaystyle{\sum_{j\in\N}} \phi_j \bfu_j&\quad\text{in
        $\mathcal{O}_{\theta,\lambda}$.}
  \end{cases}
\end{align}
\textcolor{black}{Our results about the operator $T^{\theta,\lambda}$ are collected in the following statement}.
\begin{theorem}
  \label{lem:Tlest}{\rm{\bf [Truncation operator acting on $\bfu$ and $\ep(\bfu)$]}} 
\\
 (i) Assume that $\bfu \in E^1L^1(\R^n)$. Then   $T^{\theta,\lambda} \bfu \in E^1L^1(\rn)\cap E^1L^\infty(\R^n)$. Moreover, there exists a constant  $c=c(n)$ such that:
  \begin{align}
      \label{itm:TlestLinfty2} \abs{T^{\theta,\lambda} \bfu} \leq c\,
    \theta \chi_{\mathcal{O}_{\theta,\lambda}} + \abs{\bfu} \chi_{\R^n
      \setminus \mathcal{O}_{\theta,\lambda}}\quad&\text{and}\quad \abs{T^{\theta,\lambda} \bfu}
    \leq c\, \theta \quad\text{a.e. in } \rn,\\
   \label{itm:TlestLinfty} \abs{\ep(T^{\theta,\lambda} \bfu)} \leq c\,
    \lambda \chi_{\mathcal{O}_{\theta,\lambda}} + \abs{\ep(\bfu)} \chi_{\R^n
      \setminus \mathcal{O}_{\theta,\lambda}}\quad&\text{and}\quad  \abs{\ep(T^{\theta,\lambda} \bfu)}
    \leq c\, \lambda \quad\text{a.e. in } \rn.
  \end{align}
 (ii) Assume that $\bfu \in E^{1}X(\R^n)$ for some rearrangement-invariant space $X(\R^n)$. Let $\lambda>0$ be such that
\begin{equation}\label{feb20}
\lim _{t \to \infty} \varphi_X(t) > \frac{2C_M \|\bfu\|_{ E^{1}X(\R^n)}} \lambda, 
\end{equation}
where $C_M$ is the constant appearing in \eqref{weakX}
(in particular,  $\lambda$ can be any positive number if the limit in \eqref{feb20} is infinity).
Then $T^{\lambda,\lambda} \bfu \in  E^{1}X(\R^n)$. Moreover, properties \eqref{itm:TlestLinfty2}  and \eqref{itm:TlestLinfty}  hold with $\theta=\lambda$,
 and there exists a constant $C=C(n, X)$ such that
\begin{equation}\label{jan50}
\| T^{\lambda,\lambda} \bfu \|_{E^{1}X(\R^n)} \leq C \|\bfu\|_{E^{1}X(\R^n)} .
\end{equation}
\end{theorem}

\begin{proof}  Part (i). {\color{black} The outline of the proof of this part is reminiscent of that of   \cite[Lemmas 2.1-2.3]{BrDF}.}
We begin by showing that 
\begin{align}\label{eq:TE11}
T^{\theta,\lambda} \bfu\in E^1L^1(\R^n).
\end{align}
Since $\bfu \in  E^1L^1(\R^n)$ and $T^{\theta,\lambda} \bfu = \bfu$ in $\rn \setminus \mathcal{O}_{\theta,\lambda}$, it suffices to prove that $T^{\theta,\lambda} \bfu - \bfu \in
  E^1_0L^1(\mathcal{O}_{\theta,\lambda})$. Owing to property  \ref{itm:U2}, 
 \begin{align}
  \label{july101}
 T^{\theta,\lambda}\bfu - \bfu = {\color{black} \sum_{j =1}^\infty }\phi_j(\bfu_j - \bfu) \quad \text{in  $\mathcal{O}_{\theta,\lambda}$.}
\end{align}
By property \ref{itm:U2} again, each addend in the sum in \eqref{july101} belongs to $E^1_bL^1(\mathcal{O}_{\theta,\lambda})$, and the sum is locally finite by property \ref{itm:W3}. Therefore, $T^{\theta,\lambda}\bfu - \bfu \in E^1_{\rm loc}L^1(\mathcal{O}_{\theta,\lambda})$. Moreover,
\begin{align}\label{july102}
    \ep( T^{\theta,\lambda} \bfu - \bfu) &
=  \sum_{j =1}^\infty  \big( \nabla
    \phi_j\otimes^{sym}(\bfu_j - \bfu) -\phi_j\ep( \bfu)
    \big) \quad \text{in  $\mathcal{O}_{\theta,\lambda}$,}
  \end{align}
where $\otimes^{sym}$ denotes the symmetric part of the tensor product between two vectors.
In order to show that  $T^{\theta,\lambda} \bfu - \bfu \in
  E^1_0L^1(\mathcal{O}_{\theta,\lambda})$, it thus suffices to show that the sums in \eqref{july101} and \eqref{july102} converge in $L^1 (\mathcal{O}_{\theta,\lambda})$.
%
%
%
%
  Let $k\in \setN$. Thanks to equations \eqref{itm:meanw1} and \ref{itm:U4}, there exists a constant $c=c(n)$ such that
  \begin{align}\label{eq:2001b}
    \int_{\R^n} &{\color{black} \sum_{j=k}^\infty} \bigabs{
      \nabla \phi_j\otimes^{sym}(\bfu_j - \bfu) - \phi_j \ep( \bfu) } \dx
\leq\sum_{j=k}^\infty \int_{Q_j^*} \bigabs{ \nabla
      \phi_j\otimes^{sym}(\bfu_j - \bfu)}\dx +  \sum_{j=k}^\infty\int_{Q_j^*} \abs{\ep(\bfu) } \dx
    \\ \nonumber
    &\leq\sum_{j=k}^\infty \int_{Q_j^*}
    \frac{\abs{\bfu_j - \bfu}}{r_j} \dx + \sum_{j=k}^\infty \int_{Q_j^*} \abs{\ep(\bfu)} \dx
\leq \,c\,\sum_{j=k}^\infty \int_{Q_j^*} \abs{\ep(\bfu)} \dx
    \leq\, c\, \int_{\mathcal{O}_{\theta,\lambda}} \chi_{\cup_{j=k}^\infty Q_j^*} \abs{\ep(\bfu)} \dx.
  \end{align}
  Since  $\ep(\bfu) \in   L^1(\R^n)$ and,  by \ref{itm:W4}, $\chi_{\cup_{j=k}^\infty Q_j^*} \to 0$ as  $k \to
  \infty$,  the sum in the leftmost side of equation \eqref{eq:2001b}
  converges  to $0$  in $L^{1}(\mathcal{O}_{\theta,\lambda})$. Thus, $ \ep( T^{\theta,\lambda} \bfu - \bfu)\in L^{1}(\mathcal{O}_{\theta,\lambda})$. An analogous -- in fact simpler -- argument  applied to the sum on the right-hand side of \eqref{july101} tells us that  $ T^{\theta,\lambda} \bfu - \bfu\in L^{1}(\mathcal{O}_{\theta,\lambda})$ as well.
%
 Property  \eqref{eq:TE11} is thus established.\\
  Le us next focus on properties \eqref{itm:TlestLinfty2} and \eqref{itm:TlestLinfty}.
Denote by $A_j$  the set of indices associated with cubes from the covering, which are neighbours of $Q_j^*$ (including $Q_j^*$ itself); namely
\begin{align}\label{Aj}
  A_j &= \set{k \in \setN \,:\, Q_j^* \cap Q_k^* \not= \emptyset}.
\end{align}
 Consider property \eqref{itm:TlestLinfty2}. Fix $j \in \setN$. Since
\begin{align*}
T^{\theta,\lambda} \bfu=\sum_{k\in A_j}\phi_k\bfu_k \quad \text{in $Q_j^\ast$,}
\end{align*}  
   we deduce from inequality \eqref{eq:extE1},  that
   \begin{align*}
    \abs{T^{\theta,\lambda} \bfu} &\leq c\,\sum_{k\in A_j} \norm{\bfu_k}_{L^\infty(Q_k^*)} \leq c\, \sum_{k \in A_j}
    \dashint_{Q_k^*} \abs{\bfu}\dx \quad \text{a.e. in $Q_j^\ast$,}
  \end{align*}
for some constant $c=c(n)$.
  Hence, via properties \ref{itm:W2}, \ref{itm:W3} and the definition of the maximal operator, there exists a constant $c=c(n)$ such that
    \begin{align}\label{eq:new}
 \abs{T^{\theta,\lambda} \bfu} \leq\, \sum_{k \in A_j}
    \dashint_{c_n Q_k^*} \abs{\bfu}\dx\leq\,c\,\theta \quad \text{a.e. in $Q_j^\ast$,}
  \end{align}
where $c_n$ is the constant appearing in \ref{itm:Wbnd}. Note that here we have made use of the fact that   $c_n Q_k^*\cap (\R^n\setminus\mathcal O_{\theta,\lambda}) \neq \emptyset$.  Inequality \eqref{eq:new} implies that 
  $\abs{T^{\theta,\lambda} \bfu} \leq c\,
    \theta$ in  $\mathcal{O}_{\theta,\lambda}$. In $\R^n
      \setminus \mathcal{O}_{\theta,\lambda}$, one has that $\abs{T^{\theta,\lambda} \bfu} = |\bfu| \leq M(\bfu) \leq \theta$ by the very definition of $\mathcal O_{\theta,\lambda}$. Hence,
      $\abs{T^{\theta,\lambda} \bfu} \leq c\,
    \theta$ in $\rn$. Property \eqref{itm:TlestLinfty2} is thus fully established.
\\
Consider next ~\eqref{itm:TlestLinfty}. 
We have that
\begin{align}
  \label{eq:Tlnabla}
  \ep(T^{\theta,\lambda} \bfu) &= \chi_{\R^d \setminus \mathcal{O}_{\theta,\lambda}}
  \ep(\bfu) + \chi_{\mathcal{O}_{\theta,\lambda}} \sum_{j=1}^\infty \ep(\phi_j
  \bfu_j).
\end{align}
Moreover, for each $j \in \setN$,
  \begin{align*}
    \eps (T^{\theta,\lambda} \bfu) &= \eps \Big( \sum_{k=1}^\infty \phi_k \bfu_k \Big) =
    \eps \Big(  \sum_{k=1}^\infty \phi_k (\bfu_k -\bfu_j)\Big)
    = \sum_{k=1}^\infty \nabla \phi_k \otimes^\sym (\bfu_k - \bfu_j) \quad \text{in
  $Q_j^*$,}
  \end{align*}
  where we have made use of the fact that,  by \ref{itm:U2}, $\sum_{k=1}^\infty \phi_k=1$ in
  $\mathcal{O}_{\theta,\lambda}$  and that  $\ep (\bfu_j) =0$.  Therefore, owing to the local finiteness of 
  the covering $\{Q_k^*\}$, 
  \begin{align}\label{eq:neu0}
    \abs{\ep (T^{\theta,\lambda}\bfu)} &\leq c\,\sum_{k\in A_j} \frac{1}{r_j} \norm{\bfu_j -
        \bfu_k}_{L^\infty(Q_j^*)}\quad \text{in
  $Q_j^*$,}
  \end{align}
for some constant $c=c(n)$.
By property   \ref{itm:W2}, we have that $\abs{Q_j^* \cap Q_k^*} \approx \abs{Q_j^*} \approx \abs{Q_k^*}$.  Therefore, thanks to equations \eqref{norm-equiv3} and \eqref{norm-equiv4},
%
if $ k\in A_j$, then
  \begin{align}\label{eq:2008c}
    \frac{1}{r_j}\norm{\bfu_j -
        \bfu_k}_{L^\infty(Q_j^*)} &\lesssim \dashint_{Q_j^*} \frac{\abs{\bfu_j - \bfu_k}}{r_j}\dx \lesssim 
    \frac{1}{\abs{Q_j^*}} \int_{Q_j^* \cap Q_k^*} \frac{\abs{\bfu_j -
        \bfu_k}}{r_j}\dx
    \\ \nonumber
    &\lesssim \frac{1}{\abs{Q_j^*}} \int_{Q_j^* \cap Q_k^*} \frac{\abs{\bfu -
        \bfu_j}}{r_j}\dx + \frac{1}{\abs{Q_j^*}} \int_{Q_j^* \cap Q_k^*}
    \frac{\abs{\bfu - \bfu_k}}{r_j}\dx
    \\ \nonumber
    &\lesssim \dashint_{Q_j^*} \frac{\abs{\bfu - \bfu_j}}{r_j}\dx +
     \dashint_{Q_k^*} \frac{\abs{\bfu -
        \bfu_k}}{r_k}\dx,
  \end{align}
up to constants depending on $n$.
  Observe that  we have exploited  property ~\ref{itm:W2}
  twice in the last chain.
Combining equations \eqref{eq:2008c} and \eqref{eq:neu0} yields
 \begin{align}\label{eq:neu}
    \abs{\ep (T^{\theta,\lambda})} & \leq c\, \sum_{k \in A_j}
    \dashint_{Q_k^*} \frac{\abs{\bfu - \bfu_k}}{r_k}\dx
  \end{align}
for each $j \in \setN$.
  From inequality \eqref{itm:meanw1} and the  local finiteness of 
  the covering $\{Q_k^*\}$ we deduce, via the same argument as in the poof  \eqref{eq:new}, that 
      \begin{align}\label{eq:2001}
 \abs{\ep(T^{\theta,\lambda} \bfu)} \leq\, \sum_{k \in A_j}
    \dashint_{c_n Q_k^*} \abs{\ep(\bfu)}\dx\leq\,c\,\lambda \quad \text{a.e. in $Q_j^*$,}
 \end{align}
for some constant $c=c(n)$.
 Inasmuch as $\bigcup_j Q_j^* = \mathcal{O}_{\theta,\lambda}$, inequality \eqref{eq:2001}
 implies that $\abs{\eps (T^{\theta,\lambda} \bfu)} \leq c\,\lambda$ a.e. in
  $\mathcal{O}_{\theta,\lambda}$. As a consequence, $\abs{\ep(T^{\theta,\lambda} \bfu)}
  \leq c\, \lambda \chi_{\mathcal{O}_{\theta,\lambda}} + \abs{\ep(\bfu)}
  \chi_{\rn \setminus \mathcal{O}_{\theta,\lambda}}$.  On the other hand,  $\abs{\eps (T^{\theta,\lambda} \bfu)} =
  \abs{\eps(\bfu)} \leq M(\eps(\bfu)) \leq \lambda$  in $\R^n \setminus
  \mathcal{O}_{\theta,\lambda}$, by the definition of $\mathcal O_{\theta,\lambda}$. Altogether,
  $\abs{\eps (T^{\theta,\lambda} \bfu)} \leq c\,\lambda$ a.e. in $\rn$.  Property \eqref{itm:TlestLinfty} is thus established, and hence the proof of Part (i) is complete.
 \\  Part (ii). Assume that $\lambda >0$ fulfills condition \eqref{feb20}. Observe that
\begin{align}\label{jan51}
\varphi_X(|\mathcal O_{\lambda, \lambda}|) \leq \varphi_X(|\set{M(\bfu)>\lambda}|+ |\set{M(\ep(\bfu))>\lambda}|) 
& 
\leq   \varphi_X(2|\set{M(\bfu)>\lambda}|)+ \varphi_X(2|\set{M(\ep(\bfu))>\lambda}|)\\ \nonumber &\leq     2\varphi_X(|\set{M(\bfu)>\lambda}|)+ 2\varphi_X(|\set{M(\ep(\bfu))>\lambda}|), 
\end{align}
for $\lambda >0$, where the last inequality holds since $\varphi_X(t)/t$ is non-increasing. Hence,   by the weak-type estimate \eqref{weakX} and assumption \eqref{feb20},  one has that $|\mathcal O_{\lambda, \lambda}|<\infty$. Since $\bfu \in E^1X(\rn)$ and $|\mathcal O_{\lambda, \lambda}|<\infty$, from property \eqref{l1linf} we infer  that $\bfu \in E^1L^1(\mathcal O_{\lambda, \lambda})$.
The same argument as in Part (i) then ensures that $T^{\lambda, \lambda} \bfu - \bfu \in E^1_0L^1(\mathcal O_{\lambda, \lambda})$. Since $\bfu \in E^1X(\rn)$, we thus conclude that $T^{\lambda, \lambda} \bfu \in E^1_{\rm loc}L^1(\rn)$. Properties \eqref{itm:TlestLinfty2} and \eqref{itm:TlestLinfty} continue to hold with the same proof. 
Via property \eqref{itm:TlestLinfty2}, inequality  \eqref{jan51} and the fact the the operator $M$ is of weak type in $X(\R^n)$,  one deduces that
\begin{align}\label{jan52}
\|T^{\lambda, \lambda} \bfu\|_{X(\R^n)}& \leq c \lambda \|\chi_{\mathcal O_{\lambda, \lambda}}\|_{X(\R^n)} + \| \bfu\|_{X(\R^n)} = c \lambda  \varphi_X(|\mathcal O_{\lambda, \lambda}|) + \| \bfu\|_{X(\R^n)}
\\ \nonumber & \leq 2 c \lambda \big[\varphi_X(|\set{M(\bfu)>\lambda}|)+ \varphi_X(|\set{M(\ep(\bfu))>\lambda}|)\big]+ \| \bfu\|_{X(\R^n)} \leq c' \| \bfu\|_{E^1X(\R^n)} 
\end{align}
for some constants $c=c(n)$ and  $c'=c'(n)$.
Property \eqref{itm:TlestLinfty}, via a chain analogous to \eqref{jan52},  implies that 
\begin{align}\label{jan53}
\|\ep (T^{\lambda, \lambda} \bfu)\|_{X(\R^n)}& \leq c  \| \bfu\|_{E^1X(\R^n)}
\end{align}
for some constant $c=c(n)$.  Inequality \eqref{jan50}, and hence the fact that $T^{\lambda, \lambda} \bfu \in E^1X(\rn)$, are  consequences of \eqref{jan52} and \eqref{jan53}.
\end{proof}

We now introduce  a variant of the truncation operator from Theorem \ref{lem:Tlest}, where the truncation only depends on the symmetric gradient, and  not on the function itself.
Let $X(\R^n)$ be a rearrangement-invariant space. Given $\bfu \in E^{1}X(\R^n)$,   define for $\lambda>0$ the  set
\begin{align}\label{Olambda}
\mathcal O_{\lambda}&= \set{x\in\R^n:\,M(\ep(\bfu))>\lambda},
\end{align}
where $M$ is the Hardy-Littlewood maximal operator. 
The weak-type estimate \eqref{weakX} for the maximal operator implies that
\begin{align}\label{feb29}
 \varphi_X\big(|\set{x\in\R^n:\,M(\ep(\bfu))>\lambda}|\big)\leq\frac{c}{\lambda}\|\ep(\bfu)\|_{X(\R^n)}.
\end{align}
Let $\{Q_j\}$, $\{Q_j^*\}$ and $\{\phi_j\}$ be as in Lemmas  \ref{lems:whitneysteady}--\ref{lemma:partition}, with $\mathcal{O}= \mathcal{O}_{\lambda}$.
\\
Define the function 
 $T^{\lambda}\bfu$ by
\begin{align}\label{eq:wlambda'}
  T^{\lambda} \bfu &=
  \begin{cases}
    \bfu &\quad\text{in $\R^n \setminus \mathcal{O}_{\lambda}$}
    \\[2mm]
    \displaystyle{\sum_{j=1}^\infty} \phi_j \bfu_j&\quad\text{in
        $\mathcal{O}_{\lambda}$},
  \end{cases}
\end{align}
where $\bfu_j = \mathcal{R}_{Q_j^*}(\bfu)$ for $j \in \setN$, and   $\mathcal{R}_{Q_j^*}(\bfu)$ is defined as  in \eqref{mar95}. \textcolor{black}{The next result provides us with a counterpart of Theorem  \ref{lem:Tlest} for the operator $T^{\lambda}$.}

\begin{theorem}
  \label{lem:feb1}{\rm{\bf [Truncation operator acting on  $\ep(\bfu)$]}} 
\\
 (i) Assume that $\bfu \in E^{1}L^1(\R^n)$. Then $T^{\lambda} \bfu\in E^{1}L^1(\R^n)$ and $\ep (T^{\lambda} \bfu) \in L^\infty (\rn)$. Moreover, there exists a constant $c=c(n)$ such that 
  \begin{equation}
\label{feb60'} \abs{\ep(T^{\lambda} \bfu)} \leq c\,
    \lambda \chi_{\mathcal{O}_{\lambda}} + \abs{\ep(\bfu)} \chi_{\R^n
      \setminus \mathcal{O}_{\lambda}}\quad\text{and}\quad \abs{\ep(T^{\lambda} \bfu)}
    \leq c\, \lambda\quad \text{a.e. in $\rn$.}
  \end{equation}
(ii) Assume   that $\bfu \in E^{1}_bX(\R^n)$ for some rearrangement-invariant space $X(\R^n)$.
Then $T^{\lambda} \bfu \in E^{1}_bL^1(\R^n)\cap E^{1}_bL^\infty(\R^n)$. In particular,
 $T^{\lambda} \bfu \in  E^{1}_bX(\R^n)$. Moreover,
equation \eqref{feb60'} holds, and 
 there exists a constant $C=C(n)$ such that
\begin{equation}\label{jan50'}
\| \ep(T^{\lambda} \bfu )\|_{X(\R^n)} \leq C \|\ep(\bfu)\|_{X(\R^n)} .
\end{equation}
\end{theorem}
\begin{proof}[Proof of Theorem \ref{lem:feb1}] Part (i) can be proved along the same  lines as in the proof of Theorem \ref{lem:Tlest}. We omit the datails for brevity.
\\
Part (ii). If $\bfu \in E^1_bX(\rn)$, then $\bfu \in E^{1}L^1(\R^n)$ as well.  By part (i),  $T^{\lambda} \bfu \in E^{1}L^1(\R^n)$, $\ep (T^{\lambda} \bfu) \in L^\infty (\rn)$ and property \eqref{feb60'} holds. The fact that $\bfu \in E^{1}_bL^1(\R^n)$  ensures that  the set $ \{|T^\lambda \bfu| >0\}$ is bounded. 
From inequality \eqref{rearrepu1}, applied with $\bfu$ replaced by $T^\lambda \bfu$, one has that
$$\|T^\lambda \bfu\|_{L^\infty(\rn)} = (T^\lambda \bfu)^*(0)\leq \int_0^{|\{|T^\lambda \bfu| >0\}|}\ep(T^\lambda\bfu)^*(r)r^{-\frac 1{n'}}\dr \leq n  |\{|T^\lambda \bfu| >0\}|^{\frac 1n}|\ep( T^\lambda\bfu)\|_{L^\infty(\rn)}.$$ 
%
Thus, $T^{\lambda} \bfu \in E^{1}_bL^1(\R^n)\cap E^{1}_bL^\infty(\R^n)$.
The proof of inequality \eqref{jan50'} is analogous to that of \eqref{jan50}.
%
%
\end{proof}

We conclude this section with another variant of
 Theorem \ref{lem:Tlest}, that provides us with a truncation operator which acts on functions defined on bounded open sets in $\rn$,
%
%
 and preserves zero boundary conditions.  Let $\Omega\subset\R^n$ be a bounded open set and let  $\bfu\in E^1_0L^1(\Omega)$. Then
$\bfu_e\in E^1_bL^1(\R^n)$. For simplicity, the function $\bfu_e$ will be denoted by $\bfu$ in the remaining part of this section.  Let $\mathcal O_{\theta,\lambda}$ be defined as in \eqref{mar106}. We mimic a construction from   \cite{DKS}, and  set
\begin{align}\label{may3}
 \bfu_j ^0&=
  \begin{cases}
     \mathcal{R}_{Q_j^*}(\bfu)&\quad\text{if}\,\sqrt{\frac{9}{8}}Q_j\subset\Omega
    \\[2mm]
    0&\quad\text{elsewhere}.
  \end{cases}
\end{align}
Note that the definition of $ \bfu_j ^0$ differs from that of $ \bfu_j $, given in \eqref{mar105}, only  for those indices $j \in\setN$ such that  $\sqrt{\frac{9}{8}}Q_j\cap (\rn \setminus\Omega) \neq \emptyset$. In the latter case,  $ \bfu_j ^0=0$.
Next, 
 define $T_0^{\theta,\lambda}\bfu$ by
\begin{align}\label{eq:wlambda0}
  T^{\theta,\lambda}_0 \bfu &=
  \begin{cases}
    \bfu &\quad\text{in $\R^n \setminus \mathcal{O}_{\theta,\lambda}$}
    \\[2mm]
    \displaystyle{\sum_{j\in\N}} \phi_j  \bfu_j^0&\quad\text{in
        $\mathcal{O}_{\theta,\lambda}$}.
  \end{cases}
\end{align}
The crucial observation is now that, if  $\Omega$ is regular enough, e.g. a bounded Lipschitz domain, then, since  $\bfu=0$ in $\rn \setminus \Omega$,   there exists a constant $c=c(\Omega)$ such that
\begin{align}\label{eq:1301}
\int_{Q_j^\ast}\frac{|\bfu|}{r_j}\dx\leq \,c\,\int_{Q_j^\ast}|\ep(\bfu)|\dx
\end{align}
if    $\sqrt{\frac{9}{8}}Q_j\cap (\rn \setminus\Omega) \neq \emptyset$.
Inequality \eqref{eq:1301} can be derived via Lemma \ref{poinc-meas0}. In particular, for sets of special form, the constant $c$ in inequality  \eqref{eq:1301} only depends on $n$. Lemma \ref{poinc-cubes}  ensures that this is true 
when $\Omega$ is an annulus, the case of interest in view of our applications.
\par On making use of inequality \eqref{eq:1301}, instead of \eqref{itm:meanw1},   in inequalities   \eqref{eq:2001b} and \eqref{eq:2001} for those indices $j$ such that $\sqrt{\frac{9}{8}}Q_j\cap (\rn \setminus\Omega) \neq \emptyset$,   the proof of Theorem \ref{lem:Tlest} carries over verbatim to establish the following result.
%
\begin{theorem}
  \label{cor:Tlest}{\rm{\bf [Truncation operator preserving zero boundary values]}}
  Let  $\Omega$ be a bounded Lipschitz domain in $\R^n$. 
\\ (i) If $\bfu \in E^1_0L^1(\Omega)$,  then   $T^{\theta,\lambda}_0 \bfu \in E^1_0L^\infty(\Omega)$. Moreover, there exists a constant $c=c(\Omega)$ such that
  \begin{align}
      \label{itm:TlestLinfty20} \abs{T^{\theta,\lambda}_0 \bfu} \leq c\,
    \theta \chi_{\mathcal{O}_{\theta,\lambda}} + \abs{\bfu} \chi_{\Omega
      \setminus \mathcal{O}_{\theta,\lambda}}\quad&\text{and}\quad \abs{T^{\theta,\lambda}_0 \bfu}
    \leq c\, \theta \quad\text{a.e. in } \rn,\\
  \label{itm:TlestLinfty0} \abs{\ep(T^{\theta,\lambda}_0 \bfu)} \leq c\,
    \lambda \chi_{\mathcal{O}_{\theta,\lambda}} + \abs{\ep(\bfu)} \chi_{\Omega
      \setminus \mathcal{O}_{\theta,\lambda}}\quad&\text{and}\quad \abs{\ep(T^{\theta,\lambda}_0 \bfu)}
    \leq c\, \lambda\quad\text{a.e. in }\rn.
  \end{align}
(ii) Assume that $\bfu \in E^{1}_0X(\Omega)$ for some rearrangement-invariant space $X(\Omega)$. Then  $T^{\lambda,\lambda}_0 \in E^{1}_0X(\Omega)$. Moreover, properties  \eqref{itm:TlestLinfty20}  and \eqref{itm:TlestLinfty0} hold with $\theta=\lambda$, and
 there exists a constant $C=C(\Omega)$ such that
\begin{equation}\label{jan50bis}
\| T^{\lambda,\lambda}_0 \bfu \|_{E^{1}X(\Omega)} \leq C \|\bfu\|_{E^{1}X(\Omega)} .
\end{equation}
In particular, if $\Omega$ is an annulus, then the constant $c$ in \eqref{itm:TlestLinfty20} and \eqref{itm:TlestLinfty0}, and the constant $C$ in \eqref{jan50bis} depend only on $n$.
\end{theorem}

\section{Smooth approximation up to the boundary}\label{smoothapprox}
This section is devoted to an approximation result for functions 
in $E^1L^1(\Omega)$  by functions in $C^\infty(\overline\Omega)$  in any   $(\varepsilon,\delta)$-domain $\Omega$. Recall that
an open    set $\Omega \subset \rn$, {\color{black} with $n \geq 2$,} is called an
  $(\varepsilon,\delta)$-domain if 
 there exist $\varepsilon,\delta>0$ such that for any $x,y\in\Omega$, with $|x-y|<\delta$, there exists a rectifiable curve $\gamma \subset \Omega$ connecting $x$ and $y$, with length $\ell(\gamma)$, satisfying
\begin{align}\label{eq:epsdelta1}
\ell(\gamma)&\leq\frac{1}{\varepsilon}|x-y|,\\
\mathrm{dist}(z,\partial\Omega)&\geq\varepsilon\frac{|x-z||y-z|}{|x-y|}\quad\forall z\in\gamma.\label{eq:epsdelta2}
\end{align}
The notion of  $(\varepsilon,\delta)$-domain was introduced in \cite{Jo}, where it is shown that any domain of this kind is an extension domain for the Sobolev space $W^{m,p}(\Omega)$ for every $m\in \setN$ and $p \in [1, \infty]$. The class of $(\varepsilon,\delta)$-domains is known to be, in a sense, the largest one supporting the extension property for Sobolev spaces. It includes, the class of domains with \emph{minimally smooth boundary} as defined in \cite[Chapter 6, Section 4]{stein}, and, in particular, the class of bounded Lipschitz domains.
The extension property for symmetric gradient Sobolev space is established in the next section.

\begin{theorem}\label{thm:smooth}{\rm{\bf [Smooth approximation up to the boundary on $(\varepsilon,\delta)$-domains]}}
Let $\Omega$ be an $(\varepsilon,\delta)$-domain and let $\bfu\in E^1L^1(\Omega)$. Then there exists a sequence $\{\bfu_j\} \subset C^\infty(\overline\Omega)$ such that $\bfu_j\rightarrow\bfu$ in $E^1L^1(\Omega)$.
\end{theorem}

The proof of Theorem \ref{thm:smooth} is split in several lemmas.  We begin by introducing a family of coverings of $\rn$ as follows.
For each~$j \in \setZ$, let $\{B_{j,k}\}_{k\in \setN}$ denote a  family of open balls,
with diameter~$r(B_{j,k})$   ,
 such that:
\begin{equation}\label{B2} \text{The family $\{\tfrac 78 B_{j,k}\}_k$ covers~$\R^n$;}
\end{equation}
\begin{equation}
\label{B1} \tfrac 18 \cdot 2^{-j} \leq r(B_{j,k}) \leq \tfrac 14 \cdot
  2^{-j};
\end{equation}
\begin{equation}\label{B3} \text{Each family~$\{B_{j,k}\}_k$ is locally finite, with overlap 
  constant  $c=c(n)$, i.e.
    $\sup_j \sum_k \chi_{B_{j,k}} \leq c.$}
\end{equation}
For each~$j\in \setZ$, let~$\{\eta_{j,k}\}_k\subset C^\infty_0(\rn)$ be a partition of unity with
respect to the family $\{B_{j,k}\}_k$, such that 
\begin{align}
  \label{eq:rho}
  \norm{\eta_{j,k}}_{L^\infty(\rn)} + r(B_{j,k}) 
  \norm{\nabla \eta_{j,k}}_{L^\infty(\rn)} &\leq c,
\end{align}
for some constant $c=c(n)$.
\\
Given an open set $\Omega \subset \rn$, define the inner $2^{-j}$-neighbourhood~$U_j$ of~$\partial \Omega$   by
\begin{align}\label{Uj}
  U_j = \set{ x \in \Omega\,:\, {\rm dist}(x,\partial \Omega) < 2^{-j}}.
\end{align}
If  $\Omega$ is a connected $(\varepsilon,\delta)$-domain,  one can associate with each ball $B_{j,k}$ sufficiently close to $\partial \Omega$  a \lq\lq reflected
  ball" $B_{j,k}^\sharp \subset \Omega$ enjoying the properties described in the next lemma, which is inspired by   \cite[Lemma 2.4]{Jo}.
\\ Throughout this section, the constants in the relations $\lq\lq \approx"$ and $\lq\lq \lesssim "$ depend only on $n$.

\begin{lemma}\label{itm:B1} Let $\Omega$ be a connected $(\varepsilon,\delta)$-domain and let $U_j$ be the set defined by \eqref{Uj} for $j \in \N$. Then there exists~$j_0 \in \N$ with the following properties: for each ball $B_{j,k}$ with $j \geq j_0$ and
  $B_{j,k} \cap U_j \neq \emptyset$, there exists a
  ball~$B_{j,k}^\sharp \subset \Omega$ such that
  $r(B_{j,k}^\sharp) \approx r(B_{j,k}) \approx
  {\rm dist}(B_{j,k}^\sharp, \partial \Omega)$ and ${\rm dist}(B_{j,k},B_{j,k}^\sharp)  \approx r(B_{j,k})$,
  with equivalence constants independent of~$j$ and $k$.
\end{lemma}
\begin{proof}
 Fix $C>0$ to be chosen later.  Let $j_0\in \setN$ be such that $\frac 12\tfrac{C}{\varepsilon}2^{-j_0}< {\rm diam}(\Omega)$. Hence, by equation \eqref{B1}, 
\begin{equation}\label{mar110}
2\tfrac{C}{\varepsilon}r(B_{j_0,k})< {\rm diam}(\Omega)
\end{equation}
for every $k\in \setN$. Given $j \geq j_0$ and $k \in \setN$, let $B_{jk}$ be such that
$B_{jk}\cap U_j\neq\emptyset$. Fix $x_0\in B_{jk}\cap U_j$. We claim that there exists  $y_0\in\Omega$ satisfying
\begin{align}\label{eq:smooth1}
|x_0-y_0|=\frac{C}{\varepsilon}r(B_{jk}).
\end{align}
Indeed,
equation \eqref{eq:smooth1} amounts to saying that $y_0$ belongs to the sphere centered at $x_0$ with radius $\tfrac{C}{\varepsilon}r(B_{j_0,k})$. If $y_0$ did not exist, then the intersection of this sphere with $\Omega$ would be empty. Since $\Omega$ is connected, it would be contained in the ball with the same center and radius, thus contradicting inequality \eqref{mar110}.
Since $\Omega$ is an $(\varepsilon,\delta)$-domain, there exists a curve $\gamma \subset \Omega$ connecting $x_0$ and $y_0$ and satisfying \eqref{eq:epsdelta1} and \eqref{eq:epsdelta2}. Choose a point $z_0$ on  this curve, with the property  that
\begin{align*}
    \min\{|x_0-z_0|,|y_0-z_0|\}\geq \frac{1}{2}|x_0-y_0|.
\end{align*}
From equations \eqref{eq:epsdelta2} and \eqref{eq:smooth1} we infer that
\begin{align}\label{eq:2909}
    \mathrm{dist}(z_0,\partial\Omega)\geq \frac{\varepsilon}{4}|x_0-y_0|= \frac{C}{4}r(B_{j,k}).
\end{align}
Now, choose a ball $B_{jk}^\sharp$ from the covering $\{B_{j,k}\}_k$ such that $z_0\in B_{jk}^\sharp$.
  By inequalities \eqref{B1},  one has that $r(B_{jk}^\sharp)\approx 2^{-j}$. Therefore, by \eqref{eq:2909}, the constant $C$ can be chosen so large  that $B_{jk}^\sharp\subset\Omega$. Moreover, owing to \eqref{eq:smooth1},  $\mathrm{dist}(B_{j,k}, B_{j,k}^\sharp)\approx r(B_{j,k})$.
\end{proof}
The following lemma tells us  
that a ball  and the reflected ball constructed in Lemma \ref{itm:B1} can be connected by a chain of equivalent balls. The subsequent result ensures that a similar property holds  for the reflected balls of two balls with equivalent diameters and nonempty intersection.
\begin{lemma}\label{itm:B2} Let $\Omega$ be a connected  $(\varepsilon,\delta)$-domain.    There exists $j_1 \in \setN$ with the following property.
Assume that 
  $B_{j,k}\subset\Omega$ and   $B_{j,k} \cap U_j \neq \emptyset$ for some $j \geq j_1$ and $k \in \setN$, and let $B^\sharp_{jk}$ be the reflected ball of $B_{j,k}$ introduced in Lemma \ref{itm:B1}. Then there exists a chain of
  balls~$\mathcal B_1,\dots, \mathcal B_h  \subset \Omega$, with $h \in \setN$ uniformly  bounded independently of $j$ and $k$, such
  that
  \begin{align}
  \label{a'} &\mathcal B_1 = B_{j,k}\quad \text{and}\quad \mathcal B_h = B_{j,k}^\sharp;\\
  \label{b'}  \abs{\mathcal B_i \cap \mathcal B_{i+1}} &\approx \abs{\mathcal B_i} \approx
    \abs{\mathcal B_{i+1}} \approx \abs{B_{j,k}}\quad \text{for}\quad i=1, \dots,h-1;\\
  \label{c'}  &r(\mathcal B_i) \approx
    r(B_{j,k})\quad\text{for}\quad i=1,\dots, h;
  \end{align}
  with equivalence  constants  independent of~$j,k,i$.  
\end{lemma}
\begin{proof}
By Lemma \ref{itm:B1}, if $j \geq j_0$, then   ${\rm dist}(B_{j,k},B_{j,k}^\sharp) \approx r(B_{j,k})  \approx 2^{-j}$ for $k \in \setN$. Thus, 
since $\Omega$ is an $(\varepsilon,\delta)$-domain, there exists $j_1\geq j_0$ such that, if   $j \geq j_1$, then
 the centers $x$ and $x^\sharp$ of the  balls $B_{j,k}$ and $B_{j,k}^\sharp$ fulfill the inequality $|x-x^\sharp|<\delta$, and hence can be joined by a curve $\gamma$ satisfying properties \eqref{eq:epsdelta1} and \eqref{eq:epsdelta2}. 
 By inequality \eqref{eq:epsdelta1}, we have that $\ell(\gamma)\approx 2^{-j}$. Moreover, inequality \eqref{eq:epsdelta2}
yields
\begin{align*}
    \mathrm{dist}(z,\partial\Omega)\geq c(\varepsilon)2^{-j} \quad \text{for $z\in\gamma\setminus \big(B_{j,k}\cap B_{jk}^\sharp\big)$. }
\end{align*}
Now, it suffices to choose balls $\mathcal B_i$ centered on $\gamma\setminus \big(B_{j,k}\cap B_{jk}^\sharp\big)$, with  diameter $\tfrac{1}{2}c(\varepsilon)2^{-j}$.
The number of balls needed to cover $\gamma$, with the measure of the intersection of consecutive balls  equivalent to their measure,    is proportional to $ \frac{\ell(\gamma)}{\tfrac{1}{2}c(\varepsilon)2^{-j}}$, a quantity which is in its turn equivalent to a constant depending only on $\varepsilon$.
%
\end{proof}

\begin{lemma}\label{itm:B3} Let $\Omega$ be a connected $(\varepsilon,\delta)$-domain  and let $U_j$ be the set defined by \eqref{Uj} for $j \in \N$.  There exists $j_2\in \setN$  with the following properties: if $B_{j,k} \cap B_{l,m} \neq \emptyset$,  $B_{j,k} \cap U_j \neq \emptyset$, $B_{l,m} \cap U_m \neq \emptyset$, with
  $j,l \geq j_2$, $\abs{j-l} \leq 1$ and $k \in\setN$, then there exists a chain of
  balls~$\mathcal B_1,\dots, \mathcal B_h   \subset \Omega$, with $h\in \setN$ uniformly bounded independently of $j, k,l,m$,  such that
  \begin{align}
  \label{lem:53a}\mathcal B_1 &= B_{j,k}^\sharp\quad \text{and} \quad \mathcal B_h = B_{l,m}^\sharp;\\
 \label{lem:53b} \abs{\mathcal B_i \cap \mathcal B_{i+1}} &\approx \abs{\mathcal B_i} \approx
    \abs{\mathcal B_{i+1}} \approx \abs{B_{j,k}}\quad\text{for}\quad i=1, \dots,h-1;\\
  \label{lem:53c}\mathrm{dist}(\mathcal B_i, \partial \Omega) &\approx r(\mathcal B_i) \approx
    r(B_{j,k})\quad\text{for}\quad i=1,\dots, h;
  \end{align}
  the equivalence constants being independent of~$j,k,i$.
\end{lemma}
\begin{proof}
 The assumption  $\abs{j-l} \leq 1$ and property \eqref{B1} ensure that  $r(B_{j,k})\approx r(B_{l,m})$. Thus, by Lemma \ref{itm:B1},    $r (B_{j,k}^\sharp) \approx r (B_{l,m}^\sharp)$,  and 
 $\mathrm{dist}(B_{j,k}^\sharp,B_{l,m}^\sharp)\approx 2^{-j}$. The proof can then be accomplished along the same lines as that of Lemma \ref{itm:B2}. Let us just point out that, in connection with property \eqref{lem:53c},  an estimate of the form 
$\mathrm{dist}(\mathcal B_i, \partial \Omega)  \lesssim 2^{-j}$
holds 
since $\mathrm{dist}(B_{j,k}^\sharp, \partial \Omega) \approx  2^{-j}$ and  $\mathrm{dist}(B_{l,m}^\sharp, \partial \Omega)\approx  2^{-j}$. Hence, 
if the maximal distance from $ \partial \Omega$ of the  points on the curve  appearing in the definition of $(\varepsilon, \delta)$-domain grew faster than $ 2^{-j}$, then property \eqref{eq:epsdelta1} would be violated.
%
\end{proof}
Let us denote by $\Pi_{B_{j,k}^\sharp}$  the orthogonal projection in $L^2(B_{j,k}^\sharp)$ onto the space $\mathcal R$.
Since $\mathcal R$ is a finite dimensional space, the operator $\Pi_{B_{j,k}^\sharp}$ can be extended to  a linear bounded operator  on $L^1(B_{j,k}^\sharp)$ as follows.
Let $\{\bfv_1,\dots,\bfv_N\}$ be an orthonormal basis of $\mathcal R$ with respect to the scalar product in $L^2(B_{j,k}^\sharp)$. Then  
\begin{align}\label{eq:1803}
   \Pi_{B_{j,k}^\sharp} \bfu=\sum_{i=1}^N\bigg(\int_{B_{j,k}^\sharp}\bfu\cdot\bfv_i\dx\bigg)\bfv_i
\end{align}
for   $\bfu\in L^2(B_{j,k}^\sharp)$. Since the functions $\bfv_i \in L^\infty (B_{j,k}^\sharp)$,
the right-hand side of equation \eqref{eq:1803} is well defined also for $\bfu\in L^1(B_{j,k}^\sharp)$. 
Moreover,  there exists a constant $c=c(n)$ such that
\begin{align}
  \label{eq:PiB-L1}
  \dashint_{B_{j,k}^\sharp} \abs{\Pi_{B_{j,k}^\sharp} \bfu}\dx \leq \,c\,  \dashint_{B_{j,k}^\sharp} \abs{\bfu}\dx
  \end{align}
for $\bfu\in L^1(B_{j,k}^\sharp)$. \\
Next, let $\{\eta_{jk}\}$ and $\{U_j\}$ be the families of functions and of subsets of $\Omega$ introduced above, and let $j_2\in\setN$ be the index appearing in the statement of Lemma  \ref{itm:B3}.
Let $\{\rho_j\} \subset   C^\infty(\rn)$ be a sequence such that
$\chi_{U_{j+1}} \leq \rho_j \leq \chi_{U_j}$ and
$\norm{\nabla \rho_j}_{L^\infty(\rn)} \lesssim 2^j$ for $j \in \setN$.
 For each $j \geq j_2$, define  the function  $T_j \bfu :   \rn  \to \R^n$
as
\begin{align}
  \label{eq:def-Tj}
  T_j \bfu  =
          (1-\rho_j) \bfu + \rho_j \sum_{k=1}^\infty \eta_{j,k} \Pi_{j,k},
\end{align}
where the term  $(1-\rho_j) \bfu$ is extended by $0$ in $\rn \setminus \Omega$. 
Note that
$$
 T_j \bfu = \bfu - \rho_j \sum_{k=1}^\infty \eta_{j,k} \big(\bfu - \Pi_{j,k} \bfu\big) \quad \text{in $\Omega$.}
$$
Thanks to Lemmas \ref{itm:B1}--\ref{itm:B3}, one can apply  \cite[Corollary 4.10]{BDG} and deduce the following result.
\begin{lemma}\label{lem:BDG}
    Let $\Omega$ be a connected  $(\varepsilon,\delta)$-domain and let
$\bfu \in E^1L^1(\Omega)$. Then  $T_j\bfu\in C^\infty(\overline{U}_{j+1})$ and   $T_j\bfu\rightarrow \bfu$ in $E^1L^1(\Omega)$.
\end{lemma}

We are now ready to accomplish the proof of Theorem  \ref{thm:smooth}.

\begin{proof}[Proof of Theorem \ref{thm:smooth}]  It suffices to prove the result under the assumption that $\Omega$ is connected. Indeed, if $\Omega$ is disconnected, then the distance between two connected components is at least $\delta$. Hence, an approximating sequence on the whole of $\Omega$ can be obtained by gluing the approximating sequences in each connected component, via cut-off functions {\color{black} with disjoint supports and equibounded gradients}.
\\ We make use of an argument from  \cite[Proof of Corollary 4.15]{BDG}. Let $\zeta_\kappa \colon\R^{n}\to\R$ be a (even and non-negative) standard mollifier. It is well known that, for each $j\in \setN$, the sequence defined as 
  \begin{equation}\label{mar130}
\bfu_{j,\kappa} = \rho_{j+1} T_j \bfu + ((1-\rho_{j+1}) T_j \bfu) *
  \zeta_\kappa
\end{equation}
  converges to $T_j \bfu$ in $E^1L^1(\Omega)$ as  $\kappa \to 0$. Hence, for each $j \in \setN$, there exists ~$\kappa_j$ such that
  \begin{align*}
    \norm{\bfu_{j,\kappa_j} - T_j \bfu}_{L^1(\Omega)} &\leq 2^{-j},
    \\
    \norm{\ep(T_j \bfu) - \ep(\bfu_{j,\kappa_j})}_{L^1(\Omega)}
                                                  &\leq 2^{-j}.
  \end{align*}
  Since,  by Lemma \ref{lem:BDG},  $T_j\bfu \to \bfu$   in $E^1L^1(\Omega)$,   the sequence $\{\bfu_j\}$, defined as  $\bfu_j = \bfu_{j,\kappa_j}$, converges to $\bfu$ in $E^1L^1(\Omega)$. Moreover,  
  $\bfu_j\in C^\infty(\overline\Omega)$, since $\mathrm{supp}(\rho_{j})\subset \overline U_{j+1}$ and $T^j \bfu \in C^\infty (\overline U_{j+1})$ by Lemma \ref{lem:BDG}.
\end{proof}

\section{Extension operators}\label{extension}

The objective of this section is to establish an extension theorem on  $(\varepsilon,\delta)$-domains $\Omega$, that carries the result of \cite{Jo} over to symmetric gradient Sobolev spaces. This is a key step in view of the computation of the $K$-functional of symmetric gradient Sobolev spaces on domains of this kind. The $K$-functional is in  turn a main ingredient for our characterization of optimal embeddings for these spaces. A basic, yet fundamental, version of the extension result in question is the subject of Theorem \ref{thm:extensionoperator} below. Let us mention that an extension operator on $E^1L^1 (\Omega )$, under stronger regularity assumptions on $\Omega$,  is also constructed in \cite[Theorem 4.1]{GmRa}. Our approach is reminiscent of those of \cite{Jo} and   \cite{GmRa}. 
However, a major novelty of Theorem \ref{thm:extensionoperator} is a pointwise estimate for the rearrangement of the symmetric gradient of the extension of functions defined in the domain.
%
%
%
%

\begin{theorem}\label{thm:extensionoperator} {\rm{\bf [Extension operator on $(\varepsilon,\delta)$-domains]}}
Assume that $\Omega$ is an $(\varepsilon,\delta)$-domain in $\rn$. Then there exists a 
 linear extension operator $\mathscr{E}_\Omega : L^1(\Omega) + L^\infty (\Omega)\to  L^1(\rn) + L^\infty (\rn)$, such that
\begin{equation}\label{ext0}
\mathscr{E}_\Omega \bfu = \bfu \quad \hbox{in $\Omega$}
\end{equation}
for $\bfu\in  L^1(\Omega) + L^\infty (\Omega)$, and  
\begin{align}\label{eq:epu1'}
\mathscr{E}_\Omega : L^1 (\Omega ) \to L^1 (\rn ), \quad
\mathscr{E}_\Omega : L^\infty (\Omega ) \to L^\infty (\rn ).
%
%
\end{align}
Moreover,
\begin{equation}\label{new18}
\mathscr{E}_\Omega : E^1 L^1(\Omega) +E^1 L^\infty (\Omega) \to  E^1 L^1(\rn) + E^1 L^\infty (\rn),
\end{equation}
\begin{equation}\label{ext1}
\mathscr{E}_\Omega : E^1L^1 (\Omega ) \to E^1L^1 (\rn ), \quad
\mathscr{E}_\Omega : E^1L^\infty (\Omega ) \to E^1L^\infty (\rn ),
\end{equation}
and there exists a constant $c=c(\Omega)$ such that
\begin{equation}\label{ext*}
\int_0^t \mathcal E( \mathscr{E}_\Omega\bfu)^*(s)\, ds \leq c \int_0^t \mathcal E( \bfu)^*(s)\, ds \qquad \text{for $t \geq 0$,}
\end{equation}
for every $\bfu\in  E^1 L^1(\Omega) + E^1L^\infty (\Omega) $. Here, $\mathcal E$ denotes the operator defined by \eqref{new23}.
\end{theorem}

The next theorem generalizes Theorem \ref{thm:extensionoperator} via an interpolation argument, and provides us with an extension operator from 
 $E^1X(\Omega)$ into $E^1\widehat X(\Omega)$ for any rearrangement-invariant space $X(\Omega)$ and any of its extensions $\widehat X(\Omega)$ as defined  in Subsection \ref{ri}. This is a special case of a more general interpolation result, of independent interest, which is stated in the same theorem.
%
\begin{theorem}\label{C1}{\rm{\bf [Interpolation for extension operators on $(\varepsilon,\delta)$-domains]}}
Assume that $\Omega$ is an $(\varepsilon,\delta)$-domain in $\rn$.
%
%
Let $T : E^1L^1(\Omega) +
 E^1L^\infty(\Omega) \to E^1L^1(\rn) +
 E^1L^\infty(\rn)$ be a linear operator  
such that
\begin{equation}\label{boundT}
T: E^1L^1(\Omega) \to E^1L^1(\rn) \quad \hbox{and} \quad T:
E^1L^\infty(\Omega)  \to E^1L^\infty(\rn)\,,
\end{equation}
with norms  $M_1$ and $M_\infty$, respectively. Let $X (\Omega )$ be  a rearrangement-invariant space and let
$\widehat X(\rn )$ be  an extension of $X(\Omega )$.
 Then
\begin{equation}\label{boundTX}
 T : E^1X(\Omega) \to E^1\widehat X(\rn)\,,
 \end{equation}
 with norm depending on {\color{black} $\Omega$, $M_1$ and
$M_\infty$}.\\
In particular, the extension operator $\mathscr{E}_\Omega$ provided by Theorem \ref{thm:extensionoperator}  is such that
%
\begin{equation}\label{feb2} 
\mathscr{E}_\Omega : E^1X(\Omega) \to E^1\widehat X(\R^n),
\end{equation}
for every rearrangement-invariant space  $X (\Omega )$ and any of its extensions $\widehat X(\rn )$.
%
\end{theorem}

The remaining part of this section is devoted to a proof of Theorem 
 \ref{thm:extensionoperator}. The proof of Theorem \ref{C1} requires the use of $K$-functionals, and is given at the end of the next section.

\begin{proof}[Proof of Theorem \ref{thm:extensionoperator}]
Let  $\{Q_j\}$ be a   covering of $\R^n\setminus\overline\Omega$ and  $\{\widetilde{Q}_j\}$ a   covering of $\Omega$ as in Lemma \ref{lems:whitneysteady}. Denote the centre of  the cube $Q_j$ by $x_j$ and its side-length by $r_j$, and by $\widetilde x_j$ and $\widetilde r_j$ the centre and the side-length of $\widetilde Q_j$. Define $\mathbb M$ as 
the  subset of $\N$ of those indices $j$ such that  $Q_j$ is \lq\lq close to $\partial \Omega$" in the sense that 
\begin{equation}\label{red}
r_j\leq\, c\varepsilon\delta
\end{equation}
 for some constant $c$. We choose this constant $c$, depending on $n$, in such a way that \cite[Lemmas 2.4--2.7]{Jo} apply to the family of cubes $\{Q_j\}_{j\in \mathbb M}$. 
In particular, with each cube
 $Q_j$,  $j\in\mathbb M$, it is  associated a  cube  from the  covering $\{\widetilde{Q}_j\}$, which, after   relabelling,   can be assumed  to be $\widetilde Q_j$ and has the property that
\begin{equation}
 \label{itm:R1} 
r_j\leq \widetilde r_j\leq 4r_j \quad \text{ for $j\in\mathbb M$,}
\end{equation}
\begin{equation}
\label{itm:R2}  \mathrm{dist}(Q_j,\widetilde Q_j)\leq C r_j \quad \text{ for $j\in\mathbb M$,}
\end{equation}
for some constant $C=C(n)$. 
Moreover, if the cubes $Q_j$ and $Q_k$ are neighbours,  then \cite[Lemma 2.8]{Jo} enables us to connect the associated cubes $\widetilde Q_j$ and $\widetilde Q_k$ by a chain of $m$ cubes of equivalent size.
The number $m=m_{jk}$ of cubes in this chain is uniformly bounded independently of $j$ and $k$. Precisely, there exist
cubes $\widetilde Q_{j_1},\dots,\widetilde Q_{j_m}$ with $\widetilde Q_{j_1}=\widetilde Q_j$ and $\widetilde Q_{j_m}=\widetilde Q_{k}$ such that $\widetilde Q_{j_h}\cap \widetilde Q_{j_{h+1}}\neq \emptyset$ for $h=1, \dots, m_{jk}-1$.
By property \ref{itm:W3'}, we have that $2^{1-h} \widetilde r_j\leq \widetilde r_{j_h}\leq 2^{h-1} \widetilde r_j$ and $2^{h-m_{jk}} \widetilde r_k\leq \widetilde r_{j_h}\leq 2^{m_{jk}-h} \widetilde r_k$, 
whence
\begin{align}\label{eq:chain'}
    \frac{1}{c}\min\{\widetilde r_j,\widetilde r_k\}\leq \widetilde r_{j_h}\leq \,c\,\max\{\widetilde r_j,\widetilde r_k\}\quad \text{for $h=1,\dots,m_{jk}$,}
\end{align}
for some positive constant $c$ independent of $j$ and $k$, since  $m_{jk}$ is uniformly bounded.\\
Let $Q_j^\ast$ and $\widetilde Q_j^\ast$ be the cubes associated with $Q_j$ and $\widetilde Q_j$, respectively, as in \eqref{Q*}, and hence enjoying the properties described in Corollary \ref{cor:whitney}.  Let 
$\{\phi_j\}_{j\in\mathbb M}$ be  a partition of unity 
associated with the cubes $\{Q_j^\ast\}_{j\in\mathbb M}$ as in Lemma \ref{lemma:partition}, and hence
satisfying properties  \ref{itm:U2}--\ref{itm:U4}. 
\\
{\color{black}
Given a function $\bfu \in  L^1(\Omega) + L^\infty (\Omega)$, define the function $  \mathscr E_\Omega \bfu \in  L^1(\rn)+ L^\infty(\rn)$ as
 \begin{align}\label{Eomega}
  \mathscr E_\Omega \bfu &=
  \begin{cases}
    \bfu &\quad\text{in  $ \Omega$}
    \\[2mm]
    \displaystyle{\sum_{j\in\mathbb M}} \phi_j \widetilde\bfu_j&\quad\text{in
        $\R^n\setminus  \overline \Omega$,}
  \end{cases}
\end{align}}
where $\widetilde\bfu_j = \widetilde{\mathcal{R}}_{j} \bfu$ for $j \in \mathbb M$ and 
we have set $\widetilde{\mathcal{R}}_j=\mathcal{R}_{\widetilde Q_j^\ast}$, according to  definition \eqref{mar95}.
  Observe  that equation \eqref{Eomega}  defines $\mathscr E_\Omega \bfu $ a.e. in $\rn$, since $|\partial \Omega|=0$, by \cite[Lemma  2.3]{Jo}.
Finally,  since $\phi_j$ is  compactly supported in  $Q_j^\ast$, then   $\sum_{j\in\mathbb M} \phi_j \widetilde\bfu_j$ vanishes  outside
  $\bigcup_{j\in\mathbb M}\mathrm{supp}(\phi_j)$.   In particular, $\mathscr E_\Omega \bfu$ is supported in a neighbourhood of $\Omega$. 
\\
We begin by
 establishing property \eqref{eq:epu1'}.
 By inequality \eqref{new35}, there exists a constant $c=c(n)$ such that 
 \begin{align*}
\|\widetilde\bfu_j\|_{L^\infty( Q_j^\ast)}=\|\widetilde{\mathcal{R}}_j(\bfu)\|_{L^\infty(Q_j^\ast)}
\leq\,c\,\|\bfu\|_{L^\infty(\widetilde Q_j^\ast)}\leq\,c\,\|\bfu\|_{L^\infty(\Omega)}
\end{align*}
for every $j\in\mathbb M$. 
Since {\color{black} $\{\phi_j\}$} is a partition of unity,  property \eqref{eq:epu1'}  for $L^\infty$ follows.
Similarly, if $\bfu \in L^1(\Omega)$, then, by 
  property   \ref{itm:W3} and inequality \eqref{new35}
\begin{align}\label{eq:2708}
\|\mathscr E_\Omega \bfu\|_{L^1(\rn \setminus \overline \Omega)}=   \int_{\R^n}\Big|\sum_{j\in\mathbb M} \phi_j \widetilde \bfu_j\Big|\dx&\leq \sum_{j\in\mathbb M}\int_{Q_j^\ast}|\widetilde \bfu_j|\dx=\sum_{j\in\mathbb M}\int_{Q_j^\ast}|\widetilde{\mathcal{R}}_j(\bfu) |\dx
\\
\nonumber
&  
\leq\,c\,\sum_{j\in\mathbb M}\int_{\widetilde Q_j^\ast}|\bfu|\dx\leq\,c'\,\int_\Omega|\bfu|\dx,
\end{align}
for some constants $c=c(n)$ and $c'=c'(n)$, whence \eqref{eq:epu1'}  follows for $L^1$ as well.
\\
{\color{black}
Assume now that $\bfu \in E^1 L^1(\Omega) + L^\infty(\Omega)$. Hence, in particular, $\bfu\in E^1L^1(\widetilde Q_j^\ast)$ for $j\in \mathbb M$.
Owing to Lemma \ref{lem:repr},  any such function 
 can be represented as
\begin{align}\label{eq:reprj}
\bfu =\widetilde{\mathcal{R}}_j(\bfu)+\widetilde{\mathscr{L}}_j(\ep(\bfu)) \quad \text{in $\widetilde Q_j^\ast$,}
\end{align}
where we have set   $\widetilde{\mathscr{L}}_j=\mathscr L_{\widetilde Q_j^\ast}$, according to  \eqref{mar93}.}
\\
Given  $j\in \mathbb M$
define, as in \eqref{Aj}, $A_j = \set{k \in \setN \,:\, Q_j^* \cap Q_k^* \not= \emptyset}$, and 
\begin{equation}\label{ajm}
A_j ^{\mathbb M}= \set{k \in \mathbb M \,:\, Q_j^* \cap Q_k^* \not= \emptyset}.
\end{equation}
 Consider the subset of $\mathbb M$ defined by $$\mathbb P = \{j\in \setN :  A_j ^{\mathbb M}= A_j\}.$$
Thus, $\{Q_j\}_{j\in \mathbb P}$ is the family of those cubes  from the family $\{Q_j\}_{j\in \mathbb M}$, whose neighbours also belong to the family $\{Q_j\}_{j\in \mathbb M}$.
Define $\mathcal W_{\mathbb P} \, = \cup_{j \in \mathbb P}\, Q_j$.
Note that, owing to \ref{itm:U2}, we have that  $ \widetilde\bfu_j= \sum_{k \in A_j ^{\mathbb M}} \phi_k \widetilde \bfu_j$ in $Q_j^\ast\cap \mathcal W_{\mathbb P}$. Hence, since $\ep( \widetilde\bfu_j)=0$, it follows that $ \sum_{k \in A_j ^{\mathbb M}} \nabla \phi_k \otimes^\sym  \widetilde\bfu_j=0$ in $Q_j^\ast\cap\mathcal W_{\mathbb P}$. Inasmuch as  $\ep( \widetilde\bfu_k)=0$ as well, if $x\in Q_j^\ast\cap \mathcal W_{\mathbb P}$ then  
  \begin{align}\label{apr201}
    \eps (\mathscr E_\Omega\bfu)(x) &
    = \sum_{k \in A_j ^{\mathbb M}} \nabla \phi_k(x) \otimes^\sym (\widetilde\bfu_k(x) - \widetilde\bfu_j(x))\\ \nonumber &= \sum_{k \in A_j ^{\mathbb M}} \nabla \phi_k \otimes^\sym (\widetilde{\mathcal{R}}_k(\bfu)(x) - \widetilde{\mathcal{R}}_j(\bfu)(x)).
 \end{align} 
  If $Q_j^\ast\setminus\mathcal W_{\mathbb P}\neq \emptyset$ and
$x\in Q_j^\ast\setminus\mathcal W_{\mathbb P}$,  then the sums  in equation \eqref{apr201} have to be complemented with the  extra term
  \begin{align}\label{eq:2308}
 \sum_{k \in A_j^{\mathbb M}} \nabla \phi_k(x) \otimes^\sym \widetilde \bfu _j(x),
  \end{align}
 since  $\sum_{k \in A_j^{\mathbb M}}{\phi_k}\leq 1$ in $Q_j^\ast\setminus\mathcal W_{\mathbb P}$, and hence 
 $ \sum_{k \in A_j^{\mathbb M}}\nabla \phi_k \otimes^\sym  \widetilde\bfu_j\neq 0$.
\\ Now, observe that, if 
$k\in A_j^{\mathbb M}$  then  $\partial Q_j\cap \partial Q_k \neq \emptyset$,  but, in general, it may happen  that $\partial \widetilde Q_j \cap \partial \widetilde Q_k= \emptyset $. However,  
$\widetilde Q_j$ and $\widetilde Q_k$ can be connected by a chain of cubes $\widetilde Q_{j_1},\dots,\widetilde Q_{j_{m_{jk}}}$, such that  $\widetilde Q_{j_1}=\widetilde Q_j$ and $\widetilde Q_{j_{m_{jk}}}=\widetilde Q_{k}$, and satisfying \eqref{eq:chain'}. 
Consequently, by \ref{itm:W3} and \eqref{eq:chain'}, one has that $|\widetilde Q_{j_h}^\ast\cap \widetilde Q_{j_{h-1}}^\ast| \approx |\widetilde Q_j| \approx |\widetilde Q_k|$ for $h=2,\dots, m_{jk}$. 
%
%
Thanks to equation  \eqref{eq:reprj},  the following equality holds in $Q_j^\ast\cap \mathcal W_{\mathbb P}$:
    \begin{align}\label{eq:chain}
  \eps (\mathscr E_\Omega\bfu)
&    =  \sum_{k \in A_j^{\mathbb M}}\sum_{h=2}^{m_{jk}} \nabla \phi_k \otimes^\sym \Big(\widetilde{\mathcal{R}}_{j_h}(\bfu)  - \widetilde {\mathcal R}_{j_{h-1}}(\bfu)\Big).
  \end{align}
  In $Q_j^\ast\setminus \mathcal W_{\mathbb P}$, 
  the extra addend  \eqref{eq:2308} appears in  \eqref{eq:chain}.
Now,  let us set
  \begin{align*}
\mathscr L(\bfu)&=\chi_{\Omega}\eps(\bfu)+\mathscr L^{1}(\bfu)+\mathscr L^{2}(\bfu),\\
    \mathscr L^{1}(\bfu)&= \sum_{j\in\mathbb M}\chi_{Q_j^\ast\cap \mathcal W_{\mathbb P}}\sum_{k \in A_j^{\mathbb M}}\sum_{h=2}^{m_{jk}} \nabla \phi_k \otimes^\sym \Big(\widetilde{\mathcal{R}}_{j_h}(\bfu) - \widetilde {\mathcal R}_{j_{h-1}}(\bfu)\Big), \\
      \mathscr L^{2}(\bfu)&=\sum_{j\in\mathbb M}\chi_{Q_j^\ast\setminus \mathcal W_{\mathbb P}}\sum_{k \in A_j^{\mathbb M}} \nabla \phi_k \otimes^\sym \widetilde\bfu_j.
  \end{align*}
As a next step,
we prove that $\mathscr L^1 : E^1L^\infty (\Omega) \to L^\infty(\rn)$ and
$\mathscr L^1 : E^1L^1 (\Omega) \to L^1(\rn)$,  as well as $\mathscr L^2  : L^\infty (\Omega) \to L^\infty (\rn)$ and $\mathscr L^2  : L^1 (\Omega) \to L^1 (\rn)$. \\
  First, consider $\mathscr L^{1}$.  
  Property \ref{itm:U4} implies  that $|\nabla\phi_k|\lesssim r_k^{-1}$. Hence,   if $\bfu \in E^1L^\infty(\Omega)$, then
 \begin{align}\label{eq:1408}
\|\mathscr L^{1}(\bfu)\|_{L^\infty(Q_j^*)}
&\lesssim \sum_{k\in A_j^{\mathbb M}}\sum_{h=2}^{m_{jk}}r_k^{-1}\|\widetilde{\mathcal{R}}_{j_h}(\bfu) - \widetilde {\mathcal R}_{j_{h-1}}(\bfu)\|_{L^\infty(Q_{j_h}^\ast)}
\\ \nonumber
&\lesssim\sum_{k\in A_j^{\mathbb M}}\sum_{h=2}^{m_{jk}}r_k^{-1}\|\widetilde{\mathcal{R}}_{j_h}(\bfu) - \widetilde {\mathcal R}_{j_{h-1}}(\bfu)\|_{L^\infty(\tilde Q_{j_h}^\ast\cap \tilde Q_{j_{h-1}}^\ast)}
\\ \nonumber
&= \sum_{k\in A_j^{\mathbb M}}\sum_{h=2}^{m_{jk}}r_k^{-1}\|\widetilde{\mathscr{L}}_{j_h}(\ep(\bfu)) - \widetilde {\mathscr L}_{j_{h-1}}(\ep(\bfu))\|_{L^\infty(\tilde Q_{j_h}^\ast\cap \tilde Q_{j_{h-1}}^\ast)}
\\ \nonumber
&\lesssim \sum_{k\in A_j^{\mathbb M}}\sum_{h=2}^{m_{jk}}\big(\|\ep(\bfu)\|_{L^\infty(\widetilde Q_{j_h}^\ast)} +\|\ep(\bfu)\|_{L^\infty(\tilde Q_{j_{h-1}}^\ast)}\big) \lesssim \|\eps(\bfu)\|_{L^\infty(\Omega)}
\end{align}
for every  $Q_j^*$, where the  second inequality holds by equation \eqref{norm-equiv4},   the third inequality by inequality \eqref{eq:extE2}, and the last one
since the sums   are extended to a finite set of addends, with a uniform upper bound for the number of addends.
Therefore, $\mathscr L^{1}: E^1L^\infty(\Omega)\rightarrow L^\infty(\R^n)$. 
\\ The boundedness of $\mathscr L^1 : E^1L^1 (\Omega) \to L^1(\rn)$ follows from the fact that,  similarly to \eqref{eq:1408},  
\begin{align*}
\int_{\R^n}&\bigg|\sum_{j\in\mathbb M}\chi_{Q_j^\ast\cap\mathcal W_{\mathbb P}}\sum_{k\in A_j^{\mathbb M}}\sum_{h=2}^{m_{jk}}\nabla \phi_k \otimes^\sym\big(\widetilde{\mathcal{R}}_{j_h}(\bfu)- \widetilde {\mathcal R}_{j_{h-1}}(\bfu)\big)\bigg|\dx
\\ \nonumber & \lesssim  \sum_{j\in\mathbb M}\sum_{k\in A_j^{\mathbb M}}\sum_{h=2}^{m_{jk}}  \frac 1{r_k} \|\widetilde{\mathscr{L}}_{j_h}(\ep(\bfu)) - \widetilde {\mathscr L}_{j_{h-1}}(\ep(\bfu))\|_{L^1(\tilde Q_{j_h}^\ast\cap \tilde Q_{j_{h-1}}^\ast)}
\\ \nonumber &
\lesssim \sum_{j\in\mathbb M}\sum_{k\in A_j^{\mathbb M}}\sum_{h=2}^{m_{jk}}\big(\|\eps(\bfu)\|_{L^1(\widetilde Q_{j_{h}}^\ast)}+ \|\ep(\bfu)\|_{L^1(\tilde Q_{j_{h-1}}^\ast)}\big) \lesssim \|\eps(\bfu)\|_{L^1(\Omega)}
\end{align*}
for $\bfu \in  E^1L^1(\Omega)$.  
\\
As far as $\mathscr L^{2}$ is concerned, notice that only cubes   $Q_j^*$, with $j\in\mathbb M$, such that $Q_j^\ast\setminus\mathcal W_{\mathbb P}\neq \emptyset$  are in question. By \ref{itm:W2}, each cube with this property  satisfies {\color{black} $\partial Q_j \cap \partial Q_0$ for some  cube $Q_0$ in the family $\{Q_j\}_{j\in\N\setminus\mathbb P}$. Hence, by property \ref{itm:W3'} and the definition of $\mathbb M$, its side-length $r_0$ satisfies
$r_0> \tfrac c2 \varepsilon\delta$.} Thus, by \ref{itm:W3'} again,  
\begin{align}\label{eq:2308b}
r_j\geq \frac{1}{2}r_0\geq \frac{c}{4}\varepsilon\delta.
\end{align}
Inequality \eqref{eq:2308b} in turn  implies a  parallel uniform lower bound for   the side-length $r_k$ of all cubes entering the definition of $\mathscr L^{2}$. Hence,   we obtain that  $|\nabla\phi_k|\lesssim \frac 1{\varepsilon \delta}$ for every $k \in A_j^{\mathbb M}$. 
Thus, 
 owing to properties
\ref{itm:W3} and \eqref{new35}, 
 \begin{align}\label{eq:1408'}
\|\mathscr L^{2}(\bfu)\|_{L^\infty(Q_j^\ast)}
\lesssim \|\widetilde{\bfu}_j\|_{L^\infty(Q_j^\ast)}
 \lesssim \|\bfu\|_{L^\infty(\widetilde Q_j^\ast)}\lesssim \|\bfu\|_{L^\infty(\Omega)}
\end{align}
for $j \in \mathbb M$ and $\bfu \in L^\infty (\Omega)$.
This implies that $\mathscr L^{2}:L^\infty(\Omega)\rightarrow L^\infty(\R^n)$. 
\\ The boundedness of  $\mathscr L^2:L^1(\Omega)\rightarrow L^1(\R^n)$  follows from the fact that, similarly to \eqref{eq:2708} and  owing to \eqref{eq:2308b},
\begin{align*}
\int_{\R^n}\bigg|\sum_{j\in\mathbb M}\chi_{Q_j^\ast\cap W_{\mathbb P}}\sum_{k \in A_j}^{\mathbb M} \nabla \phi_k \otimes^\sym \widetilde\bfu_j\bigg|\dx\lesssim \sum_{j\in\mathbb M}\int_{Q_j^\ast}|\widetilde\bfu_j|\dx \lesssim \int_\Omega|\bfu|\dx.
\end{align*}
Let us next establish property \eqref{new18}. Thanks to the estimates  proved above, it suffices to show that, if  $\bfu \in E^1L^1(\Omega)+E^1 L^\infty (\Omega)$, then
\begin{equation}\label{mar125}
\mathscr E_\Omega\bfu \in E^1_{\rm loc}L^1(\rn).
\end{equation}
 Property  \eqref{mar125} will follow if we  show that $\mathscr E_\Omega\bfu \in E^1L^1(B)$, where $B$ is a ball in each of the  following three cases: $\overline B \subset \Omega$, $\overline B \subset \rn \setminus \overline \Omega$, the center of $B$ belongs to $\partial \Omega$.
\\ The case when $\overline B \subset \Omega$ is trivial, since $\mathscr E_\Omega\bfu = \bfu$ in $\Omega$, and $\bfu \in E^1L^1(B)$, inasmuch as {\color{black} $\bfu  \in E^1_{\rm loc} L^1(\Omega)$}.
%
\\ If $\overline B \subset \rn \setminus \overline \Omega$, then the conclusion follows from the fact that, by \eqref{Eomega} and  \ref{itm:W3},  in $B$ the function $\mathscr E_\Omega\bfu$ agrees with a finite sum of products of affine functions times smooth functions.
\\ Assume now that the center of $B$ belongs to $\partial \Omega$. The properties of the cubes $\{Q_j\}$ and of their reflected cubes ensure that a function $\phi \in C^\infty_0(\rn)$ can be chosen with a  support so large  that $\phi=1$ in $B$ and $\mathscr E_\Omega (\phi \bfu) = \mathscr E_\Omega \bfu$ in $B$. Thus, on replacing $\bfu$ by $\phi \bfu$, it suffices to prove that $\mathscr E_\Omega\bfu \in E^1L^1(B)$ under the additional assumption that $\bfu$ has a bounded support. This piece of information, combined with the assumption that $ E^1L^1(\Omega)+E^1 L^\infty (\Omega)$,
ensures that, in fact, $\bfu  \in E^1L^1(\Omega)$. Assume that we already know that, if 
$\bfu\in C^\infty(\overline{\Omega})$, then $\mathscr E_\Omega \bfu\in  E^1L^1(\rn)$. Hence,  
by the estimates
for the operators $\mathscr L^1$ and $\mathscr L^2$ established above, there exists a constant $c=c(\Omega)$ such that
\begin{equation}\label{mar126}
\|\mathscr E_\Omega \bfu\|_{E^1L^1(\rn)} \leq c \| \bfu\|_{E^1L^1(\Omega)}.
\end{equation}
Owing to Theorem \ref{thm:smooth}, there exists a sequence $\{\bfu_k\}\subset C^\infty(\rn)$ such that  $\bfu_k \to \bfu$ in $E^1L^1(\Omega)$,  where $\bfu_k$ still denotes   the restriction of $\bfu_k$ to $\Omega$. Moreover, since $\bfu$ has bounded support, the functions $\bfu_k$, defined according to equation \eqref{mar130}, also have (uniformly) bounded supports.
From inequality  \eqref{mar126} one has that $\mathscr E_\Omega \bfu_k$ is a Cauchy sequence in $E^1L^1(\rn)$, and hence converges to some function $\overline \bfu \in E^1L^1(\rn)$. Since $\mathscr E_\Omega \bfu_k = \bfu_k$ and  $\mathscr E_\Omega \bfu = \bfu$ in $\Omega$, we have that  $\mathscr E_\Omega \bfu_k \to \mathscr E_\Omega\bfu$ in $L^1(\Omega)$. Moreover, owing to equation \eqref{eq:2708} applied to $\bfu_k - \bfu$, we have that  $\mathscr E_\Omega \bfu_k \to \mathscr E_\Omega\bfu$ in $L^1(\rn \setminus \Omega)$. Therefore,  $\mathscr E_\Omega \bfu_k \to \mathscr E_\Omega\bfu$ in $L^1(\rn)$, whence $\mathscr E_\Omega\bfu = \overline{\bfu} \in  E^1L^1(\rn)$.
\\ It thus remains to show that if $\bfu\in C^\infty(\overline{\Omega})$ and has bounded support, then $\mathscr E_\Omega \bfu\in  E^1L^1(\rn)$. 
Since, as observed above,  the function $\mathscr E_\Omega \bfu$ has also a bounded support in $\rn$, it suffices to prove that 
\begin{equation}\label{mar131}
\text{$\mathscr E_\Omega 	\bfu$ is Lipschitz continuous in $\rn$.}
\end{equation}
Equation \eqref{mar131} will in turn follow if we show that
\begin{equation}\label{mar132}
\text{$\mathscr E_\Omega 	\bfu$ is Lipschitz continuous in $\Omega$,}
\end{equation}
\begin{equation}\label{mar133}
\text{$\mathscr E_\Omega 	\bfu$ is Lipschitz continuous in $\rn \setminus \overline \Omega$,}
\end{equation}
\begin{equation}\label{mar134}
\text{$\mathscr E_\Omega 	\bfu$ is  continuous in $\rn$.}
\end{equation}
Property \eqref{mar132} holds since  $\mathscr E_\Omega \bfu = \bfu$ in $\Omega$ and $\bfu \in C^{\infty}(\overline \Omega)$ and has a bounded support, and hence $\mathscr E_\Omega \bfu$ agrees, in $\Omega$, with the restriction of a Lipschitz continuous function. 
\\ Let us now focus on \eqref{mar133}. 
%
%
%
Similarly to equation \eqref{eq:chain}, owing to  \ref{itm:U2} one has that
\begin{align}\label{mar135}
    \nabla \mathscr E_\Omega\bfu 
    &= \sum_{k \in A_j^{\mathbb M}} \nabla \phi_k\otimes \widetilde\bfu_k+\sum_{k \in A_j^{\mathbb M}}  \phi_k \nabla\widetilde\bfu_k 
   \\ \nonumber
   &
    = \sum_{k \in A_j^{\mathbb M}} \nabla \phi_k \otimes (\widetilde\bfu_k - \widetilde\bfu_j)+\sum_{k \in A_j^{\mathbb M}}  \phi_k \nabla\widetilde\bfu_k
 + \chi_{Q_j^\ast\setminus \mathcal W_{\mathbb P}} \sum_{k \in A_j^{\mathbb M}} \nabla \phi_k \otimes \widetilde \bfu _j \quad
 \text{in $Q_j^\ast$.}
    \end{align}
One can estimate the $L^\infty (Q_j^\ast)$ norm  of the first and last  term on the rightmost side of equation \eqref{mar135}  via analogous arguments as in the proof of the $L^\infty$ estimates for the operators  $\mathscr L^1$ and $\mathscr L^2$ above,  after replacing $\otimes^{sym}$  by $\otimes$.
As for the middle term,  by \eqref{eq:extgrad} and \ref{itm:W3}, one has that
    \begin{align*}
        \bigg\|\sum_{k \in A_j^{\mathbb M}}  \phi_k \nabla\widetilde\bfu_k\bigg\|_{L^\infty(Q_j^\ast)}&\lesssim \sum_{k\in A_j^{\mathbb M}}\|\nabla\widetilde\bfu_k\|_{L^\infty(Q_k^\ast)}\lesssim \sum_{k\in A_j^{\mathbb M}}\|\nabla\widetilde{\mathcal R}_k\|_{L^\infty(\widetilde Q_k^\ast)}
\\
  &
=  \sum_{k\in A_j^{\mathbb M}}|\widetilde{\bfR}_k|\lesssim  \sum_{k\in A_j^{\mathbb M}}\|\nabla\bfu\|_{L^\infty(\widetilde Q_k^\ast)}\lesssim \|\nabla\bfu\|_{L^\infty(\Omega)}.
    \end{align*}
 This shows that
\begin{align}\label{sep10}
\mathscr E_\Omega\bfu\in W^{1,\infty}(\R^n\setminus \overline \Omega).
\end{align}
 The next step consists in proving that property \eqref{sep10} implies \eqref{mar133}.
To this purpose, we first show that
\begin{align}\label{eq:2008b}
    \|\mathscr E_\Omega\bfu-\bfu_{\widetilde Q_{j_0}}\|_{L^\infty(Q_{j_0})}\leq\,c\,r_{j_0}\|\nabla\bfu\|_{L^\infty(\Omega)}
\end{align}
for any cube $Q_{j_0}\in \{Q_j\}_{j\in\mathbb M}$. Here $\widetilde Q_{j_0}$ denotes the reflected cube of $Q_{j_0}$ and $\bfu_{\widetilde Q_{j_0}}$ the mean value of $\bfu$ over $\widetilde Q_{j_0}$.
In order to prove \eqref{eq:2008b}, observe that, by  \ref{itm:U2},
\begin{align}\label{mar140}
\mathscr E_\Omega\bfu-\bfu_{\widetilde Q_{j_0}}=\sum_{k\in A_{j_0}^{\mathbb M}}\phi_k(\widetilde\bfu_k-\widetilde\bfu_{j_0})+\widetilde\bfu_{j_0}-\bfu_{\widetilde Q_{j_0}} \quad \text{in $Q_{j_0}$.}
\end{align}
Here $A_{j_0}^{\mathbb M}$ is defined as in  \eqref{ajm}.
By the definition of  $\widetilde\bfu_k$ and equations \eqref{norm-equiv4} and \eqref{norm-equiv3}, one has that
%
\begin{align}\label{mar141}
    \norm{\widetilde\bfu_k-
        \widetilde\bfu_{j_0}}_{L^\infty(Q_{j_0})}\lesssim \norm{\widetilde{\mathcal R}_ k(\bfu)-
        \widetilde{\mathcal R}_{j_0}(\bfu)}_{L^\infty( Q_{j_0}^*)}\lesssim \dashint_{\tilde Q_{j_0}^*} \abs{\widetilde{\mathcal R}_ k(\bfu)-
        \widetilde{\mathcal R}_{j_0}(\bfu)}\dx.
\end{align}
Since $\widetilde Q_k$ and $\widetilde Q_{j_0}$ are not necessarily  neighbours, we argue as in \eqref{eq:chain} and connect them via a chain of cubes $\widetilde Q_{j_1},\dots,\widetilde Q_{j_m}$, with $\widetilde Q_{j_1}=\widetilde Q_{j_0}$ and $\widetilde Q_{j_m}=\widetilde Q_{k}$, such that $m=m_{{j_0},k}$ is uniformly bounded independently of ${j_0}$ and $k$,  the cubes $\widetilde Q_{j_h}$, with $h=1, \dots, m$,  have side-lengths equivalent to $r_{j_0}$, and the measure of their intersections is equivalent  to that of $\widetilde Q_{j_0}$.  We infer that
 \begin{align}\label{mar137}
    \dashint_{\widetilde Q_{j_0}^*} \abs{\widetilde{\mathcal R}_ k(\bfu)-
        \widetilde{\mathcal R}_{j_0}(\bfu)}\dx\leq\,\sum_{h=2}^{m}\dashint_{\widetilde Q_{j_0}^\ast} \abs{\widetilde{\mathcal R}_ {j_h}(\bfu)-
        \widetilde{\mathcal R}_{j_{h-1}}(\bfu)}\dx\lesssim \sum_{h=2}^{m}\dashint_{\widetilde Q_{j_{h}}^*} \abs{\widetilde{\mathcal R}_ {j_h}(\bfu)-
        \widetilde{\mathcal R}_{j_{h-1}}(\bfu)}\dx.
\end{align}
Observe that the last inequality holds owing to equation  \eqref{norm-equiv4},
%
since $|\widetilde Q_{j_{h}}^*|\approx |\widetilde Q_{j_0}^*|$ for $h=1, \dots, m$.
 Thanks to equation \eqref{norm-equiv4} and inequality \eqref{itm:meanw1},  
 \begin{align}\label{mar138}
 \dashint_{\widetilde Q_{j_{h}}^*} \abs{\widetilde{\mathcal R}_ {j_h}(\bfu)-
        \widetilde{\mathcal R}_{j_{h-1}}(\bfu)}\dx &
 \lesssim
 \dashint_{\widetilde Q_{j_{h-1}}^* \cap \widetilde Q_{j_{h}}^*} \abs{\widetilde{\mathcal R}_ {j_h}(\bfu)-
        \widetilde{\mathcal R}_{j_{h-1}}(\bfu)}\dx
\\ \nonumber &
\lesssim \dashint_{\widetilde Q_{j_{h}}^*}\abs{\bfu - \widetilde{\mathcal R}_{j_{h}}(\bfu)}\dx +
    \dashint_{\widetilde Q_{j_{h-1}}^*} \abs{\bfu -
        \widetilde{\mathcal R}_{j_{h-1}}(\bfu)}\dx
\\ \nonumber
        &\lesssim r_{j_h}\, \dashint_{\widetilde Q_{j_{h}}^*}\abs{\ep(\bfu)}\dx +
    r_{j_{h-1}}\, \dashint_{\widetilde Q_{j_{h-1}}^*} \abs{\ep(\bfu)}\dx
\\ \nonumber &
\lesssim \,r_{j_0}\|\ep(\bfu)\|_{L^\infty(\Omega)} \lesssim r_{j_0}\|\nabla \bfu\|_{L^\infty(\Omega)}
  \end{align}
 for $h=1, \dots, m$. Notice that  we have also made use of the equivalence $r_{j_h}\approx r_{j_0}$ for these values of $h$.
  Finally, on exploiting equations \eqref{norm-equiv4} and \eqref{norm-equiv3} again, we deduce from inequality  \eqref{itm:meanw1} and a standard Poincar\'e inequality  that
   \begin{align}\label{mar139}
      \|\widetilde\bfu_{j_0}-\bfu_{\widetilde Q_{j_0}}\|_{L^\infty(Q_{j_0})}&\lesssim \norm{\widetilde{\mathcal R}_ {j_0}(\bfu) -
        \bfu_{\widetilde Q_{j_0}}}_{L^\infty(\widetilde Q_{j_0})}
\lesssim  \dashint_{\widetilde Q_{j_0}^*} |\widetilde{\mathcal R}_ {j_0}(\bfu) -
        \bfu_{\widetilde Q_{j_0}}|\dx 
\\  \nonumber & \lesssim  \dashint_{\widetilde Q_{j_0}^*} |\widetilde{\mathcal R}_ {j_0}(\bfu) -
        \bfu|\dx 
+ \dashint_{\widetilde Q_{j_0}^*} | \bfu -
        \bfu_{\widetilde Q_{j_0}}|\dx 
\\  \nonumber &
\lesssim  \dashint_{\widetilde Q_{j_0}^*} |
       \ep(\bfu)|\dx 
+ \dashint_{\widetilde Q_{j_0}^*} | \nabla \bfu 
        |\dx 
\lesssim \,r_{j_0}\|\nabla\bfu\|_{L^\infty(\Omega)}.
  \end{align}
%
%
Inequality  \eqref{eq:2008b} follows from inequalities \eqref{mar141}--\eqref{mar139}.
\\ We are now in a position to accomplish the proof of property \eqref{mar133}. 
Assume that $x,y\in\R^n\setminus \overline\Omega$. We shall 
  prove that
\begin{align}\label{eq:0409}
    |\mathscr E_\Omega\bfu(x)-\mathscr E_\Omega\bfu(y)|\leq\,c\,|x-y|\|\nabla\bfu\|_{L^\infty(\Omega)}
\end{align}
for some constant $c$ independent of $x$ and $y$.
If   either $|x-y|<\mathrm{dist}(x,\partial\Omega)$,  or $|x-y|<\mathrm{dist}(y,\partial\Omega)$, then equation \eqref{eq:0409} trivially holds, since the points $x,y$ can be connected by a straight line in $\R^n\setminus\Omega$.   Consequently, we can assume that 
\begin{align}\label{eq:2209}
    |x-y|>\max\{\mathrm{dist}(x,\partial\Omega),\mathrm{dist}(y,\partial\Omega)\}.
\end{align}
We can also assume that  $|x-y|\leq c\varepsilon\delta$, for some constant $c$, otherwise \eqref{eq:0409} is trivial again.
If this constant $c$ is chosen sufficiently small,  depending only the  constant appearing in equation \eqref{red}, it follows from \eqref{eq:2209} that
there exist cubes $Q_x,Q_y\subset \{Q_j\}_{j\in\mathbb M}$, whose  side-lengths are bounded by $|x-y|$, containing $x$ and $y$. We consider the reflected cubes $\widetilde Q_x$ and $\widetilde Q_y$, containing points $\widetilde x$ and $\widetilde y$ respectively. Note that the side-lengths of $\widetilde Q_x$ and $\widetilde Q_y$ and the distances between each cube and its reflected cube are bounded by $|x-y|$ as well. Thereby,
\begin{align*}
    |\widetilde x-\widetilde y|\leq|\widetilde x-x|+|x-y|+|y-\widetilde y|\lesssim |x-y|.
\end{align*}
Hence, owing to  \eqref{eq:2008b} and to the Lipschitz continuity of $\bfu$ in $\Omega$, there exist constants $c$ and $c'$, independent of $\bfu$, such that
\begin{align*}
   |\mathscr E_\Omega\bfu(x)-\mathscr E_\Omega\bfu(y)|&\leq |\mathscr E_\Omega\bfu(x)-\bfu_{\widetilde Q_x}|+|\bfu_{\widetilde Q_x}-\bfu_{\widetilde Q_y}|+|\bfu_{\widetilde Q_y}-\mathscr E_\Omega\bfu(y)|\\
   &\leq\,c\,\big(r(\widetilde Q_x)+r(\widetilde Q_y)\big)\|\nabla\bfu\|_{L^\infty(\Omega)}+c\|\nabla\bfu\|_{L^\infty(\Omega)}\big(|\widetilde x-\widetilde y|+r(\widetilde Q_x)+r(\widetilde Q_y)\big)\\
   &\leq\,c\,|x-y|\|\nabla\bfu\|_{L^\infty(\Omega)}.
\end{align*}
Hence, inequality \eqref{eq:0409} follows. Property \eqref{mar133} is thus established.
\\ Finally, we are going to prove that inequality
\eqref{eq:2008b} implies Lipschitz continuity of $\mathscr E_\Omega\bfu$ across $\partial\Omega$,  and hence \eqref{mar134}. Let $x\in\R^n\setminus\Omega$ and $y\in\Omega$. We may assume that 
$|x-y|<c\varepsilon\delta$ for some constant $c=c(n)$, the case when $|x-y|\geq\,c\varepsilon\delta$ being trivial. Hence, in particular, $$\mathrm{dist}(x,\partial\Omega)\leq|x-y|\leq\, c\varepsilon\delta.$$
The constant $c$  can be chosen, depending on the constant appearing in inequality \eqref{red}, in such a way that there exists a cube $Q_x\in\{Q_j\}_{j\in\mathbb M}$ with $x\in Q_x$.  Consequently, by property     \ref{itm:Wbnd}, one has that $r(Q_x)\lesssim |x-y|$. Now,   
\begin{align*}
    \mathscr E_\Omega\bfu(x)-\mathscr E_\Omega\bfu(y)=\mathscr E_\Omega\bfu(x)-\bfu_{\widetilde Q_x}+\bfu_{\widetilde Q_x}-\bfu(y),
\end{align*}
where $\widetilde Q_x$ is the reflected cube of $Q_x$.
The absolute value of the first difference on the right-hand side of this equality can be bounded via inequality \eqref{eq:2008b}, with $Q_{j_0}$ replaced by $Q_x$. As for the second difference, one can  make use of property \eqref{mar132}, combined   with the inequality
\begin{align*}
    \mathrm{dist}(y,\widetilde Q_x)\leq\,|y-x|+\mathrm{dist}(x,\widetilde Q_x)\lesssim |x-y|.
\end{align*}
Note that the last inequality relies upon property \eqref{itm:R2}  and on the fact that $r(Q_x)\lesssim |x-y|$. Altogether, this shows that inequality \eqref{eq:0409} also holds if $x\in\R^n\setminus\Omega$ and $y\in\Omega$, thus establishing property \eqref{mar134}. 
\\ We conclude with a proof of inequality \eqref{ext*}. Thanks to properties \eqref{ext0} and \eqref{ALT}, to this purpose it suffices to show that there exists a constant $c=c(\Omega)$ such that
\begin{align}\label{new15}
\int_{\rn \setminus \overline \Omega} A\big(|\mathscr E_\Omega \bfu|+|\ep(\mathscr E_\Omega \bfu)|)\dx 
 \leq \int_\Omega A(c(|\bfu|+|\ep(\bfu)|)\dx
\end{align}
for every Young function $A$.
 \\ To begin with, one has that
\begin{align}\label{new8}
\int_{\rn \setminus \overline \Omega} A\big(|\mathscr E_\Omega \bfu|)\, dx & = 
  \int_{\R^n}A\Big(\Big|\sum_{j\in\mathbb M} \phi_j \widetilde \bfu_j\Big|\Big)\dx\leq \sum_{j\in\mathbb M}\int_{Q_j^\ast}A(|\widetilde \bfu_j|)\dx=\sum_{j\in\mathbb M}\int_{Q_j^\ast}A(|\widetilde{\mathcal{R}}_j(\bfu)|)\dx
\\
\nonumber
&  
\leq \sum_{j\in\mathbb M}\int_{\widetilde Q_j^\ast}A(c|\bfu|)\dx\leq \int_\Omega A(c'|\bfu|)\dx,
\end{align}
for suitable constants $c$ and $c'$ depending on $n$ and $\Omega$ and for every function $\bfu \in \mathcal M(\Omega)$ making the integral on the rightmost  side of  chain \eqref{new8} finite.
 Notice that we have made use of the convexity of $A$ in the first inequality,  of  equations 
\eqref{new35} and \eqref{new20} in the second one,
and of 
  property   \ref{itm:W3} in the last one.
\\ Our next task is to show that 
\begin{align}\label{new11}
\int_{\rn \setminus \overline \Omega} A\big(|\ep(\mathscr E_\Omega \bfu)|)\dx 
 \leq \int_\Omega A(c(|\bfu|+|\ep(\bfu|))\dx.
\end{align}
We have that
 \begin{align}\label{new10}
\int_{\rn}A\big(|\mathscr L^{1}(\bfu)|\big)\dx & =
\int_{\rn}A\bigg(\bigg|\sum_{j\in\mathbb M}\chi_{Q_j^\ast\cap\mathcal W_{\mathbb P}}\sum_{k\in A_j^{\mathbb M}}\sum_{h=2}^{m_{jk}}\nabla \phi_k \otimes^\sym\big(\widetilde{\mathcal{R}}_{j_h}(\bfu)- \widetilde {\mathcal R}_{j_{h-1}}(\bfu)\big)    \bigg|\bigg)\dx
 \\  \nonumber &\lesssim \sum_{j\in\mathbb M}\sum_{k\in A_j^{\mathbb M}}\sum_{h=2}^{m_{jk}}   \int_{Q_{j_h}^\ast}A\big(cr_k^{-1}|
\widetilde{\mathcal{R}}_{j_h}(\bfu) - \widetilde {\mathcal R}_{j_{h-1}}(\bfu)|\big)\dx
\\ \nonumber
&\lesssim \sum_{j\in\mathbb M} \sum_{k\in A_j^{\mathbb M}}\sum_{h=2}^{m_{jk}}   \int_{\tilde Q_{j_h}^\ast\cap \tilde Q_{j_{h-1}}^\ast}A\big(cr_k^{-1}|
\widetilde{\mathcal{R}}_{j_h}(\bfu) - \widetilde {\mathcal R}_{j_{h-1}}(\bfu)|\big)\dx
\\ \nonumber
&= \sum_{j\in\mathbb M}\sum_{k\in A_j^{\mathbb M}}\sum_{h=2}^{m_{jk}}    \int_{\tilde Q_{j_h}^\ast\cap \tilde Q_{j_{h-1}}^\ast}A\big(c'r_k^{-1}|
\widetilde{\mathscr{L}}_{j_h}(\ep(\bfu)) - \widetilde {\mathscr L}_{j_{h-1}}(\ep(\bfu)|\big)\dx
\\  \nonumber
&\lesssim \sum_{j\in\mathbb M}\sum_{k\in A_j^{\mathbb M}}\sum_{h=2}^{m_{jk}} \bigg(  \int_{\tilde Q_{j_h}^\ast}A\big(c'|
\ep(\bfu)|) \dx + \int_{\tilde Q_{j_{h-1}}^\ast}A\big(c'|
\ep(\bfu)|) \dx\bigg)
\\ \nonumber
&\lesssim  \int_{\Omega}A\big(c'|
\ep(\bfu)|) \dx,
\end{align}
for suitable constants $c$ and $c'$ depending on $n$ and $\Omega$. 
Note that  the first inequality holds since $|\nabla \phi_k|\lesssim r_k^{-1}$ and  all sums   are finite, with a uniform upper bound for the number of addends, the second
 inequality holds by properties \eqref{norm-equiv4} and \eqref{new20},  the third one by \eqref{eq:extE2} and \eqref{new20}, 
 and the last one 
by the the finiteness of  the sums   again.
\\ On the other hand, thanks to inequality \eqref{eq:2308b}, one deduces analogously to \eqref{new8} that
\begin{align}\label{new45}
\int_{\rn}A\big(|\mathscr L^{2}(\bfu)|\big)\dx & =
\int_{\R^n}A\bigg(\bigg|\sum_{j\in\mathbb M}\chi_{Q_j^\ast\cap W_{\mathbb P}}\sum_{k \in A_j}^{\mathbb M} \nabla \phi_k \otimes^\sym \widetilde\bfu_j\bigg|\bigg)\dx\\ \nonumber & \lesssim \sum_{j\in\mathbb M}\int_{Q_j^\ast}A(c|\widetilde\bfu_j|)\dx \lesssim \int_\Omega A(c' |\bfu|)\dx
\end{align}
 for some constants $c$ and $c'$ depending on $n$ and $\Omega$. 
Inequality \eqref{new15} follows from \eqref{new10}and \eqref{new45}.

\end{proof}

\section{The $K$-functional for symmetric gradient Sobolev spaces}\label{Kfunct}

The notion of $K$-functional is fundamental in the theory of real interpolation of normed spaces. Loosely speaking, knowledge of the  $K$-functional for two couples of spaces enables one to deduce the boundedness of any sublinear operator between any pair of suitably related intermediate spaces from the boundedness of the relevant operator between the endpoint spaces.  Recall that, given a couple of normed spaces $(Z_0, Z_1)$, which are both continuously embedded into some Hausdorff vector space, the $K$-functional is defined for each $\zeta \in Z_0 + Z_1$ and $t>0$ as 
$$
K(\zeta, t; Z_0, Z_1) = \inf _{\begin{tiny} \begin{array}{c}
                       {\zeta = \zeta_0 + \zeta_1}\\
                        {\zeta_0 \in Z_0, \, \zeta_1 \in Z_1}
                      \end{array}
                      \end{tiny}} \big(\|\zeta_0\|_{Z_0} + t \|\zeta_1\|_{Z_1}\big).
$$
The $K$-functional for the couple of classical  $k$-th order Sobolev spaces $(W^{k,1}(\rn), W^{k,\infty}(\rn))$, with $k\in \setN$, has been computed (up to equivalence) in the paper \cite{DeSc}. In particular, \cite[Theorem 1]{DeSc} tells us that 
\begin{equation}\label{KW} K(\bfu, t; W^{1,1}(\rn), W^{1,\infty}(\rn)) \approx \int_0^t (\mathcal D\bfu)^*(s)\, \ds \quad \hbox{for $t>0$,}
\end{equation}
with equivalence constants depending on $n$, where $\mathcal D\bfu$ is defined by \eqref{Du}. A parallel formula is established in the same paper for the couple $(W^{k,1}(\Omega), W^{k,\infty}(\Omega))$, where $\Omega$ is any  open set in $\rn$  with a
minimally smooth boundary in the sense of  \cite[Chapter 6, Section 4]{stein}.
\par A version of equation \eqref{KW} for symmetric gradient Sobolev spaces in $\rn$, and a counterpart in any $(\varepsilon, \delta)$-domain $\Omega$, are the content of the following result.

\begin{theorem}\label{K-funct}{\rm{\bf [$K$-functional for the couple $(E^1L^1,E^{1} L^\infty)$]}}
\\ (i) One has that
\begin{align}\label{nov1}
K(\bfu,t; E^1L^1(\R^n),E^{1} L^\infty(\R^n))\approx \int _0^t \mathcal E(\bfu)^*(s) \ds \quad \text{for $t>0$,}
\end{align}
for every $\bfu \in E^1L^1(\R^n) + E^{1} L^\infty(\R^n)$, with equivalence constants depending on $n$.
\\ 
(ii) One has that
\begin{align}\label{nov2'}
K(\bfu,t; E^1_b L^1(\R^n), E^1_b  L^\infty(\R^n))\approx \int _0^t \ep(\bfu)^*(s) \ds \quad \text{for $t>0$,}
\end{align}
for every $\bfu \in E^1_bL^1(\R^n) + E^{1}_bL^\infty(\R^n)=E^1_bL^1(\R^n)$, and 
\begin{align}\label{nov2''}
K(\bfu,t; E^1_0 L^1(\R^n), E^1_0  L^\infty(\R^n))\approx \int _0^t \ep(\bfu)^*(s) \ds \quad \text{for $t>0$,}
\end{align}
for every $\bfu \in E^1_0L^1(\R^n) + E^{1}_0L^\infty(\R^n)$,  with equivalence constants depending on $n$.
\\ (iii) Assume that
 $\Omega$ is an 
$(\varepsilon, \delta)$-domain.  Then 
\begin{align}\label{nov3}
K(\bfu,t;E^{1}L^1(\Omega),E^{1}L^\infty(\Omega))\approx \int _0^t \mathcal E(\bfu)^*(s) \ds \quad \text{for $t>0$,}
\end{align}
for every $\bfu \in E^{1}L^1(\Omega)+E^{1}L^\infty(\Omega)$, with equivalence constants depending on $\Omega$.

\end{theorem}

Theorem \ref{K-funct} is the main  step in the proof of formulas for the $K$-functional of the couples $(E^{1}L^1(\R^n),E^{1}L^{n,1}(\R^n))$, $({E^1_b}L^1(\R^n), {E^1_b}L^{n,1}(\R^n))$ and $(E^{1}L^1(\Omega),E^{1}L^{n,1}(\Omega))$. These formulas play a decisive role in our embedding theorems for symmetric gradient Sobolev spaces built upon arbitrary rearrangement-invariant spaces.

\begin{theorem}\label{corollary}{\rm{\bf [$K$-functional for the couple $(E^1L^1,E^{1}L^{n,1})$]}}
\\
(i)  One has that
\begin{align}\label{feb5}
K(\bfu,t;E^{1}L^1(\R^n),E^{1}L^{n,1}(\R^n))\approx \int _0^{t^{n'}} \mathcal E(\bfu)^*(s) \ds + t \int_{t^{n'}}^\infty  \mathcal E(\bfu)^*(s) s^{-\frac 1{n'}} \ds  \quad \hbox{for $t>0$,}
\end{align}
for every $\bfu \in E^{1}L^1(\R^n) + E^{1}L^{n,1}(\R^n)$, with equivalence constants depending on $n$.
\\ (ii) One has that
\begin{align}\label{feb6}
K(\bfu,t; {E^1_b}L^1(\R^n), {E^1_b}L^{n,1}(\R^n))\approx \int _0^{t^{n'}} \ep(\bfu)^*(s) \ds + t \int_{t^{n'}}^\infty  \ep(\bfu)^*(s) s^{-\frac 1{n'}} \ds  \quad \hbox{for $t>0$,}
\end{align}
for every $\bfu \in E^1_bL^1(\R^n) + E^{1}_bL^{n,1}(\R^n){\color{black} = E^1_bL^1(\R^n)}$, and
\begin{align}\label{feb6'}
K(\bfu,t; {E^1_0}L^1(\R^n), {E^1_0}L^{n,1}(\R^n))\approx \int _0^{t^{n'}} \ep(\bfu)^*(s) \ds + t \int_{t^{n'}}^\infty  \ep(\bfu)^*(s) s^{-\frac 1{n'}} \ds \quad \hbox{for $t>0$,}
\end{align}
for every $\bfu \in E^1_0L^1(\R^n) + E^{1}_0L^{n,1}(\R^n)$, with equivalence constants depending on $n$.
\\ (iii) Assume that
 $\Omega$ is an 
$(\varepsilon, \delta)$-domain. Then 
\begin{align}\label{feb7}
K(\bfu,t;E^{1}L^1(\Omega),E^{1}L^{n,1}(\Omega))\approx \int _0^{t^{n'}} \mathcal E(\bfu)^*(s) \ds + t \int_{t^{n'}}^\infty  \mathcal E(\bfu)^*(s) s^{-\frac 1{n'}} \ds  \quad \hbox{for $t>0$,}
\end{align}
for every $\bfu \in E^{1}L^1(\Omega)+E^{1}L^{n,1}(\Omega)$, with equivalence constants depending on $\Omega$.
\end{theorem}
  
We premise to the proofs of Theorems \ref{K-funct} and \ref{corollary} a lemma on the coincidence between the $K$-functional of a couple of rearrangement-invariant spaces  $(X(\Omega, \mathbb R^m), Y(\Omega, \mathbb R^m))$  of $\mathbb R^m$-valued functions, and the $K$-functional of the couple $(X(\Omega), Y(\Omega))$ of scalar-valued functions,  built upon the same rearrangement-invariant function norms, and evaluated at the modulus of the $\mathbb R^m$-valued functions. Of course, the same result holds for spaces of matrix-valued functions. This result will be exploited without explicitly mentioning it in what follows.

\begin{lemma}\label{K-vector}
Let $\Omega$ be a measurable set in $\rn$ and let $\|\cdot \|_{X(0,|\Omega|)}$ and  $\|\cdot \|_{Y(0,|\Omega|)}$ be rearrangement-invariant function norms.   Let   $m\in \N$. 
Then 
\begin{equation}\label{K-v1}
K(|\bfU|, t; X(\Omega), Y(\Omega)) =K(\bfU, t; X(\Omega, \mathbb R^m), Y(\Omega, \mathbb R^m)) 
\quad \text{for $t>0$.}
\end{equation}
\end{lemma}
\begin{proof}
We have that
\begin{align*}
K(\bfU, t; X(\Omega, \mathbb R^m), Y(\Omega, \mathbb R^m))  & = \inf _{\bfU = \bfU_0 + \bfU_1}\big(\|\bfU_0\|_{X(\Omega, \mathbb R^m)} + t \|\bfU_1\|_{Y(\Omega, \mathbb R^m)}\big)
  \\ \nonumber & =  \inf _{\bfU = \bfU _0 + \bfU _1}\big(  \||\bfU _0|\|_{X(\Omega)} + t   \||\bfU_1|\|_{Y(\Omega)}\big)
\leq  \inf _{|\bfU| = v _0 + v_1}\big(\|v _0\|_{X(\Omega)} + t   \|v _1|\|_{Y(\Omega)}\big)
\\ \nonumber &
 =  K(|\bfU|, t; X(\Omega), Y(\Omega)).
\end{align*}
Note that the  inequality  holds since, if $|\bfU| = v_0 + v_1$ for some $v_0$ and $v_1$, then 
$\bfU = \bfU _0 + \bfU _1$ with $\bfU _0 = \frac{\bfU}{|\bfU|}v_0$ and  $\bfU_1= \frac{\bfU}{|\bfU|}v_1$. 
The reverse inequality holds since 
\begin{align*}
K(\bfU, t; X(\Omega, \mathbb R^m), Y(\Omega, \mathbb R^m))  
& = \inf _{\bfU = \bfU_0 + \bfU_1}\{\|\bfU_0\|_{X(\Omega, \mathbb R^m)} + t \|\bfU_1\|_{Y(\Omega, \mathbb R^m)}\}  
 \\ \nonumber & = \inf _{\bfU = \bfU _0 + \bfU _1}\{  \||\bfU _0|\|_{X(\Omega)} + t   \||\bfU_1|\|_{Y(\Omega)}\} 
\geq \inf _
{\begin{tiny} \begin{array}{c}
                       {|\bfU|\leq  v_0 + v_1}\\
                        { v_0, v_1 \geq0}
                      \end{array}
                      \end{tiny}}
\{  \|v_0\|_{X(\Omega)} + t   \|v_1\|_{Y(\Omega)}\} 
\\ \nonumber & \geq \inf _{|\bfU|=  v_0 + v_1}\{  \|v_0\|_{X(\Omega)} + t   \|v_1\|_{Y(\Omega)}\} 
= K(|\bfU|, t; X(\Omega), Y(\Omega)).
\end{align*}
Observe that the second inequality holds since, given nonnegative functions $v_0$ and $v_1$ such that $|\bfU|\leq  v_0 + v_1$, the functions $\widehat v_0$ and $\widehat v_1$, defined as
$$\widehat v_0 = \frac{v_0 |\bfU|}{v_0 + v_1} \quad \text{and} \quad \widehat v_1 = \frac{v_1 |\bfU|}{v_0 + v_1},$$
enjoy the following properties:
$$|\bfU| = \widehat v_0 + \widehat v_1,$$
$$0 \leq \widehat v_0 \leq v_0, \quad 0\leq \widehat v_ 1 \leq v_1.$$
Hence, by the monotonicity of rearrangement-invariant function norms with respect to pointwise inequalities of functions,
$$\|v_0\|_{X(\Omega)} + t   \|v_1\|_{Y(\Omega)} \geq \|\widehat v_0\|_{X(\Omega)} + t   \|\widehat v_1\|_{Y(\Omega)} \quad \text{for $t >0$.}$$
The proof is complete
\end{proof}

Our proof of Theorem \ref{K-funct} relies upon the use of maximal functions, in the spirit of \cite[Thm. 3.1]{CM}, and of the truncation operators introduced in Section \ref{truncation}. The proof of Part (iii) makes use of Part (i) and of the extension results of Section \ref{extension}.

\begin{proof}[Proof of Theorem \ref{K-funct}]
Part (i).  The inequality ``$\gtrsim$'' in \eqref{nov1}  is a  consequence of the formula
\begin{equation}\label{apr30}
K(v,t;L^{1}(\rn),L^{\infty}(\rn))\approx \int_0^t v^\ast(s)\ds  \quad \text{for $t >0$,}
\end{equation}
for $v \in L^1(\rn) + L^\infty(\rn)$ --
see e.g. \cite[Theorem 1.6, Chapter 5]{BS}. Here, and throughout the proof of Part (i), the equivalence constants depend only on $n$.
 Indeed, given $t>0$, we have that
\begin{align}\label{jan22}
&K(\bfu,t;E^1L^1(\rn),E^1L^\infty(\rn))=\inf_{\bfu=\bfu_0+\bfu_1}\big(\|\bfu_0\|_{E^1L^1(\rn)}+t\|\bfu_1\|_{E^1L^\infty(\rn)}\big)
\\  \nonumber
&\geq\inf_ 
{\begin{tiny} \begin{array}{c}
                       {\bfu=\bfu_0+\bfu_1 }\\
                        {\bfu_0 \in E^1L^1, \bfu_1 \in E^1L^\infty }
                      \end{array}
                      \end{tiny}} \tfrac{1}{2}
\big(\|\bfu_0\|_{L^{1}(\rn)}+t\|\bfu_1\|_{L^{\infty}(\rn)}\big)+ \inf_{\begin{tiny} \begin{array}{c}
                       {\bfu=\bfu_0+\bfu_1 }\\
                        {\bfu_0 \in E^1L^1, \bfu_1 \in E^1L^\infty }
                      \end{array}
                      \end{tiny}} 
                      \tfrac{1}{2} \big(\|\ep(\bfu_0)\|_{L^{1}(\rn)}+t\|\ep(\bfu_1)\|_{L^{\infty}(\rn)}\big)
\\  \nonumber
&\geq \inf_{\begin{tiny} \begin{array}{c}
                       {\bfu=\bfv_0+\bfv_1 }\\
                        {\bfv_0 \in L^1, \bfv_1 \in L^\infty }
                      \end{array}
                      \end{tiny}}  \tfrac{1}{2}
\big(\|\bfv_0\|_{L^{1}(\rn)}+t\|\bfv_1\|_{L^{\infty}(\rn)}\big)+\inf
_{\begin{tiny} \begin{array}{c}
                       {\ep(\bfu)=\bfE_0+\bfE_1}\\
                        {\bfE_0 \in L^1, \bfE_1 \in L^\infty }
                      \end{array}
                      \end{tiny}}  \tfrac{1}{2}
\big(\|\bfE_0\|_{L^{1}(\rn)}+t\|\bfE_1\|_{L^{\infty}(\rn)}\big)
\\ \nonumber
&= \tfrac{1}{2}\,K(\bfu,t;L^{1}(\rn),L^{\infty}(\rn))+\tfrac{1}{2}\,K(\ep(\bfu),t;L^{1}(\rn),L^{\infty}(\rn)) \gtrsim \bigg(\int_0^t \bfu^\ast(s)\ds+\int_0^t \ep(\bfu)^\ast(s)\ds\bigg)
\\ \nonumber
& \geq  \int _0^t \mathcal E(\bfu)^*(s) \ds.
\end{align}
Note that the last inequality holds owing to property \eqref{subadd} of rearrangements.
 \\
Let us now prove the reverse inequality ``$\lesssim$''. 
 Assume, for the time being, that   $\bfu \in E^1L^1(\R^n)$.
Given any $\theta,\lambda>0$, let $\mathcal O_{\theta,\lambda}$ be the set in $\rn$ defined by \eqref{mar106}, and let $\bfu_{\theta,\lambda} =T^{\theta, \lambda}\bfu$ be the function given by  \eqref{eq:wlambda}. By  Theorem \ref{lem:Tlest}, Part (i),
 $\bfu_{\theta,\lambda}\in E^1L^\infty(\R^n)$, and 
%
\begin{align}\label{july110}
|\bfu_{\theta,\lambda}|\lesssim \theta \quad \text{and} \quad  |\ep(\bfu_{\theta,\lambda})|\lesssim\lambda.
\end{align}
Fix $t>0$. We choose $\theta=M(\bfu)^\ast(t)$, $\lambda=M(\ep(\bfu))^\ast(t)$ and set $\bfu_1=\bfu_{\theta,\lambda}$ and $\bfu_0=\bfu-\bfu_1$. Thus, 
\begin{align}\label{may1}
K(\bfu,t;E^1L^1(\rn),E^1L^\infty(\rn))\leq \|\bfu_0\|_{E^1L^1(\rn)}+t\|\bfu_1\|_{E^1L^\infty(\rn)}.
\end{align}
By equation \eqref{july110} and the choice of $\theta$ and $\lambda$,  
\begin{align}\label{a}
\|\bfu_1\|_{E^1L^\infty(\rn)}=\|\bfu_1\|_{L^{\infty}(\rn)}+\|\ep(\bfu_1)\|_{L^{\infty}(\rn)}\lesssim \theta+\,\lambda = M(\bfu)^\ast(t)+  M(\ep(\bfu))^\ast(t).
\end{align}
Next, since $\bfu_0=0$  in $\rn \setminus \mathcal O_{\theta,\lambda}$, 
\begin{align}\label{b}
\|\bfu_0\|_{E^1L^1(\rn)}\leq \|\bfu\chi_{\mathcal O_{\theta,\lambda}}\|_{L^{1}(\rn)}+\|\ep(\bfu)\chi_{\mathcal O_{\theta,\lambda}}\|_{L^{1}(\rn)}+\|\bfu_1\chi_{\mathcal O_{\theta,\lambda}}\|_{L^{1}(\rn)}+\|\ep(\bfu_1)\chi_{\mathcal O_{\theta,\lambda}}\|_{L^{1}(\rn)}.
\end{align}
Owing to the choice of $\theta$ and $\lambda$, one has that $|\mathcal O_{\theta,\lambda}|\leq c\,t$ for some constant $c=c(n)$. 
 Hence,
\begin{align}\label{c}
\|\bfu\chi_{\mathcal O_{\theta,\lambda}}\|_{L^{1}(\rn)}&  \leq\|\bfu^\ast\chi_{(0,ct)}\|_{L^{1}(0,\infty)}   \leq\,\max\{1,c\}\,\|\bfu^\ast\chi_{(0,t)}\|_{L^{1}(0, \infty)}=
\,\max\{1,c\}\,\int_0^t\bfu^\ast(s)\ds,
\end{align}
and similarly
\begin{align}\label{d}
\|\ep(\bfu)\chi_{\mathcal O_{\theta,\lambda}}\|_{L^{1}(\rn)}\leq \,\max\{1,c\}\,\int_0^t\ep(\bfu)^\ast(s)\ds.
\end{align}
Moreover,
\begin{align}\label{e}
\|\bfu_1\chi_{\mathcal O_{\theta,\lambda}}\|_{L^{1}(\rn)}+\|\ep(\bfu_1)\chi_{\mathcal O_{\theta,\lambda}}\|_{L^{1}(\rn)}&\lesssim t\,\theta+t\lambda =t\big(M(\bfu)^\ast(t)+M(\ep(\bfu))^\ast(t)\big).
\end{align}
On the other hand, by the second inequality in \eqref{riesz} and inequality \eqref{subadd},
\begin{align}\label{may2}
M(\bfu)^\ast(t)+M(\ep(\bfu))^\ast(t) \lesssim 
\frac 1t \,\bigg(\int_0^t\bfu^\ast(s)\ds+\int_0^t\ep(\bfu)^\ast(s)\ds\bigg)
\leq\, \frac 2t \,\bigg(\int_0^t\mathcal E(\bfu)^*(s)\, \ds\bigg).
\end{align}
Combining inequalities \eqref{may1}--\eqref{may2}  proves the inequality $\lesssim$ in \eqref{nov1}. This inequality is thus established under the assumption that $\bfu \in E^1L^1(\rn)$. It remains to remove this assumption. 
\\ To this purpose, given any function $\bfu \in E^1L^1(\R^n) + E^{1} L^\infty(\R^n)$, fix $t>0$ and set
$$M_0 = M(\mathcal E(\bfu))^*(t)$$
and 
$$\mathcal O = \{x \in \mathbb R^n:  M(\mathcal E(\bfu))> M_0\}.$$
Hence, 
\begin{equation}\label{jan5}
|\mathcal O|\leq t.
\end{equation}
Define the decreasing sequences $\{t _j\}$ by $t_j=2^{-j}t$, and 
$\{\varepsilon _j\}$ in such a way that $\varepsilon _1=t$ and 
\begin{equation}\label{5.36}
\int_0^{\varepsilon _j}\mathcal E(\bfu)^*(s) \, \ds \leq t_j M_0
\end{equation}
for $j\geq 2$. Notice that the function $\mathcal E(\bfu)^*$ is actually integrable near $0$, owing to the first inequality in  \eqref{riesz}. Next, let $\{s_j\}$ be an increasing sequence such that $s_0=0$,  $s_j \to \infty$, $s_{j+1}>s_j+1$ and that, on setting $B_j=\{x\in \mathbb R^n: |x|<s_{j-1}\}$, we have that
\begin{equation}\label{5.37}
 |\mathcal O \setminus B_j|\leq \varepsilon _j \quad \text{for $j\in \N$.}
\end{equation}
Consider the annulus $G_j= \{x\in \mathbb R^n: s_{j-1}\leq |x| < s_{j+1}\}$ and note that $G_j\cap B_j = \emptyset$ for $j \in \N$.  Let $\{\psi_j\}$ be a partition of unity associated with the covering $\{G_j\}$ of $\mathbb R^n$. Thus,
$$\sum _{j=1}^\infty \psi_j = 1 \quad \text{in $\mathbb R^n$,} \qquad {\rm supp}\, \psi _j \subset G_j,$$
and
\begin{equation}\label{jan7}
\|\psi_j\|_{L^\infty(\rn)} + \|\nabla \psi_j\|_{L^\infty(\rn)} \leq c,
\end{equation}
for some constant $c=c(n)$ and every $j \in \setN$. Consider the sequence $\{\bfu_j\}$ defined by
$$\bfu_j = \bfu \,\psi_j$$
for $j \in \setN$.  Clearly, $\sum _{j=1}^\infty \bfu_j=\bfu$. Moreover,   $\bfu_j \in E^1L^1(\rn)$, since $\bfu_j \in L^1(\rn)$ and 
\begin{equation}\label{jan8}
\ep (\bfu_j) = \psi _j \ep (\bfu ) + \nabla \psi_j \otimes^\sym\bfu \in L^1(\rn)
\end{equation}
for $j \in \setN$.  We may thus apply equation \eqref{nov1} to the function $\bfu_j$, and deduce that
\begin{equation}\label{5.39}
K(\bfu_j,t_j; E^1L^1(\R^n),E^{1} L^\infty(\R^n))\approx \int _0^{t_j} \mathcal E(\bfu_j)^*(s) \ds
\end{equation}
for  $j \in \setN$. Let $E_j\subset G_j$ be such that 
\begin{equation}\label{jan9}
|E_j| \leq t_j
\end{equation}
and
\begin{equation}\label{5:40}
\int_0^{t_j} \mathcal E(\bfu_j)^*(s)\, \ds = \int _{E_j}|\mathcal E(\bfu_j)|\, \dx
\end{equation}
for $j \in \setN$. Owing to equations \eqref{jan7} and \eqref{jan8},  
$$|\mathcal E(\bfu_j)| \lesssim |\mathcal E(\bfu)|$$
for $j \in \setN$. Also, $E_j \cap B_j = \emptyset$ and
\begin{equation}\label{jan10}
|\mathcal E(\bfu)| \leq M_0 \quad \text{in $\rn \setminus \mathcal O$.}
\end{equation}
Therefore,
\begin{align}\label{5.41}
\int_{E_j} |\mathcal E(\bfu_j)|\, \dx & \lesssim  \int_{E_j\cap \mathcal O} |\mathcal E(\bfu)|\, \dx  + 
  \int_{E_j\setminus \mathcal O} |\mathcal E(\bfu)|\, \dx
\lesssim \int_{\mathcal O \setminus B_j} |\mathcal E(\bfu)|\, \dx  + 
  t_j\, M_0 
\\ \nonumber & \lesssim \int_0^{\varepsilon _j}\mathcal E(\bfu)^*(s) \, \ds + 
  t_j\, M_0  \lesssim t_j\, M_0
\end{align}
for $j \in \setN$, where the second inequality holds by equations \eqref{jan9} and \eqref{jan10}, the third one by \eqref{5.37} and the fourth one by \eqref{5.36}. From inequalities \eqref{5.39}--\eqref{5.41} we deduce that
\begin{equation}\label{5.42}
K(\bfu_j,t_j;E^1L^1(\R^n),E^{1} L^\infty(\R^n))\lesssim t_j\, M_0
\end{equation}
for $j \in \setN$. Next, for each $j \in \setN$, set  $\theta_j =M(\bfu_j)^\ast(t)$ and  $\lambda_j=M(\ep(\bfu_j))^\ast(t)$, and $(\bfu_j)_{\theta_j, \lambda_j}= T^{\theta_j,\lambda_j}_0\bfu_j$, where $T^{\theta_j,\lambda_j}_0\bfu_j$ is defined as in  \eqref{may3}--\eqref{eq:wlambda0}, with $\Omega$ replaced by $G_j$ and $\bfu$ by $\bfu_j$. Note that this is consistent, since ${\rm supp}\, \bfu_j \subset G_j$. Set, for simplicity, $\bfv_j = (\bfu_j)_{{\theta_{j}}{\lambda_{j}}}$.  Thanks to
Theorem \ref{cor:Tlest},
%
we have that $\bfv_j \in E^1L^\infty (\rn)$ and  ${\rm supp}\, \bfv_j \subset G_j$. Moreover,
\begin{align}\label{5.43}
\|\bfu_j - \bfv_j\|_{E^1L^1(\rn)} + t_j \|\bfv_j\|_{ E^1L^\infty (\rn)} \lesssim \int _0^{t_j} \mathcal E(\bfu_j)^*(s) \ds \lesssim K(\bfu_j,t_j;E^1L^1(\R^n),E^{1} L^\infty(\R^n))
\end{align}
for $j \in \setN$,
where the first inequality holds owing to equations \eqref{a}--\eqref{may2} applied with $\bfu$ replaced by $\bfu_j$, and the second one owing to equation \eqref{5.39}. Equations \eqref{5.42} and \eqref{5.43} yield
\begin{align}\label{5.44}
\|\bfu_j - \bfv_j\|_{E^1L^1(\rn)} \lesssim  t_j\, M_0
\end{align}
and 
\begin{align}\label{5.45}
\|\bfv_j\|_{ E^1L^\infty (\rn)} \lesssim   M_0
\end{align}
for $j \in \setN$.
Define the function $\bfv : \rn \to \rn$ as 
$$\bfv = \sum_{j\geq 2} \bfv_j.$$
Notice that $\bfv \in E^1L^\infty (\rn)$, since  ${\rm supp}\, \bfv_j \subset G_j$  for $j \in \setN$, and hence
 the sum in the last equality is locally extended to a fixed finite number of indices $j$. Furthermore,  thanks to equation \eqref{5.45},
\begin{equation}\label{jan20}
\|\bfv\|_{ E^1L^\infty (\rn)} \lesssim   M_0.
\end{equation}
Inasmuch as 
$$ \bfu = \sum_{j=1}^\infty \bfu_j = \bfu_1 + \sum_{j=2}^\infty (\bfu_j - \bfv_j) + \bfv,$$
one has that
\begin{align}\label{jan21}
K(\bfu,& t; E^1L^1(\R^n),E^{1} L^\infty(\R^n))\\ \nonumber & \leq K(\bfu_1, t; E^1L^1(\R^n),E^{1} L^\infty(\R^n))
+ \sum_{j=2}^\infty \|\bfu_j - \bfv_j\|_{E^1L^1(\rn)} + t \|\bfv\|_{E^1L^\infty(\rn)}
\\ \nonumber 
& \lesssim  \int _0^{t} \mathcal E(\bfu_1)^*(s) \ds + \sum_{j=2}^\infty t_j M_0 + t M_0  
\lesssim  \int _0^{t} \mathcal E(\bfu)^*(s) \ds +   t M_0
\\ \nonumber 
& =  \int _0^{t} \mathcal E(\bfu)^*(s) \ds +  t M(\mathcal E(\bfu))^*(t) \lesssim \int _0^{t} \mathcal E(\bfu)^*(s) \ds,
\end{align}
where the second inequality relies upon equations \eqref{5.39}, \eqref{5.44} and \eqref{5.45}, and the last inequality is a consequence of inequality  \eqref{riesz}. Inequality \eqref{jan21} provides us with the 
relation $\lesssim$ in \eqref{nov1}. The proof of equation \eqref{nov1} is thus complete.
\\ 
Part (ii). The  inequality $\gtrsim$ in equation \eqref{nov2'} follows via a chain analogous to \eqref{jan22}. In order to establish
the reverse inequality,   assume that $\bfu \in E^1_bL^1(\R^n)+ E^1_bL^\infty(\rn)$. We have to show that
\begin{align}\label{feb65}
K(\bfu,t; { E^1_b} L^1(\R^n),  { E^1_b} L^\infty (\R^n))\lesssim 
\int_0^t\ep(\bfu)^\ast(s)\ds.
\end{align}
Here, and throughout the proof of Part (ii), the equivalence constants depend only on $n$.
 Given   $\lambda >0$, let  $ \mathcal O_\lambda$ and $T^\lambda \bfu$ be the set and the function defined as in \eqref{Olambda} and \eqref{eq:wlambda'}, respectively. 
%
%
%
Thus,  
owing to Theorem \ref{lem:feb1}, Part (ii), applied with $E^1_bX(\rn)=E^1_bL^1(\R^n)+ E^1_bL^\infty(\rn)$,  we have that $T^\lambda \bfu \in E^{1}_bL^\infty (\R^n)$,  and  $T^\lambda \bfu  = \bfu$ in  $\R^n \setminus \mathcal O_{\lambda}$.
Moreover,  
\begin{align}\label{feb30}
|\ep(T^\lambda \bfu)|\lesssim \lambda.
\end{align}
Fix $t>0$, choose  $\lambda=M(\ep(\bfu))^\ast(t)$ and set $\bfu_1=T^\lambda \bfu$ and $\bfu_0=\bfu-\bfu_1$. Plainly, 
\begin{align}\label{feb35}
K(\bfu,t; { E^1_b} L^1(\R^n),  { E^1_b} L^\infty (\R^n))\leq \|\ep(\bfu_0)\|_{L^{1}(\R^n)}+t\|\ep(\bfu_1)\|_{L^{\infty}(\R^n)}.
\end{align}
By \eqref{feb30} and the choice of  $\lambda$,  
\begin{align}\label{feb31}
\|\ep(\bfu_1)\|_{L^{\infty}(\R^n)}\lesssim \lambda =  M(\ep(\bfu))^\ast(t) \lesssim \frac 1t \int_0^t\ep(\bfu)^\ast(s)\ds,
\end{align}
where the last inequality holds thanks to  inequality \eqref{riesz}. 
Next, note that, since $\bfu_0=0$ in  $\R^n \setminus \mathcal O_{\lambda}$,  
\begin{align}\label{feb32}
\|\ep(\bfu_0)\|_{L^{1}(\R^n)}\leq \|\ep(\bfu)\chi_{\mathcal O_{\lambda}}\|_{L^{1}(\R^n)}+ \|\ep(\bfu_1)\chi_{\mathcal O_{\lambda}}\|_{L^{1}(\R^n)}.
\end{align}
The 
 choice of   $\lambda$ ensures  that $|\mathcal O_{\lambda}|\leq t$. 
 Therefore,
\begin{align}\label{feb33}
\|\ep(\bfu)\chi_{\mathcal O_{\lambda}}\|_{L^{1}(\R^n)} 
&  \leq\|\ep(\bfu)^\ast\chi_{(0,t)}\|_{L^{1}(0, \infty)}  =\,\int_0^t\ep(\bfu)^\ast(s)\ds.
\end{align}
Also,
\begin{align}\label{feb34}
\|\ep(\bfu_1)\chi_{\mathcal O_{\lambda}}\|_{L^{1}(\R^n)}&\lesssim \lambda = t M(\ep(\bfu))^\ast(t) \lesssim  \int_0^t\ep(\bfu)^\ast(s)\ds,
\end{align}
where the last inequality holds by    inequality \eqref{riesz} again. Combining inequalities \eqref{feb35}--\eqref{feb34} yields inequality \eqref{feb65}.
Equation \eqref{nov2'} is thus established.
\\ Now, consider equation \eqref{nov2''}.  The proof of the inequality $\lq\lq \gtrsim"$ in this equation follows as in \eqref{jan22}. In order to prove the inequality $\lq\lq \lesssim"$, assume  that  $\bfu \in E^1_0L^1(\R^n) + E^{1}_0L^\infty(\R^n)$.  Then $\bfu = \bfv_0 + \bfv_1$, for some $\bfv_0\in E^1_0L^1(\R^n)$ and $\bfv_1 \in E^{1}_0L^\infty(\R^n)$.  Given any $\varepsilon >0$, there exist $\overline {\bfv}_0 \in E^1_bL^1(\R^n)$ and $\overline {\bfv}_1 \in E^1_bL^\infty(\R^n)$ such that
\begin{equation}\label{mar60} \|\ep(\bfv_0) -  \ep(\overline {\bfv}_0)\|_{L^1(\R^n)}  + \|\ep(\bfv_1) -  \ep(\overline {\bfv}_1)\|_{L^\infty(\R^n)}
< \varepsilon.
\end{equation}
Thus, 
on setting $\overline \bfu = \overline {\bfv}_0 + \overline\bfv_1$, we have that $\overline\bfu \in E^1_bL^1(\R^n)+ E^1_bL^\infty(\R^n)$.
Fix $t>0$. An application of inequality \eqref{feb65} to the function  $\overline \bfu$  and the definition of the $K$-functional ensure that there exist functions $\overline {\bfu}_0 \in E^1_bL^1(\R^n)$ and $\overline {\bfu}_1 \in  E^1_bL^\infty(\R^n)$ such that $\overline \bfu = \overline {\bfu}_0 + \overline {\bfu}_1$ and 
\begin{align}\label{feb66}
\| \ep(\overline {\bfu}_0) \|_{L^1(\R^n)} + t \|\ep( \overline {\bfu}_1) \|_{L^\infty(\R^n)} & \lesssim \int_0^t\ep(\overline \bfu)^\ast(s)\ds \leq \int_0^t\ep(\overline \bfu- \bfu)^\ast(s)\ds +\int_0^t\ep(\bfu)^\ast(s)\ds,
\end{align}
where the second inequality holds owing to property \eqref{subadd}. 
Observe  that
\begin{equation}\label{mar61}
\bfu = (\bfv_0 -\overline {\bfv}_0) + \overline {\bfu}_0 +  (\bfv_1 -\overline {\bfv}_1) + \overline {\bfu}_1.
\end{equation}
Thus, 
{\color{black} inequality \eqref{mar60} yields}
\begin{align}\label{mar62}
 \int_0^t\ep(\overline \bfu - \bfu)^*(s)\ds & \leq  \int_0^t\ep(\bfv_0 -\overline {\bfv}_0)^*(s)\ds +  \int_0^t\ep(\bfv_1 -\overline {\bfv}_1)^*(s)\ds \\ \nonumber &
\leq   \|\ep(\bfv_0) -  \ep(\overline {\bfv}_0)\|_{L^1(\R^n)}  + t \|\ep(\bfv_1) -  \ep(\overline {\bfv}_1)\|_{L^\infty(\R^n)} \leq \varepsilon (1+t).
\end{align}
From the definition of the $K$-functional and inequalities \eqref{mar60}, \eqref{feb66} and \eqref{mar62} we deduce that
\begin{align}\label{mar63}
 K(\bfu,t;  &{E^1_0}L^1(\R^n), {E^1_0}L^{\infty}(\R^n)) 
 \leq \|\ep (\bfv_0 -\overline {\bfv}_0)  +\ep ( \overline {\bfu}_0)\|_{L^1(\rn)} + t \|\ep (\bfv_1 -\overline {\bfv}_1) + \ep(\overline {\bfu}_1)\|_{L^{\infty}(\rn)} \\ \nonumber & \leq 
\|\ep (\bfv_0 -\overline {\bfv}_0)\|_{L^1(\rn)}  + \| \ep  (\overline {\bfu}_0)\|_{L^1(\rn)} + t \|\ep (\bfv_1 -\overline {\bfv}_1) \|_{L^{\infty}(\rn)} + t\|\ep(\overline {\bfu}_1)\|_{L^{\infty}(\rn)}
\\ \nonumber & \lesssim \varepsilon (1+t) +   \int_0^t\ep(\overline \bfu - \bfu)^*(s)\ds +  \int_0^t\ep(\bfu)^*(s)\ds
 \lesssim  \varepsilon (1+t)  +  \int_0^t\ep(\bfu)^*(s) \ds.
\end{align}
Hence, the inequality $\lesssim$ in \eqref{feb35} follows, owing  to the arbitrariness of $\varepsilon$.
\par \noindent Part (iii).
The inequality $\gtrsim$ in equation \eqref{nov3} follows as in \eqref{jan22}.  Consider the reverse inequality. Fix $t>0$. Assume that $\bfu \in E^1L^1(\Omega) + E^1L^\infty (\Omega)$, and let $\mathscr E_\Omega$ be the extension operator provided by 
Theorem \ref{thm:extensionoperator}.
%
We have that
\begin{align}\label{feb26mar}
    K(\bfu,t;E^1L^1(\Omega),E^1L^\infty(\Omega))&=K(\mathscr E_\Omega\bfu,t;E^1L^1(\Omega),E^1L^\infty(\Omega)) \leq K(\mathscr E_\Omega\bfu,t;E^1L^1(\R^n),E^1L^\infty(\R^n))\\ \nonumber
    & \lesssim
\int_0^t\mathcal E(\mathscr E_\Omega\bfu)^*(s)\, \ds    \lesssim  \int_0^t\mathcal  E(\bfu)^*(s)\, \ds,
%
\end{align}
where the second  inequality follows from equation \eqref{nov1} and the last inequality holds by inequality  \eqref{ext*}. This proves  the inequality $\lesssim$ in equation \eqref{nov3}.
\end{proof}

\begin{proof}[Proof of Theorem \ref{corollary}] Part (i). 
The interpolation space $(L^1(\rn), L^\infty(\rn))_{\frac 1{n'}, 1}$ is defined as the rearrangement-invariant space equipped with the norm defined as
$$ \|\bfv\|_{(L^1(\rn), L^\infty(\rn))_{\frac 1{n'}, 1}} = 
\int _0^\infty t^{-\frac 1{n'}} K(\bfv, t; L^1(\rn), L^\infty(\rn)) \frac{\dt}{t}$$
for $\bfv \in  \mathcal M(\rn)$ -- see e.g. \cite[Chapter 5, Definition 1.7]{BS}.
%
Thus, by equation \eqref{apr30},
\begin{equation}\label{apr31}
 \|\bfv\|_{(L^1(\rn), L^\infty(\rn))_{\frac 1{n'}, 1}} \approx
\int _0^\infty t^{-\frac 1{n'}}\bigg( \int_0^t \bfv^*(s)\,\ds\bigg) \frac{\dt}{t}.
\end{equation}
By \cite[Chapter 5, Theorem 1.9]{BS}, 
\begin{equation}\label{july3}
L^{n.1}(\rn) = (L^1(\rn), L^\infty(\rn))_{\frac 1{n'}, 1}.
\end{equation}
Owing to equations  \eqref{july3}, \eqref{apr31} and  \eqref{nov1},
\begin{align}\label{july5}
E^1L^{n,1}(\rn) & = E^1  (L^1(\rn), L^\infty(\rn))_{\frac 1{n'}, 1} = \big\{\bfu: \|\mathcal E(\bfu)\|_{(L^1(\rn), L^\infty(\rn))_{\frac 1{n'}, 1}} < \infty\big\}
\\ \nonumber & = \bigg\{\bfu: \int _0^\infty t^{-\frac 1{n'}}\bigg( \int_0^t \mathcal E(\bfu)^*(s)\,\ds\bigg) \frac{\dt}{t} < \infty\bigg\}
\\ \nonumber & =  \bigg\{\bfu: \int _0^\infty t^{-\frac 1{n'}}K (\bfu, t; E^{1}L^1(\rn), E^{1}L^\infty(\rn)) \frac{\dt}{t} < \infty\bigg\}
\\ \nonumber & = (E^{1}L^1(\rn), E^{1}L^\infty(\rn))_{\frac 1{n'}, 1}\,.
\end{align}
From equation \eqref{july5}, \cite[Corollary 3.6.2]{BerghLofstrom} and \eqref{nov1} one deduces that
\begin{align}\label{july7}
K(\bfu, t^{\frac 1{n'}}; E^{1}L^1(\rn), E^1L^{n,1}(\rn)) &  \approx t^{\frac 1{n'}}\int _t^\infty s^{-\frac 1{n'}} K(\bfu, s;E^{1}L^1(\rn), E^{1}L^\infty(\rn))\, \frac{\ds}{s} 
\\ \nonumber & \approx t^{\frac 1{n'}}\int _t^\infty s^{-\frac 1{n'}-1}\bigg( \int _0^s \mathcal E(\bfu)^* (r)\, \dr\bigg)\, \ds 
\\ \nonumber &
\approx \int _0^{t} \mathcal E(\bfu)^*(s) \ds + t^{\frac 1{n'}} \int_{t}^\infty  \mathcal E(\bfu)^*(s) s^{-\frac 1{n'}} \ds \quad \text{for $t>0$,}
\end{align}
%
%
%
%
where the last equivalence holds owing to Fubini's theorem. Equation \eqref{feb5} is equivalent to \eqref{july7}.
\\  Part (ii).  Equation \eqref{feb6} follows from equation \eqref{nov2'} via the same argument  that yields \eqref{feb5} from \eqref{nov1}. Equation \eqref{feb6'} can be deduced from  \eqref{feb6} via an approximation argument analogous to the argument which implies \eqref{nov2''}  via \eqref{nov2'}.
\\
Part (iii). The proof is completely analogous to   that of Part (i). One has just to replace $\rn$ by $\Omega$ in equations \eqref{july3}--\eqref{july7}, and make use of equation \eqref{nov3} instead of \eqref{nov1}. 
\end{proof}

As a direct application of Theorem \ref{K-funct}, we can now provide a proof of Theorem \ref{C1}.

\begin{proof}[Proof of Theorem \ref{C1}] By assumption \eqref{boundT} and a basic
property of the $K$-functional -- see \cite[Theorem 1.11, Chapter
5]{BS}  -- one has that
\begin{align}\label{c2}
K&\left( T\bfu, s; E^1L^1(\rn), E^1L^\infty(\rn)\right)
\leq (M_1 + M_\infty )
K\left( \bfu, s; E^1L^1(\Omega), E^1L^\infty(\Omega)\right) \qquad
\hbox{for $s>0$},
\end{align}
for every $\bfu\in\, E^1L^1(\Omega)+E^1L^\infty(\Omega)$.
Owing to inequality (\ref{c2}) and equations \eqref{nov1} and \eqref{nov3}, 
\begin{equation}\label{c4}
\mathcal E(T\bfu)^{**}(s)\leq C
\mathcal E(\bfu)^{**}(s) \qquad \hbox{for
$s>0$},
\end{equation}
for some constant $C$ depending  on $\Omega$, $M_1$ and
$M_\infty$, and for every $\bfu\in\, E^1L^1(\Omega)+E^1L^\infty(\Omega)$.
From inequality \eqref{c4}, via \eqref{hardy} and \eqref{mar170},
we deduce that
\begin{align}\label{may9}
\|T\bfu \|_{E^1\widehat X (\rn )} & \approx \|\mathcal E(T\bfu)\|_{\widehat X (\rn )} =
 \|\mathcal E(T\bfu)^*\|_{\widehat X (0, \infty)} \leq  C 
 \|\mathcal E(\bfu)^*\|_{\widehat X (0, \infty)} \\ \nonumber &=  C
 \|\mathcal E(\bfu)^*\|_{X (0, |\Omega|)} =    C \|\mathcal E(\bfu) \|_{E^1X  (\Omega )}
 \approx C \|\bfu \|_{E^1X  (\Omega )}
\end{align}
for every $\bfu \in E^1X  (\Omega )$, where $C$ is the constant appearing in inequality \eqref{c4} and the equivalence constants depend only on $n$.
Hence, property
\eqref{boundTX} follows.
\end{proof}

\section{Sobolev embeddings into rearrangement-invariant spaces}\label{sobemb}
The discussion of Sobolev embeddings for symmetric gradient Sobolev spaces -- the core of the present paper -- begins with this section, which is  concerned with embeddings into rearrangement-invariant spaces. As disclosed in Section \ref{intro}, the point of departure and central step of our approach is a reduction principle. This principle tells us that any embedding of a symmetric gradient Sobolev space built upon a rearrangement-invariant space is equivalent to a one-dimensional Hardy type inequality involving the corresponding function norms. Since a parallel equivalence is known to hold for their full gradient analogues, it turns out that symmetric gradient and full gradient Sobolev embeddings into rearrangement-invariant spaces are always equivalent.

\subsection{Embeddings on open sets with finite measure}\label{sscomega}

The reduction theorem on open sets with finite measure reads as follows.

\begin{theorem}\label{reduction} {\bf [Sobolev embeddings on open sets into rearrangement-invariant spaces] }
Let $\Omega$ be an open set in $\R^n$ with $|\Omega|< \infty$.  Let $X(\Omega)$ and $Y(\Omega)$ be rearrangement-invariant spaces. 

\vspace{0.1cm}
\par\noindent 
{\bf Part I [Embeddings for  $E^1_0X(\Omega)$]}  The following facts are equivalent:
\\ (i) The embedding $E^1_0X(\Omega) \to Y(\Omega)$ holds, namely there exists a constant $c_1$ such that
\begin{equation}\label{july8}
\|\bfu \|_{Y(\Omega)} \leq c_1 \|\ep(\bfu)\|_{X(\Omega)}
\end{equation}
 for every $\bfu \in E^1_0X(\Omega)$;
\\ (ii) The embedding $ W^1_0X(\Omega) \to Y(\Omega)$ holds, namely there exists a constant $c_2$ such that
\begin{equation}\label{july9}
\|\bfu \|_{Y(\Omega)} \leq c_2 \|\nabla \bfu\|_{X(\Omega)}
%
\end{equation}
 for every $\bfu \in W^1_0X(\Omega)$;
\\ (iii) There exists a constant $c_3$ such that
\begin{equation}\label{july10}
\bigg\|\int_s^{|\Omega|} f (r)r^{-1+\frac 1{n}}\dr \bigg\|_{Y(0, |\Omega|)} \leq  c_3\|f\|_{X(0, |\Omega|)}
\end{equation}
 for every non-increasing function $f: (0, |\Omega|)\to [0, \infty)$. Moreover the constants $c_1$ and $c_2$ in inequalities  \eqref{july8} and \eqref{july9} depend only on the constant $c_3$   in \eqref{july10} {\color{black} and on $n$}.

\vspace{0.1cm}
\par\noindent 
  {\bf Part II [Embeddings for  $E^1X(\Omega)$]} Assume, in addition that $\Omega$ is an $(\ve, \delta)$-domain. Then the following facts are equivalent:
\\ (i) The embedding $E^1X(\Omega) \to Y(\Omega)$ holds, namely there exists a constant $c_1$ such that
\begin{equation}\label{jan40}
\|\bfu \|_{Y(\Omega)} \leq c_1 \|\bfu\|_{ E^1X(\Omega)}
\end{equation}
 for every $\bfu \in E^1X(\Omega)$;
%
\\ (ii) The embedding $W^1X(\Omega) \to Y(\Omega)$ holds, namely there exists a constant $c_2$ such that
\begin{equation}\label{jan41}
\|\bfu \|_{Y(\Omega)} \leq c_2 \|\bfu\|_{W^1X(\Omega)}
%
\end{equation}
for every $\bfu \in W^1X(\Omega)$;
%
\\ (iii) Inequality \eqref{july10} holds for every non-increasing function $f: (0, |\Omega|)\to [0, \infty)$. 
%
%
\\ Moreover the constants $c_1$ and $c_2$ in inequalities  \eqref{jan40} and \eqref{jan41} depend only on the constant $c_3$   in \eqref{july10}  and on $\Omega$. 
\end{theorem}

\begin{remark}\label{increasing}
{\rm Inequality \eqref{july10} holds for every non-increasing   function $f: (0, |\Omega|)\to [0, \infty)$ if and only if it just holds for every measurable   function $f: (0, |\Omega|)\to [0, \infty)$ -- see \cite[Corollary 9.8]{CPS_advances}. This nontrivial property can be of use in applications of Theorem \ref{sobemb}, since the characterizations of Hardy type inequalities of the form \eqref{july10}  for specific spaces are different, in general, depending on whether arbitrary functions or just non-increasing trial functions $f$ are considered.}
\end{remark}

The next result ensures that a Sobolev-Poincar\'e type inequality on  connected   $(\ve, \delta)$-domains with finite measure is available whenever a corresponding embedding holds.

\begin{theorem}\label{poincare-omega}{\rm{\bf [Sobolev-Poincar\'e type inequalities] }}
Let $\Omega$ be a connected   $(\ve, \delta)$-domain
 in $\R^n$ with $|\Omega|< \infty$.  Assume that $X(\Omega)$ and $Y(\Omega)$ are rearrangement-invariant  spaces for which any of the equivalent properties (i)-(iii) of Theorem \ref{reduction}, Part II, holds. Then 
there exists a constant $c$ such that
\begin{equation}\label{feb11}
\inf _{\bfv \in \mathcal R}\|\bfu -\bfv\|_{Y(\Omega)} \leq c \|\ep(\bfu)\|_{X(\Omega)}
\end{equation}
for every $\bfu \in E^1X(\Omega)$. 
\end{theorem}

Theorem \ref{reduction}, coupled with \cite[Theorem 3.5]{CP-arkiv}, yields the following characterization of embeddings of $E^1_0X(\Omega)$ or $E^1X(\Omega)$ into  $L^\infty(\Omega)$, and of parallel Sobolev-Poincar\'e type inequalities. The optimal condition on the function norm $\|\cdot \|_{X(0, |\Omega|)}$ for such an embedding to hold amounts to the finiteness of the norm $\|\cdot \|_{X'(0, |\Omega|)}$ of the function $r^{-\frac 1{n'}}$. Note that, in the scale of Lebesgue function norms $\|\cdot \|_{L^p(0, |\Omega|)}$, this condition reproduces the  well known assumption that $p>n$.

\begin{corollary}\label{Linfa}{\rm{\bf [Sobolev embeddings into $L^\infty$] }}Let $\Omega$ be an open set in $\R^n$ with $|\Omega|< \infty$, and   let $X(\Omega)$   be a rearrangement-invariant space. Then inequalities \eqref{july8} and \eqref{july9} hold with $Y(\Omega) = L^\infty(\Omega)$ if and only if 
\begin{equation}\label{apr40}\|r^{-\frac 1{n'}}\|_{X'(0, |\Omega|)}< \infty.
\end{equation}
Moreover, the latter condition is equivalent to the embedding
\begin{equation}\label{apr41}
X(\Omega) \to L^{n,1}(\Omega).
\end{equation}
Under the additional assumption that $\Omega$ is an $(\ve, \delta)$-domain, this condition and embedding are also equivalent to inequalities \eqref{jan40} and \eqref{jan41}.
\end{corollary}

The optimal rearrangement-invariant  function norm $\|\cdot\|_{Y(0,|\Omega|)}$ in  inequality \eqref{july10}, and  the corresponding  optimal target space $Y(\Omega)$ in inequalities \eqref{july8} and \eqref{july9},  are denoted by  $\|\cdot\|_{X_1(0, |\Omega|)}$ and $X_1(\Omega)$,  respectively.  The relevant function norm obeys
\begin{align}\label{Z}
\|f\|_{X_1'(0, |\Omega|)}=\|s\sp{\frac 1n}f\sp{**}(s)\|_{X'(0,|\Omega|)}
\end{align}
for $f \in \mathcal \Mpl (0, |\Omega|)$. \textcolor{black}{This is the content of the following result.}

\begin{theorem}\label{optomega}{\rm{\bf [Optimal rearrangement-invariant target in Sobolev embeddings on open sets] }}
Let $\Omega$ be an open set in $\R^n$ with $|\Omega|< \infty$.  Let $X(\Omega)$  be a rearrangement-invariant space and let $X_1(\Omega)$ be the rearrangement-invariant space defined via the function norm given by \eqref{Z}.

\vspace{0.1cm}
\par\noindent 
{\bf Part I [Embeddings for $E^1_0X(\Omega)$]}   The embedding $E^1_0X(\Omega) \to X_1(\Omega)$ holds, namely there exists a constant $c$ such that
\begin{equation}\label{jan74}
\|\bfu\|_{X_1(\Omega)} \leq c \|\ep (\bfu)\|_{X(\Omega)}
\end{equation}
for every $\bfu \in E^1_0X(\Omega)$. Moreover,
$X_1(\Omega)$ is the optimal (smallest possible) rearrangement-invariant  target space in \eqref{jan74}.

\vspace{0.1cm}
\par\noindent 
 {\bf Part II [Embeddings for $E^1X(\Omega)$]}
Assume, in addition that $\Omega$ is an $(\ve, \delta)$-domain.  Then   $E^1X(\Omega) \to X_1(\Omega)$, namely there   exists a constant $c$ such that 
\begin{equation}\label{jan75}
\|\bfu\|_{X_1(\Omega)} \leq c \|\bfu\|_{E^1X(\Omega)}
\end{equation}
for every $\bfu \in E^1X(\Omega)$. Moreover, $X_1(\Omega)$ is the optimal (smallest possible) rearrangement-invariant  target space in \eqref{jan75}.
\end{theorem}

\begin{proof}[Proof of Theorem \ref{reduction}] \emph{Part} I.
Embedding \eqref{july8} trivially implies embedding \eqref{july9}. The fact that the latter embedding implies inequality \eqref{july10} is well known, and  basically follows by considering radially decreasing functions in $W^1_0X(\Omega)$ that are compactly supported in a ball contained in $\Omega$ -- see \cite{EKP}.
\\ Let us prove that inequality \eqref{july10} implies embedding \eqref{july8}. A  result of \cite{SV} ensures that 
\begin{equation}\label{feb22}
\|\bfu\|_{ L^{n',1}(\rn)} \leq C \|\ep(\bfu)\|_{L^1(\rn)}
\end{equation}
for some constant $C=C(n)$ and for every $\bfu \in E^1L^1(\rn)$.
Specifically,   inequality \eqref{feb22} follows from inequality \cite[Theorem 1]{SV} for smooth compactly supported functions in $\rn$, owing to the fact that the latter are dense in $E^1L^1(\rn)$. 
%
%
\\
On the other hand, we claim that
\begin{equation}\label{july12}
\|\bfu\|_{ L^{\infty}(\rn)} \leq C \|\ep(\bfu)\|_{L^{n,1}(\rn)}
\end{equation}
for every $\bfu \in E^1L^{n,1}(\rn)$. 
Indeed, if $\bfu \in C^\infty_0(\R^n)$, then inequality \eqref{july12} follows from \eqref{rearrepu1}, which implies that  
\begin{equation*}
\|\bfu\|_{ L^{\infty}(\rn)}= \bfu^*(0) \leq C \int_0^\infty \ep(\bfu)^*(r)r^{-\frac 1{n'}}\dr =C \|\ep(\bfu)\|_{L^{n,1}(\rn)}
\end{equation*}
for some constant $C=C(n)$.
 Inequality \eqref{july12} continues to hold for every $\bfu \in E^1L^{n,1}(\rn)$, since the space $C^\infty_0(\R^n)$ is dense in $E^1L^{n,1}(\R^n)$.
\par\noindent
Now, denote as above by $\bfu_r : \Omega \to \rn$ the restriction to $\Omega$ of a function $\bfu : \rn \to \rn$.  
 Inequality \eqref{feb22} entails, 
in particular, that
\begin{equation}\label{feb37}
\|\bfu_r\|_{ L^{n',1}(\Omega)} \leq C \|\ep(\bfu)\|_{L^1(\rn)}
\end{equation}
for every $\bfu \in {E^1_b}L^1(\rn)$, and inequality  \eqref{july12} implies that 
\begin{equation}\label{feb38}
\|\bfu_r\|_{ L^{\infty}(\Omega)} \leq C \|\ep(\bfu)\|_{L^{n,1}(\rn)}
\end{equation}
for every $\bfu \in  E^1_bL^{n,1}(\rn)$. 
Inequalities \eqref{feb37} and \eqref{feb38} tell us that the linear operator that maps $\bfu$ to $\bfu_r$
%
%
\begin{equation}\label{feb39}
\text{is bounded from ${E^1_b}L^1(\rn)$ into $L^{n',1}(\Omega)$ and from $ {E^1_b}L^{n,1}(\rn)$ into $L^{\infty}(\Omega)$.}
%
\end{equation}
 A basic property of the $K$-functional -- see e.g. \cite[Chapter 5, Theorem 1.11]{BS} -- then tells us that
\begin{equation}\label{july16}
K(\bfu_r, t;L^{n',1}(\Omega), L^\infty(\Omega))
\leq C K(\bfu, t/C;  {E^1_b}L^1(\R^n),   {E^1_b}L^{n,1}(\R^n)) \quad\text{for $t>0$,}
\end{equation}
for some constant $C=C(n)$ and for every $\bfu\in   {E^1_b}L^1(\R^n)+  {E^1_b}L^{n,1}(\R^n){\color{black} = {E^1_b}L^1(\R^n)}$.  Owing to \cite[Corollary 2.3, Chapter 5]{BS}, 
\begin{equation}\label{july17}
K(\bfu_r, t;L^{n',1}(\Omega), L^\infty(\Omega)) \approx \int_0^{t^{n'}} s^{-\frac{1}{n}} \bfu_r^*(s)\ds \quad\text{for $t>0$.}
\end{equation} 
By equations  \eqref{feb6} and \eqref{july17}, inequality \eqref{july16} yields
\begin{equation}\label{july18}
 \int_0^{t^{n'}} s^{-\frac{1}{n}} \bfu_r^*(s)\ds
\leq C 
 \bigg(\int _0^{t^{n'}/C} \ep(\bfu)^*(s) \ds + t \int_{t^{n'}/C}^\infty  \ep(\bfu)^*(s) s^{-\frac 1{n'}} \ds\bigg) \quad\text{for $t>0$.}
\end{equation}
If $\bfu \in E^1_bX(\Omega)$, then, by the definition of the latter space, its extension $\bfu_e \in E^1_bX(\R^n)$.
 Thus, from an application of inequality \eqref{july18} with $\bfu$ replaced by $\bfu_e$ one obtains that
\begin{equation}\label{feb40}
 \int_0^{t^{n'}} s^{-\frac{1}{n}} \bfu^*(s)\ds
\leq C 
 \bigg(\int _0^{t^{n'}/C} \ep(\bfu_e)^*(s) \ds + t \int_{t^{n'}/C}^\infty  \ep(\bfu_e)^*(s) s^{-\frac 1{n'}} \ds\bigg) \quad\text{for $t>0$.}
\end{equation}
An estimate for the right-hand side of \eqref{feb40} as in \cite[Proof of Theorem 4.2]{KermanPick} enables us to deduce that
\begin{equation}\label{july19}
 \int_0^{t^{n'}} s^{-\frac{1}{n}} \bfu^*(s)\ds
\leq C 
 \int_0^{ct^{n'}} s^{-\frac{1}{n}} \int_s^{\infty}\ep(\bfu_e)^*(r) r^{-\frac{1}{n'}}\,\mathrm{d}r\ds  \quad\text{for $t>0$,}
\end{equation}
for some constants   $c=c(n)\geq 1$ and $C=C(n)$. Observe that  $\ep(\bfu_e)^*(s)= \ep(\bfu)^*(s)$ if $s\in (0, |\Omega|)$, and $\ep(\bfu_e)^*(s)=0$ if $s\in [|\Omega|, \infty)$.
Hence, 
via a change of variables,   
\begin{align}\label{july20'}
\int_0^t  s^{-\frac{1}{n}} \bfu^*(s)\ds
\leq C\int_0^{t}  s^{-\frac{1}{n}} \int_{s/c}^{|\Omega|} \ep(\bfu)^*(r) r^{-\frac{1}{n'}}\,\mathrm{d}r\ds  \quad\text{for $t\in (0, |\Omega|)$,}
\end{align}
for some constants  $c=c(n)\geq 1$ and $C=C(n)$. An argument from \cite[Proof of Theorem A]{KermanPick} (see also \cite[Proof of Theorem 4.1]{CPS_Frostman} for an alternative simpler proof), relying only on one-dimensional inequalities, tells us that, if inequality \eqref{july10} holds, and the  measurable functions $f, g : (0,  |\Omega|) \to \mathbb R$ are such that
\begin{align}\label{july21}
\int_0^t  s^{-\frac{1}{n}} g^*(s)\ds
\leq C\int_0^{t}  s^{-\frac{1}{n}} \int_{s/c}^{|\Omega|} f^*(r) r^{-\frac{1}{n'}}\,\mathrm{d}r\ds  \quad\text{for $t\in (0, |\Omega|)$,}
\end{align}
for some constants   $c>1$ and $C>0$, then
\begin{align}\label{july22}
\|g\|_{Y(0, |\Omega|)} \leq C' \|f\|_{X(0, |\Omega|)},
\end{align}
for a suitable constant $C'=C'(c,C)$. Owing to inequality \eqref{july20'}, an application of this property with $g= \bfu^*$ and $f =\ep (\bfu)^*$ tells
us that
\begin{align}\label{july23}
\|\bfu\|_{Y(\Omega)} = \| \bfu^*\|_{Y(0,|\Omega|)} \leq C'   \|\ep(\bfu)^*\|_{X(0,|\Omega|)} =   C' \|\ep(\bfu)\|_{X(\Omega)}
\end{align}
for every $\bfu \in E^1_bX(\Omega)$, where $C'=C'(n)$. We have thus shown that, if $\bfu \in E^1_bX(\Omega)$, then  $\bfu \in Y(\Omega)$, and 
 inequality \eqref{july10} holds.
\\ Assume now that $\bfu \in  E^1_0X(\Omega)$. Let $\{\bfu_k\}$ be a sequence in $E^1_bX(\Omega)$ such that  $\bfu_k \to \bfu$ in $E^1_bX(\Omega)$. Thus, $\{\ep(\bfu_k)\}$ is a Cauchy sequence in $X(\Omega)$, and there exists a function $\bfU : \Omega \to \mathbb R^{n\times n}$ such that $\ep(\bfu_k) \to \bfU$ in $X(\Omega)$. Owing to inequality \eqref{july23} applied to $\{\bfu_ k-\bfu_m\}$ for $k,m \in \mathbb N$,  $\{\bfu_k\}$ is a Cauchly sequence in $Y(\Omega)$.  Hence, it converges to some function $\bfv$ in $Y(\Omega)$ and (up to subsequences) a.e. in $\Omega$. Altogether, we have that $\bfu=\bfv \in Y(\Omega)$ and $\ep(\bfu)= \bfU \in X(\Omega)$. Moreover, passing  to the limit in inequality \eqref{july23} applied with $\bfu$ replaced by $\bfu_k$, yields \eqref{july23} for $\bfu$, thanks to the Fatou's lemma for rearrangement-invariant function norms (see \cite[Chapter 1, Theorem 1.7]{BS}).
%
%

\smallskip
\par\noindent \emph{Part} II. 
Plainly, inequality \eqref{jan40}  implies inequality \eqref{jan41}, and  the latter   implies inequality \eqref{july10}, as shown in \cite[Theorem A]{KermanPick}. As for the fact that inequality  \eqref{july10} implies inequality \eqref{jan40}, note that inequalities \eqref{feb22} and \eqref{july12},
combined  with the use of the extension operator $\mathscr{E}_\Omega$ from Theorems \ref{thm:extensionoperator} and  \ref{C1}, ensure that there exists a constant $C=C(\Omega)$ such that
\begin{equation}\label{feb3}
\|\bfu \|_{L^{n',1}(\Omega)} \leq C \|\bfu\|_{E^{1}L^1(\Omega)}
\end{equation}
for every $\bfu \in E^{1}L^1(\Omega)$, 
and
\begin{equation}\label{feb4}
 \|\bfu\|_{L^{\infty}(\Omega)} \leq \|\bfu\|_{E^{1}L^{n,1}(\Omega)}
\end{equation}
for every $\bfu \in E^{1}L^{n,1}(\Omega)$.
On making us of embeddings \eqref{feb3} and \eqref{feb4} instead of \eqref{feb22} and \eqref{july12}, and of equation \eqref{feb7} instead of \eqref{feb6},  embedding \eqref{jan40} can be deduced from inequality \eqref{july10} via the same argument that has been exploited in Part I to deduce embedding \eqref{july8} from inequality \eqref{july10}.
\end{proof}

\begin{proof}[Proof of Theorem \ref{optomega}] The conclusions follow from Theorem \ref{reduction} and the fact that, by \cite[Proposition 5.2]{KermanPick}, the function norm $\|\cdot\|_{X_1'(0, |\Omega|)}$ satisfying  \eqref{Z} is the optimal 
target norm in inequality \eqref{july10}.
\end{proof} 
%

The next lemma, of use in our proof of Theorem \ref{poincare-omega}, tells us that any domain as in that theorem  satisfies a relative isoperimetric inequality with exponent $\frac 1{n'}$. In the statement, $P(E; \Omega)$ denotes the perimeter in the sense of De Giorgi of a set $E$ relative to $\Omega$, that can be defined as the total variation of the function $\chi_E$ in $\Omega$.

\begin{lemma}\label{isop} Let  $\Omega$ be a connected   $(\ve, \delta)$-domain
 in $\R^n$ with $|\Omega|< \infty$. Then there exists a positive constant $c$ such that
\begin{equation}\label{isop1}
c |E|^{\frac 1{n'}} \leq  P(E; \Omega)
\end{equation}
for every measurable set $E\subset \Omega$ such that $|E|\leq \frac 12 |\Omega|$.
\end{lemma}

\begin{proof}
 By  \cite[Lemma 5.2.3/2 and Corollary 5.2.3]{Mazya}, inequality \eqref{isop1} will follow if we show that
\begin{equation}\label{apr35}
\|u\|_{L^{n'}(\Omega)}\leq c(\|\nabla u\|_{L^1(\Omega)} + \|u\|_{L^1(G)})
\end{equation}
for some open set $G$ such that $\overline G \subset \Omega$, for some constant $c$ and for every weakly differentiable function $u$ such that $\nabla u \in L^1(\Omega)$. Since, by \cite[Theorem 1]{Jo}, the set $\Omega$ is an extension domain for $W^{1,1}(\Omega)$, inequality \eqref{apr35} certainly holds with $G$ replaced by $\Omega$. Namely, 
\begin{equation}\label{apr36}
\|u\|_{L^{n'}(\Omega)}\leq c'(\|\nabla u\|_{L^1(\Omega)} + \|u\|_{L^1(\Omega)})
\end{equation}
for some constant $c'$ and for every function $u \in  W^{1,1}(\Omega)$. We claim that \eqref{apr36} implies \eqref{apr35}. Note that it suffices to prove that inequality  \eqref{apr35} holds under the additional assumption that $u \in L^\infty(\Omega)$. The conclusion for general $u$ then follows by truncating $u$ at levels $-t$ and $t$, letting $t\to \infty$, and making use of the monotone convergence theorem. Now, choose $G$ in \eqref{apr35} in such a way that $c'|\Omega \setminus G|^{\frac 1{n}}\leq \frac 12$, and fix any function $\rho \in C^\infty_0(\Omega)$ such that $0\leq \rho \leq 1$ and $\rho =1$ in $G$. Assume that $\nabla u \in L^1(\Omega)$ and $u\in L^\infty(\Omega)$. One has that,
\begin{equation}\label{apr37}
\|u\|_{L^1(\Omega)}\leq \|\rho u\|_{L^1(G)}+ \|(1-\rho)u\|_{L^1(\Omega\setminus G)} \leq  \| u\|_{L^1(G)}+ \|u\|_{L^{n'}(\Omega\setminus G)} |\Omega \setminus G|^{\frac 1{n}}\leq   \| u\|_{L^1(G)}+ \frac 1{2c'}\|u\|_{L^{n'}(\Omega)}.
\end{equation}
Note that the norm $\|u\|_{L^{n'}(\Omega)}$ is certainly finite, since we are assuming that $u \in L^\infty(\Omega)$. Coupling inequality \eqref{apr36} with \eqref{apr37}, and absorbing the term $\frac 1{2c'}\|u\|_{L^{n'}(\Omega)}$ on the left-hand side yields \eqref{apr35}. 
\end{proof}

\begin{proof}[Proof of Theorem \ref{poincare-omega}]
To begin with, we claim that, if a sequence $\{\bfu_k\} \subset E^1X(\Omega)$ is bounded in $E^1X(\Omega)$, then there exists a function $\bfu \in \mathcal M(\Omega)$ and a subsequence, still denoted by $\{\bfu_k\}$, such that $\bfu_k \to \bfu$ a.e. in $\Omega$. Indeed, owing to the second embedding in \eqref{l1linf}, the sequence is also bounded in $E^1L^1(\Omega)$.  Our claim  then follows on choosing a  countable covering of $\Omega$ by open balls $B_j\subset \Omega$,  exploiting the compactness of the embedding  $E^1L^1(B_j)\to L^1(B_j)$ for each $j$ to extract a convergent subsequence for each $j$, and then employing a diagonal argument -- see \cite[Proof of Lemma 5.5]{Slav2} for an implementation of this argument in the case of Sobolev spaces for the full gradient.
\\ Next, we show that the embedding 
\begin{equation}\label{compact}
E^1X(\Omega) \to X(\Omega)
\end{equation}
is compact. Assume first that $X(\Omega) \neq L^\infty(\Omega)$. The compactness of embedding \eqref{compact} relies upon the fact that inequalities  \eqref{july10} and  \eqref{jan40} hold with the spaces $Y(0, |\Omega|)=X_1(0, |\Omega|)$ and $Y(\Omega)=X_1(\Omega)$ built upon the function norm defined by \eqref{Z}, and its proof makes use of a characterization of compact embeddings from \cite{Slav1}.  Specifically, one can verify that, for every $L>0$,
\begin{equation*}
\bigg\|\int_s^Lf (r)r^{-1+\frac 1{n}}\dr \bigg\|_{L^1(0, L)} \leq  L^{\frac 1n} \|f\|_{L^1(0, L)}\quad  \text{and}\quad \bigg\|\int_s^Lf (r)r^{-1+\frac 1{n}}\dr \bigg\|_{L^\infty(0, L)} \leq  nL^{\frac 1n} \|f\|_{L^\infty(0, L)}
\end{equation*}
for $f \in L^1(0,L)$ and $f \in L^\infty(0,L)$, respectively. Hence, an interpolation theorem by Calder\'on \cite[Theorem 2.12, Chapter 3]{BS} ensures that, if $0<L<|\Omega|$, then
\begin{equation}\label{apr11}
\bigg\|\int_s^Lf (r)r^{-1+\frac 1{n}}\dr \bigg\|_{X_r(0, L)} \leq  n L^{\frac 1n} \|f\|_{X_r(0, L)}
\end{equation}
for $f \in X(0,L)$. As a consequence,
\begin{equation}\label{apr12}
\lim_{L\to 0^+} \sup_{\|f\|_{X_r(0,L)}\leq 1} \left\| \int_r^L f(\varrho) \varrho^{-1+\frac{1}{n}}\,d\varrho\right\|_{X_r(0,L)} =0.
\end{equation}
By \cite[Theorem 4.2]{Slav2}, equation \eqref{apr12} implies that 
\begin{equation}\label{apr13}
X_1(0, |\Omega|) \to X(0, |\Omega|)
\end{equation}
almost compactly, in the sense that $
\lim_{L\to 0^+} \sup_{\|f\|_{X(0, |\Omega|)}\leq 1} \|f^*\|_{(X_1)_r(0,L)}=0$. Next, let $\{\bfu _k\}$ be a bounded sequence in $E^1X(\Omega)$. As shown above, there exist a function $\bfu \in \mathcal M(\Omega)$ and a subsequence still denoted by $\{\bfu_k\}$, such that $\bfu_k \to \bfu$ a.e. in $\Omega$. By the embedding $E^1X(\Omega) \to X_1(\Omega)$, which is guaranteed by Theorem \ref{optomega}, the sequence  $\{\bfu_k\}$ and is bounded in  $X_1(\Omega)$.  Fatou's  lemma for rearrangement-invariant norms tells us that $\bfu \in  X_1(\Omega)$. Therefore, the sequence  $\{\bfu_k-\bfu\}$ is bounded in $X_1(\Omega)$, and $\bfu_k-\bfu \to 0$ a.e. in $\Omega$. By the almost-compact embedding \eqref{apr13}, this implies that $\bfu_k\to \bfu$ in $X(\Omega)$ \cite[Theorem 3.1]{Slav1}. The compactness of embedding \eqref{compact} is thus established when $X(\Omega) \neq L^\infty(\Omega)$. 
\\ Suppose next that 
 $X(\Omega) = L^\infty(\Omega)$. Fix any $p\in (n, \infty)$. We claim that
\begin{equation}\label{apr20}
E^1L^\infty(\Omega) \to W^{1,p}(\Omega).
\end{equation}
To verify this claim, consider the extension operator $\mathscr{E}_\Omega  : E^1L^p(\Omega) \to E^1L^\infty(\rn)$ provided by Theorem \ref{C1}. Since $|\Omega|< \infty$, one has  $E^1L^\infty(\Omega) \to E^1L^p(\Omega)$. Thus, given $\bfu \in E^1L^\infty(\Omega)$, we have that $\bfu \in E^1L^p(\Omega)$,  and there exists a constant $c$ such that $\|\mathscr{E}_\Omega (\bfu) \|_{E^1L^p(\rn)}\leq c \| \bfu \|_{E^1L^p(\Omega)}$ for every $\bfu  \in E^1L^\infty(\Omega)$. Since the space  $C^\infty_0(\rn)$ is dense in $E^1L^p(\rn)$, there exists a sequence of functions $\{\bfv_k\}\subset C^\infty_0(\rn)$ such that $\bfv_k \to \mathscr{E}_\Omega (\bfu)$ in  $E^1L^p(\rn)$. Moreover, by Korn's inequality, there exists a constant $C$ such that
\begin{equation}\label{apr21}
\|\nabla \bfv_k\|_{L^p(\Omega)}\leq \|\nabla \bfv_k\|_{L^p(\rn )}\leq C \|\ep (\bfv_k)\|_{L^p(\rn )}.
\end{equation}
Also,
\begin{equation}\label{apr22}
\lim_{k\to \infty} \|\ep (\bfv_k)\|_{L^p(\rn )}=  \|\ep (\mathscr{E}_\Omega (\bfu))\|_{L^p(\rn )} \leq C  \|\bfu\|_{E^1L^p(\Omega)}
\end{equation}
for some constant $C$.
Thus, the sequence $\{\nabla \bfv_k\}$ is bounded in $L^p(\Omega)$, and hence there exists a subsequence, still indexed by $k$,  and a function $\bfV \in L^p(\Omega)$, such that $\nabla \bfv_k \rightharpoonup \bfV$ weakly in $L^p(\Omega)$. Since $\bfv_k \to \bfu$ in $L^p(\Omega)$, one has that $\bfu \in W^{1,p}(\Omega)$ and $\nabla \bfu = \bfV$. Moreover, owing to equations \eqref{apr21} and \eqref{apr22},
\begin{equation}\label{apr23}
\|\nabla \bfu\|_{L^p(\Omega)}\leq C  \|\bfu\|_{E^1L^p(\Omega)}
\end{equation}
for some constant $C$.
Embedding \eqref{apr20} follows from inequality \eqref{apr23}. On the other hand, the embedding 
\begin{equation}\label{apr24}
W^{1,p}(\Omega) \to L^\infty(\Omega)
\end{equation}
 is compact. This follows, for instance, from \cite[Theorem 7.6]{Slav2}, whose assumption are fulfilled thanks to Lemma \ref{isop}. 
%
The compactness of  embedding  \eqref{compact} with $X(\Omega)=L^\infty(\Omega)$ 
is a consequence of \eqref{apr20} and of the compact embedding \eqref{apr24}.
\\ With the compact embedding \eqref{compact} at disposal, one can conclude via an argument similar to that in  the proof of Lemma \ref{poinc-meas0}. Denote by $\Pi_\Omega$ the bounded linear projection  operator $\Pi_\Omega : E^1X(\Omega) \to \mathcal R$ defined as in \eqref{eq:1803}. 
\\ Let us  show that
\begin{equation}\label{apr14}
\|\bfu - \Pi_\Omega (\bfu)\|_{X(\Omega)} \leq C \|\ep(\bfu)\|_{X(\Omega)}
\end{equation}
for some constant $C$ and for every $\bfu \in E^1X(\Omega)$. Assume,  by contradiction, that there exists a sequence $\{\bfu_k\} \subset E^1X(\Omega)$ such that
\begin{equation}\label{apr15}
\|\bfu_k\textcolor{black}{-\Pi_\Omega (\bfu_k)}\|_{X(\Omega)} \geq k \|\ep(\bfu_k)\|_{X(\Omega)}
\end{equation}
for $k\in \setN$. Without loss of generality, we may also assume that
\begin{equation}\label{apr16}
 \Pi_\Omega (\bfu_k)=0
\end{equation}
and
\begin{equation}\label{apr17}
\|\bfu_k\|_{X(\Omega)}=1
\end{equation}
for $k\in \setN$.
By the compact embedding \eqref{compact}, there exist  function $\bfu \in X(\Omega)$ and a subsequence, still denoted by $\{\bfu_k\}$, such that $\bfu_k \to \bfu$ in $X(\Omega)$. Moreover, equation \eqref{apr15} implies that $\ep(\bfu_k) \to 0$ in $X(\Omega)$. Hence, $\bfu \in E^1X(\Omega)$ and $\ep(\bfu)=0$. Therefore, $\bfu \in \mathcal R$, whence $\Pi_\Omega (\bfu)=\bfu$. Inasmuch as $0=\Pi_\Omega(\bfu_k) \to \Pi (\bfu)=\bfu$, we have that $\bfu=0$. From this piece of information and equation \eqref{apr17} one deduces that
$$1= \lim _{k\to\infty}\|\bfu_k\|_{X(\Omega)} = \|\bfu\|_{X(\Omega)}=0,$$
a contradiction.
\\ Inequalities \eqref{apr14} and \eqref{jan40} yield
\begin{equation}\label{apr18}
\|\bfu - \Pi_\Omega (\bfu)\|_{Y(\Omega)} \leq C \big(\|\ep(\bfu - \Pi_\Omega (\bfu))\|_{X(\Omega)}+ 
\|\bfu - \Pi_\Omega (\bfu)\|_{X(\Omega)}\big)  \leq C' \|\ep(\bfu) \|_{X(\Omega)}
\end{equation}
for some constants $C$ and $C'$ and for every $\bfu \in E^1X(\Omega)$, namely inequality \eqref{feb11}.
\end{proof}

\subsection{Embeddings on $\rn$}\label{ssecrn}

The distinct characterizations of the embeddings for the spaces $E^1_0X(\rn)$ and $E^1X(\rn)$ are stated in Part I and Part II, respectively, of the following theorem.

\begin{theorem}\label{reductionrn}{\rm{\bf [Sobolev embeddings on $\rn$ into rearrangement-invariant spaces] }}
 Let $X(\R^n)$ and $Y(\R^n)$ be rearrangement-invariant spaces. 

\vspace{0.1cm}
\par\noindent 
{\bf Part I [Embeddings for $E^1_0X(\rn)$]} 
The following facts are equivalent:
\\ (i)  The embedding $E^1_0X(\rn) \to Y(\rn)$ holds, namely there 
 exists a constant $c_1$ such that
\begin{equation}\label{feb24}
\|\bfu \|_{Y(\R^n)} \leq c_1 \|\ep (\bfu)\|_{X(\rn)}
%
\end{equation}
for every $\bfu \in E^1_0X(\rn)$.
\\ (ii) The embedding $W^1_0X(\rn) \to Y(\rn)$  holds, namely there exists a constant $c_2$ such that
\begin{equation}\label{feb25}
\|\bfu \|_{Y(\R^n)} \leq c_2 \|\nabla \bfu\|_{X(\rn)}
\end{equation}
 for every $\bfu \in W^1_0X(\rn)$.
\\ (iii) There exists a constant $c_3$ such that
\begin{equation}\label{feb26}
\bigg\|\int_s^{\infty} f (r)r^{-1+\frac 1{n}}\dr \bigg\|_{Y(0, \infty)} \leq  c_3 \|f\|_{X(0, \infty)}
\end{equation}
for every non-increasing function  $f : (0, |\Omega|)\to [0, \infty)$. Moreover,   the constants  $c_1$ and $c_2$ in inequalities \eqref{feb24} and \eqref{feb25} depend only on the constant $c_3$  in \eqref{feb26} and on $n$.

\vspace{0.1cm}
\par\noindent 
 {\bf Part II [Embeddings for $E^1X(\rn)$]}  
The following facts are equivalent:
\\ (i) The embedding $E^1X(\rn) \to Y(\rn)$ holds, namely there  exists a constant $c_1$ such that
\begin{equation}\label{jan47}
\|\bfu \|_{Y(\R^n)} \leq c_1 \|\bfu\|_{E^1X(\rn)}
%
\end{equation}
 for every $\bfu \in E^1X(\rn)$.
\\ (ii) The embedding $W^1X(\rn) \to Y(\rn)$ holds, namely there  exists a constant $c_2$ such that
\begin{equation}\label{jan48}
\|\bfu \|_{Y(\R^n)} \leq c_2 \|\bfu\|_{W^1X(\rn)}
\end{equation}
 for every $\bfu \in W^1X(\R^n)$.
\\ (iii) There exists a constant $c_3$ such that
\begin{equation}\label{jan49}
\bigg\|\chi_{(0,1)}(s)\int_s^{1} f (r)r^{-1+\frac 1{n}}\dr \bigg\|_{Y(0, \infty)} \leq  c_3\|\chi_{(0,1)}f\|_{X(0, \infty)} \quad \text{and} \quad \|\chi _{(1,\infty)} f\|_{Y(0, \infty)} \leq c_3 \|f\|_{X (0,\infty)}
\end{equation}
 for every non-increasing function  $f : (0,\infty)\to [0, \infty)$. 
\\ Moreover, the constants $c_1$ and $c_2$  in inequalities \eqref{jan47} and \eqref{jan48} depend only on the constant $c_3$  in   \eqref{jan49} and on $n$.
\end{theorem}

\begin{remark}\label{increasing1}
{\rm Analogously to inequality  \eqref{july10}  (see Remark \ref{increasing}), one has that inequality \eqref{feb26} holds 
 for every non-increasing   function $f: (0, \infty)\to [0, \infty)$ if and only if it holds for every measurable   function $f: (0, \infty)\to [0, \infty)$. This is shown in \cite[Theorem 1.1]{Pesa}.}
\end{remark}

The optimal rearrangement-invariant target spaces $Y(\rn)$ in inequalities \eqref{feb24} and \eqref{jan47} are characterized in Theorem \ref{optrn} below. 
The function norm $\|\cdot \|_{X_1(0, \infty)}$ that defines the optimal target space in \eqref{feb24} is obtained as in \eqref{Z}, with $|\Omega|$ replaced by $\infty$. 
The function norm associated with the optimal target space in \eqref{jan47} is defined as follows. Given  the function norm $\|\cdot\|_{X(0,\infty)}$, consider the localized function  norm $\|\cdot \|_{X_r(0,1)}$, then build the function norm  $\|\cdot \|_{(X_r)_1(0,1)}$ as in \eqref{Z}, and   extend it back to a function norm in $(0, \infty)$ as $\|\cdot \|_{((X_r)_1)_e(0,\infty)}$. Then the optimal rearrangement-invariant target space in \eqref{jan47} is built upon the function norm $\|\cdot\|_{X_{1, \rn}(0, \infty)}$ given by 
\begin{equation}\label{may15}
\|f\|_{X_{1, \rn}(0, \infty)} = \|f \|_{((X_r)_1)_e(0,\infty)} + \|f\|_{X(0, \infty)}
\end{equation}
for $f \in \Mpl (0,\infty)$. Hence, the corresponding rearrangement-invariant space   on $\rn$ is
\begin{equation}\label{jan78}
X_{1, \rn} (\rn) = ((X_r)_1)_e (\rn)\cap X (\rn).
\end{equation}
Roughly speaking, the
norm in the optimal target space $X_{1, \rn} (\rn) $  behaves locally like  the optimal
 target norm for embeddings of the space $E^1X(B)$ on  a ball $B$,   and like the norm of  $X(\rn)$ itself near infinity.

\begin{theorem}\label{optrn}{\rm{\bf [Optimal rearrangement-invariant target in Sobolev embeddings on $\rn$] }}
 Let $X(\rn)$  be a rearrangement-invariant space.

\vspace{0.1cm}
\par\noindent 
 {\bf Part I [Embeddings for $E^1_0X(\rn)$]} Assume that 
\begin{equation}\label{jan76}
\big\|(1+r)^{-\frac 1{n'}}\big\|_{X '(0, \infty)} < \infty.
\end{equation}
Let $X_1(\rn)$ be the rearrangement-invariant space defined via the function norm \eqref{Z}, with $|\Omega|$ replaced by $\infty$. Then  $E^1_0X(\rn) \to X_1(\rn)$, namely there 
 exists a constant $c$ such that
\begin{equation}\label{feb27}
\|\bfu\|_{X_1(\rn)} \leq c \|\ep (\bfu)\|_{X(\rn)}
%
\end{equation}
for every $\bfu \in E^1_0X(\rn)$. Moreover,
 $X_1(\rn)$ is the optimal (smallest possible) rearrangement-invariant  target space in \eqref{feb27}. 
Assumption \eqref{jan76} is indispensable, in the sense that, if it is dropped, then inequality \eqref{feb24} fails for every rearrangement-invariant space $Y(\rn)$.
%

\vspace{0.1cm}
\par\noindent 
 {\bf Part II [Embeddings for $E^1X(\rn)$}
 Let $X_{1, {\rn}} (\rn) $ be the rearrangement-invariant space given by \eqref{jan78}. 
  Then $E^1X(\rn) \to X_{1, {\rn}} (\rn)$, namely there 
 exists a constant $c$ such that
\begin{equation}\label{jan79}
\|\bfu\|_{X_{1, {\rn}}(\rn)} \leq c \|\bfu\|_{E^1X(\rn)}
\end{equation}
for every $\bfu \in E^1X(\rn)$. Moreover,
 $X_{1, {\rn}}(\rn)$ is the optimal (smallest possible) rearrangement-invariant  target space in \eqref{jan79}.
\end{theorem}

Let us warn that the proof of Part I of Theorem \ref{optrn} makes use of Part I of Theorem  \ref{reductionrn}, which has an independent proof, whereas the proof of Part II of Theorem  \ref{reductionrn} makes use of Part II of Theorem \ref{optrn}, which has an independent proof.

{\color{black}
\begin{remark}\label{linfrn}
{\rm It follows from Theorem \ref{optrn} that, if the function norm $\|\cdot\|_{X(0,\infty)}$ fulfills condition \eqref{apr40} locally, in the sense that $\|\chi_{(0,1)}(r)r^{-\frac 1{n'}}\|_{X'(0,\infty)}<\infty$, then both $E^1_0X(\rn) $ and $E^1X(\rn)$ are continuously embedded  into $L^\infty(\rn)$. (Of course, the first embedding requires that assumption \eqref{jan76} be satisfied as well.) In particular, Part II tells us that 
$$E^1X(\rn) \to L^\infty(\rn)\cap X(\rn),$$
the target space being optimal among all rearrangement-invariant spaces.
}
\end{remark}}

\begin{proof}[Proof of Theorem \ref{optrn}] \emph{Part} I. Inequality \eqref{feb27}, and the optimality of the space $X_1(\rn)$, are a consequence of Theorem \ref{reductionrn}, and of the fact, established in \cite[Theorem 4.4]{EMMP},  that inequality \eqref{feb26} holds  with $Y(0,\infty)=X_1(0,\infty)$, the latter space being optimal among all rearrangement-invariant spaces. The necessity of condition \eqref{jan76} for an inequality of the form \eqref{feb26}   for any rearrangement-invariant space $Y(0,\infty)$,
is a consequence of a necessary condition for the validity of Hardy type inequalities in general rearrangement-invariant spaces, which, for instance, follows by duality from  \cite[Lemma 1]{EGP}. As a consequence of Theorem \ref{reductionrn}, condition \eqref{jan76} is also necessary 
 for inequality \eqref{feb24} to hold for any rearrangement-invariant space $Y(\Omega)$.

\smallskip
\par\noindent
  \emph{Part} II.
Assume that the function $\bfu \in E^1X(\R^n)$ satisfies 
\begin{equation}\label{jan60}
\|\bfu\|_{E^1X(\R^n)} \leq 1.
\end{equation}
Let   $\lambda_0>0$ be so large that
$$\lim_{t \to \infty} \varphi _X(t) > \frac {2C_M}{\lambda_0},$$
where $C_M$ is the constant appearing in conditions \eqref{weakX} and  \eqref{feb20}, and 
\begin{equation}\label{jan61}
\varphi^{-1}_X( 2c_0/{\lambda_0})\leq 1, 
\end{equation}
where $c_0$ is a constant depending on $n$ and $X(\rn)$ to be chosen later, and
$\varphi_X^{-1}$ denotes
the (generalized left-continuous) inverse of the fundamental function $\varphi_X$ of $X(\rn)$. Note that   $\lambda_0$ certainly exists, inasmuch as
\begin{equation*}
\lim _{t \to 0^+} \varphi_X^{-1} (t) =0.
\end{equation*}
Consider the decomposition $\bfu = \bfu_0 + \bfu_1$, where $\bfu_0 = T^{\lambda_0,\lambda_0}\bfu$.  Here,  $T^{\lambda_0,\lambda_0}$ denotes the operator defined as in \eqref{eq:wlambda}. Since $\bfu , \bfu_0 \in E^1X(\rn)$ and $\bfu _0= \bfu$ in $\rn \setminus \mathcal O_{\lambda_0,\lambda_0}$,  we have that $\bfu_1 \in E^1_0X_r(\mathcal O_{\lambda_0,\lambda_0})$ and, by inequality \eqref{jan50},
\begin{equation}\label{jan62}
\|\ep (\bfu_1)\|_{X_r(\mathcal O_{\lambda_0,\lambda_0})} \textcolor{blue}{\leq} \|\ep (\bfu_1)\|_{X(\rn)} =  \|\ep (\bfu)- \ep (\bfu_0)\|_{X(\rn)} \leq c \|\ep (\bfu)\|_{E^1X(\rn)} \leq c
\end{equation}
for some constant $c=c(n,X(\rn))$. On the other hand, by equations \eqref{jan51} and \eqref{weakX},
\begin{align}\label{jan63}
\lambda_0 \varphi_X(|\mathcal O_{\lambda_0, \lambda_0}|) & \leq     2\lambda _0(\varphi_X(|\set{M(\bfu)>\lambda_0}|)+ \varphi_X(|\set{M(\ep(\bfu))>\lambda_0}|) \\ \nonumber & \leq 2c \big(\|\bfu\|_{X(\rn)} + \|\ep (\bfu)\|_{X(\rn)}\big) = 2c \|\bfu \|_{E^1X(\rn)}\leq 2c
\end{align}
for some constant  $c=c(n,X(\rn))$.
Since \begin{equation*}
t \leq \varphi_X^{-1} (\varphi _X (t))  \qquad \hbox{for $t>0$,}
\end{equation*}
equations \eqref{jan63} and \eqref{jan61} ensure that
\begin{equation}\label{jan64}
|\mathcal O_{\lambda_0, \lambda_0}|\leq \varphi^{-1}_X( 2c_0/{\lambda_0})\leq 1,
\end{equation}
provided that $c_0$ does not exceed the constant $c$ appearing in  \eqref{jan63}. An application of Theorem \ref{optomega} with $\Omega= \mathcal O_{\lambda_0, \lambda_0}$ and inequality \eqref{jan62} then tell us that
\begin{equation}\label{jan80}
\|\bfu_1\|_{((X_r)_1)_e(\rn)} = \|\bfu_1\|_{(X_r)_1( \mathcal O_{\lambda_0, \lambda_0})} \leq c \|\ep (\bfu_1)\|_{X_r(\mathcal O_{\lambda_0, \lambda_0})}\leq c' \|\ep (\bfu)\|_{X_r(\mathcal O_{\lambda_0, \lambda_0})} \leq c'
\end{equation}
for some constants $c$ and $c'$ depending on $n$ and $X(\rn)$.
\\ On the other hand, we deduce from property \eqref{itm:TlestLinfty2} of Theorem \ref{lem:Tlest} that
\begin{equation}\label{jan81}
\|\bfu_0\|_{((X_r)_1)_e(\rn)}\leq c \|\lambda_0\|_{((X_r)_1)_e(\rn)} = c'
\end{equation}
for some constants $c=c(n)$ and  $c'=c'(n, X(\rn))$. Equations \eqref{jan80} and \eqref{jan81} imply that
\begin{equation}\label{jan82}
\|\bfu\|_{((X_r)_1)_e(\rn)}\leq \|\bfu_0\|_{((X_r)_1)_e(\rn)} + \|\bfu_1\|_{((X_r)_1)_e(\rn)}\leq c
\end{equation}
for some constant $c=c(n, X(\rn))$. Hence,
\begin{equation}\label{jan83}
\|\bfu\|_{X_{1,\rn}(\rn)} = \|\bfu\|_{((X_r)_1)_e(\rn)} + \|\bfu\|_{X(\rn)}\leq c + 1,
\end{equation}
where $c$ is the constant on the rightmost side of equation \eqref{jan82}. Inequality \eqref{jan79} is thus established.
\\ The fact that the target space $X_{1,\rn}(\rn)$ is optimal in this inequality among all rearrangement-invariant spaces follows from its optimality in a parallel  inequality, with the space $E^1X(\rn)$ replaced by  $W^1X(\rn)$ -- see \cite[Theorem 3.1]{ACPS}.
\end{proof}

\begin{proof}[Proof of Theorem \ref{reductionrn}] 
\emph{Part} I. Inequality \eqref{feb24} trivially implies \eqref{feb25}, and the latter implies \eqref{feb26}, as shown e.g. in \cite[Theorem 3.3]{mihula}. In order to prove that inequality \eqref{feb26} implies \eqref{feb24}, 
one can make use of equation \eqref{feb6} for   every  function $\bfu \in E^1_bX(\rn)$ and proceed along the same lines as in the proof of inequality \eqref{july18} to deduce that
\begin{equation}\label{feb60}
 \int_0^{t^{n'}} s^{-\frac{1}{n}} \bfu^*(s)\ds
\leq C 
 \bigg(\int _0^{t^{n'}/C} \ep(\bfu)^*(s) \ds + t \int_{t^{n'}/C}^\infty  \ep(\bfu)^*(s) s^{-\frac 1{n'}} \ds\bigg) \quad\text{for $t>0$,}
\end{equation}
for some constant $C=C(n)$.
With equation \eqref{feb60} at our disposal, equations \eqref{july21}--\eqref{july23} continue to hold, with $\Omega$ replaced by $\rn$ and $|\Omega$| replaced by $\infty$, thus establishing inequality \eqref{feb24} for every $\bfu \in E^1_bX(\rn)$. This inequality carries over to any function $\bfu \in  E^1_0X(\rn)$, as shown by the same argument as in the proof of inequality \eqref{july8}.
%
%
%
\\
\emph{Part} II.
Embedding \eqref{jan47} obviously implies \eqref{jan48}. The fact that the latter implies the inequalities in \eqref{jan49} is shown in \cite[Theorem 3.3]{ACPS}. It thus suffices to prove that
the inequalities in  \eqref{jan49} imply embedding \eqref{jan47}. To this purpose, observe that the first inequality in  \eqref{jan49}  is equivalent to 
\begin{equation}\label{jan84}
\bigg\|\int_s^{1} f(r)r^{-1+\frac 1{n}}\dr \bigg\|_{{\color{black} Y_r(0, 1)}} \leq  C\|f\|_{X_r(0, 1)} 
\end{equation}
for every  non-increasing function  $f: (0, 1)\to [0, \infty)$. By the optimality of the target space $(X_r)_1(0,1)$  in inequality \eqref{jan84}, we have that
\begin{equation}\label{jan85}
(X_r)_1(0,1) \to Y_r(0,1).
\end{equation}
Hence, 
\begin{equation}\label{jan86}
\|\chi_{(0,1)}f\|_{Y(0,\infty)} = \|f\|_{Y_r(0,1)} \leq \|f\|_{(X_r)_1(0,1)}
\end{equation}
for every  non-increasing function  $f : (0, 1)\to [0, \infty)$. Now, let $\bfu \in E^1X(\rn)$. From inequality \eqref{jan86} and the second inequality in  \eqref{jan49} we deduce that
\begin{align}\label{jan87}
\|\bfu\|_{X_{1,\rn}(\rn)} & = \|\bfu\|_{((X_r)_1)_e(\rn)} + \|\bfu\|_{X(\rn)} = \|\bfu^*\|_{((X_r)_1)_e(0,\infty)} + \|\bfu^*\|_{X(0,\infty)}
\\ \nonumber  &= \|\bfu^*\|_{(X_r)_1(0,1)} + \|\bfu^*\|_{X(0,\infty)} \geq  \|\chi_{(0,1)} \bfu^*\|_{Y(0,\infty)} + \|\chi_{(1, \infty)} \bfu^*\|_{Y(0,\infty)} \\ \nonumber  & \geq  \| \bfu^*\|_{Y(0,\infty)} = \| \bfu\|_{Y(\rn)}.
\end{align}
Embedding  \eqref{jan47} follows from inequalities  \eqref{jan79} and \eqref{jan87}.
\end{proof}

\section{Sobolev embeddings into spaces of continuous functions}\label{cont}
Here, we deal with symmetric gradient Sobolev spaces built upon 
rearrangement-invariant function norms which are  strong enough for condition \eqref{apr40} to be fulfilled. The results of the preceding section imply that, under this condition, any function from the relevant Sobolev spaces is bounded. The question thus arises of whether any such function is also continuous. Our first results provides us with a positive answer to this question.
	\par
We denote by $C^0(\Omega )$ the space of bounded continuous 
 functions $\bfu : \Omega \to \rn$ endowed with the standard norm $\|\bfu\|_{C^0(\Omega )}= \sup _{x \in \Omega}
 |\bfu(x)|$.

\begin{theorem}\label{Linfty}{\rm{\bf [Sobolev embeddings into $C^0$] }}
Let $\Omega$ be an open set in $\R^n$  with $|\Omega|<\infty$, and let $X(\Omega)$   be a rearrangement-invariant space. 

\vspace{0.1cm}
\par\noindent 
 {\bf Part I [Embeddings for $E^1_0X(\Omega)$]}
 The following facts are equivalent:
\\ (i) The embedding $E^1_0X(\Omega) \to C^0(\Omega)$ holds, namely there exists a constant $c$ such that
\begin{equation}\label{july21mar}
\|\bfu\|_{C^0(\Omega)}\leq c \|\ep(\bfu)\|_{X(\Omega)}
\end{equation}
for every $\bfu \in E^1_0X(\Omega)$.
\\ (ii) The embedding $W^1_0X(\Omega) \to C^0(\Omega)$ holds, namely there exists a constant $c$ such that
\begin{equation}\label{feb78}
\|\bfu\|_{C^0(\Omega)}\leq c \|\nabla \bfu\|_{X(\Omega)}
\end{equation}
for every $\bfu \in W^1_0X(\Omega)$.
\\ (iii) Condition \eqref{apr40} holds.
\\
 (iv) One has that
\begin{equation}\label{feb77}
X(\Omega) \to L^{n,1}(\Omega).
\end{equation}

\vspace{0.1cm}
\par\noindent 
 {\bf Part II [Embeddings for $E^1X(\Omega)$]}
 Assume, in addition, that $\Omega$ is an $(\varepsilon, \delta)$-domain. 
Then the following facts are equivalent:
\\ (i) The embedding $E^1X(\Omega) \to C^0(\Omega)$ holds, namely there exists a constant $c$ such that
\begin{equation}\label{july21'}
\|\bfu\|_{C^0(\Omega)}\leq c \|\bfu\|_{E^1X(\Omega)}
\end{equation}
for every $\bfu \in E^1X(\Omega)$.
\\ (ii) The embedding $W^1X(\Omega) \to C^0(\Omega)$ holds, namely there exists a constant $c$ such that
\begin{equation}\label{feb78'}
\|\bfu\|_{C^0(\Omega)}\leq c \|\bfu\|_{W^1X(\Omega)}
\end{equation}
for every $\bfu \in W^1X(\Omega)$.
\\ (iii) Condition \eqref{apr40} holds.
\\ (iv) Embedding \eqref{feb77} holds.
\end{theorem}

In the light of Theorem \ref{Linfty}, as  a next step we  investigate  the optimal modulus of continuity of functions  from symmetric gradient Sobolev spaces as in that theorem. It turns out that an estimate for the modulus of continuity, which is uniform for all these functions, is only possible under a slightly stronger assumption than \eqref{apr40}. Under this strengthened assumption, the optimal modulus of continuity is exhibited in the main result of this section.
%
%
{\color{black} Since continuity is a local property, we shall focus on bounded domains $\Omega$.}
\par
 Let us premise a few definitions and notations. A function   $\sigma : (0, \infty ) \to (0,
\infty )$ is said to be a modulus of continuity if it is equivalent  (up to multiplicative
constants) near $0$  to a non-decreasing function,  and
$\lim _{s\to 0^+}\sigma (s)=0$ and  $\limsup _{s\to 0^+}\frac
s{\sigma (s)} < \infty $.
%
%
\\ We denote by $C^\sigma (\Omega )$ the Banach space of all functions $\bfu : \Omega \to \rn$
for which the norm
\begin{align}\label{sobolevspacebis}
\|\bfu\|_{C^\sigma (\Omega )} = \|\bfu\|_{C^0(\Omega )} + \sup_{
\begin{tiny}
                    \begin{array}{c}
                       x,y \in \Omega \\
                        x \neq y
                      \end{array}
                      \end{tiny}}
 \frac{|\bfu(x)-\bfu(y)|}{\sigma (|x-y|)}
 \end{align}
 is finite.  Note that moduli of continuity, which are equivalent  (up to multiplicative constants) near $0$, yield
 the same spaces (up to equivalent norms).
\par
Assume that   $\|\cdot\|_{X(0, |\Omega|)}$ is a rearrangement-invariant function norm satisfying condition \eqref{apr40}.  Then the function 
 $\vartheta _{X}: (0, \infty ) \to (0, \infty )$, given by
\begin{equation}\label{july25}
\vartheta _{X}(s) = \| r^{-\frac 1{n'}}\chi _{(0,
s^n)}(r)\|_{{X_e} '(0, \infty)} \quad \text{for $s>0$, }
\end{equation}
is well defined.
Moreover, given any number {\color{black} $R>{\rm diam}(\Omega )^n$},
the function
 $\varrho _{X} : (0, \infty )
\to (0, \infty)$ given  by
\begin{equation}\label{july24}
\varrho _{X}(s)= s\,\| r^{-1}\chi _{(s^n ,
R)}(r)\|_{{X_e} '(0, \infty)} \quad \text{for $s \in (0, {\rm diam}(\Omega )]$,}
\end{equation}
and continued by $\varrho
_{X}({\rm diam}(\Omega ))$ for $s >{\rm diam}(\Omega )$, is also well defined. One can show that different choices of
 $R$,
  yield functions  $\varrho _{X}$
  which are mutually equivalent near $0$, up to multiplicative constants.
\\
Define
\begin{equation}\label{july23mar} \sigma _{X} =
 \vartheta _{X}+ \varrho _{X} .
\end{equation}
If 
\begin{equation}\label{limitesigma}
\lim _{s \to 0^+}\sigma _{X}(s)=0,
\end{equation}
then the function $\sigma _{X}$ is a modulus of continuity. This assertion can be verified via the argument employed  in \cite[Proof of Theorem 3.4]{cianchi-randolfi IUMJ}.
\par
The following theorem    tells us that $\sigma _{X}$ is the optimal modulus of continuity announced above. As mentioned in Section \ref{intro}, in sharp contrast with the case of embeddings into rearrangement-invariant  target spaces, the optimal target space $C^\sigma (\Omega )$ in a  Sobolev embedding  for symmetric gradients may be larger than the optimal one for  full gradients,  associated with the same function norm $\|\cdot\|_{X(0, |\Omega|)}$. As shown in \cite[Theorem 3.4]{cianchi-randolfi IUMJ}, the optimal target space $C^\sigma (\Omega )$ in the latter embedding is obtained with the choice $\sigma = \varrho _{X}$. Moreover, there  do exist function norms $\|\cdot\|_{X(0, |\Omega|)}$ such that $ \vartheta _{X}$ is not equivalent to $\varrho _{X}$, and hence $C^{ \varrho _{X}}(\Omega) \subsetneq C^{ \sigma _{X}}(\Omega)$ -- see Examples \ref{Linf} and \ref{exp} in the next  section.

\begin{theorem}\label{continuity}{\rm{\bf [Sobolev embeddings  into  spaces of uniformly continuous functions] }}
Let $\Omega$ be a   bounded open set in $\R^n$ and  let $X(\Omega)$ be a rearrangement-invariant space fulfilling condition \eqref{limitesigma}.

\vspace{0.1cm}
\par\noindent 
 {\bf Part I [Embedding for $E^1_0X(\Omega)$]}
 The embedding $E^1_0X(\Omega) \to C^{\sigma _{X}}(\Omega )$ holds, namely there exists a constant $c$ such that
\begin{equation}\label{cont1}
\|\bfu\|_{C^{\sigma _{X}}(\Omega )} \leq c \|\ep(\bfu)\|_{X(\Omega)}
%
\end{equation}
for every $\bfu \in E^1_0X(\Omega)$.

\vspace{0.1cm}
\par\noindent 
 {\bf Part II [Embedding for $E^1X(\Omega)$]}
Assume, in addition, that $\Omega$ is an     $(\varepsilon, \delta)$-domain. Then $E^1X(\Omega) \to C^{\sigma _{X}}(\Omega )$, namely 
there exists a constant $c$ such that
\begin{equation}\label{cont1'}
\|\bfu\|_{C^{\sigma _{X}}(\Omega )} \leq c \|\bfu\|_{E^1X(\Omega)}
%
\end{equation}
for every $\bfu \in E^1X(\Omega)$. 

\vspace{0.1cm}
\par\noindent 
The result is sharp, in the sense that if there exists a modulus
of continuity $\sigma$ such that  either inequality \eqref{cont1} or \eqref{cont1'} holds, with  $C^{\sigma _{X}}(\Omega )$ replaced by $ C^{\sigma
}(\Omega )$,
then \eqref{limitesigma}
holds, and $C^{\sigma
_{X}}(\Omega ) \to C^{\sigma}(\Omega )$.
\end{theorem}

\begin{remark}\label{notunif}
 {\rm Let us notice that there actually exist rearrangement-invariant function norms $\|\cdot\|_{X(0, |\Omega|)}$  fulfilling condition \eqref{apr40}, for which however \eqref{limitesigma} fails. This is the case, for instance, for the borderline space $L^{n,1}(\Omega)$. Thus, by Theorem \ref{Linfty},
 $E^1_0L^{n,1}(\Omega) \to C^0(\Omega)$, whereas, by Theorem \ref{continuity}, the space $E^1_0L^{n,1}(\Omega)$ is not continuously embedded into  $C^\sigma(\Omega)$ for any
   modulus of continuity $\sigma$.}
\end{remark}

\begin{proof}[Proof of Theorem \ref{Linfty}] \emph{Part} I. The equivalence of properties (ii), (iii) and (iv) is shown in \cite[Theorem 3.1]{cianchi-randolfi IUMJ}. Property (i) plainly implies (ii). In order to prove that property (iv) implies (i), recall, from   Corollary \ref{Linfa}, that
\begin{equation}\label{feb80mar}
\|\bfu\|_{L^\infty(\Omega)} \leq C \|\ep(\bfu)\|_{L^{n,1}(\Omega)}
\end{equation} 
for every $\bfu \in E^1_0L^{n,1}(\Omega)$.  We have thus just to show that every function $\bfu \in E^1_0L^{n,1}(\Omega)$ is continuous. Since $\Omega$ is bounded, we have that $E^1_0L^{n,1}(\Omega) = E^1_bL^{n,1}(\Omega)$. Let $B_\Omega$ be a ball such that $\overline \Omega \subset  B_\Omega$. Thus, if $\bfu \in E^1_0L^{n,1}(\Omega)$, then the function $\bfu_e$ is compactly supported in $B_\Omega$. {\color{black} A standard approximation argument
shows that $\bfu_e$ can be approximated  in $ E^1_0L^{n,1}(B_\Omega)$ by a sequence of functions  $\{\bfv_k\}\subset C^\infty_0(B_\Omega)$.} Passing to the limit as $k \to \infty$ in an analogue of inequality \eqref{feb80mar}, with $\Omega$ replaced by $B_\Omega$ and $\bfu$ replaced bu $\bfu_e - \bfv_k$,  tells us that $\bfu_e$, and hence $\bfu$, is continuous.
%
%
\\   \emph{Part} II.  The proof is analogous.   Notice that now we have to approximate $\bfu$   by a sequence $\{\bfu_k\}\subset C^\infty(\Omega) \cap E^1L^{n,1}(\Omega)$.
%
\end{proof}

\begin{proof}[Proof of Theorem \ref{continuity}] \emph{Part} I.
Since $\Omega$ is bounded, $ E^1_0X(\Omega)= E^1_bX(\Omega)$. Thus, if $\bfu \in E^1_0X(\Omega)$, then $\bfu_e \in E^1X_e(\rn)$ and has bounded support in $\rn$. In particular, by property (P5) of the definition of  function norm,  $\bfu_e \in E^1_bL^1(\rn)$ as well. 
Let us denote $\bfu_e$ simply by $\bfu$ and $X_e(\rn)$ by $X(\rn)$ in the remaining part of this proof.
By Lemma \ref{reprrn}, we have that
\begin{equation}\label{july26}
\bfu (x) = \int _{\R^n} K(x,y) \ep (\bfu)(y)\dy \quad \hbox{for a.e. $x \in \R^n$,}
\end{equation}
where the kernel $K: \rn \times \rn \to \R^{n\times n}$ satisfies the inequality 
$$ |K(x,y)|\leq c |x-y|^{1-n} \quad \hbox{for $x \neq y$},$$
for some constant $c$. Also, one can verify that
{\color{black} $$ |K (x+h,y)-K(x,y)|\leq c'
|h||x-y|^{-n} \quad \hbox{for $h \in \rn$ and $|x-y|\geq
3|h|$,}$$}
for some constant $c'$.
Fix any $x, h \in \rn$.  Therefore,
 \begin{align}\label{july27}
\quad |\bfu(x+h)-\bfu(x)| &\leq 
\left| \int_{\{y:\,|x-y|< 3|h|\}}(K
(x+h,y)-K (x,y))\ep(\bfu)(y)\dy\right|\\ \nonumber & 
\quad +\,\left| \int_{\{y:\,|x-y|\geq 3|h|\}}(K
(x+h,y)-K (x,y))\ep(\bfu)(y)\dy\right|\nonumber \\
\nonumber &\leq  c\int_{\{y:\,|x-y|< 4|h|\}}\frac{|\ep(\bfu)(y+h)|}{|x-y|^{n-1}}\dy + c\int_{\{y:\,|x-y|< 3|h|\}}\frac{|\ep(\bfu)(y)|}{|x-y|^{n-1}}\dy\\ \nonumber &  \quad +\, c'\,|h|\int_{\{y:\,|x-y|\geq
3|h|\}}\frac{|\ep(\bfu)(y)|}{|x-y|^{n}}\,\chi_{\Omega}(y)\dy\nonumber
\\
& \leq 2c\,\|\ep(\bfu)\|_{X(\R^n)}\left\|\frac{1}{|x-y|^{n-1}}\:\chi_{\{y:\, |x-y|<
              4|h|\}}(y)\right\|_{X^{'}(\R^n)}\nonumber\\
              & +\,c'\,|h|\,\|\ep(\bfu)\|_{X(\R^n)}\left\|\frac{\chi_{\Omega}(y)}{|x-y|^{n}}\:\chi_{\{y:\,
              |x-y|\geq 3|h|\}}(y)\right\|_{X^{'}(\R^n)}.\nonumber
\end{align}
Assume, for a moment, that
\begin{equation}\label{july28}
|h|^n \leq |\Omega|.
\end{equation}
Given any measurable set $E\subset \rn$, define $\widehat E$ as the annulus
$$\widehat E = \{y \in \rn : 3|h| \leq |x-y| \leq r(E)\},$$
where $r(E)$ is such that $|\widehat E| = |E|$. In particular, if
$E=\Omega$, then, by \eqref{july28},
\begin{equation}\label{july29}
\omega _n r(\Omega)^n = |\Omega| + \omega _n 3^n |h|^n \leq (1+ \omega _n
3^n)|\Omega|.
\end{equation}
Define the functions $v, w : \rn \to [0, \infty)$ as
\begin{equation*}
v(y)=\frac{\chi_{\Omega}(y)}{|x-y|^{n}}\:\chi_{\{y:\,
              |x-y|\geq 3|h|\}}(y) \quad \text{and} \quad 
              w(y)= \frac{\chi_{\widehat \Omega}(y)}{|x-y|^{n}}\:\chi_{\{y:\,
              |x-y|\geq 3|h|\}}(y)
 \quad \text{for $y \in \rn$.}
\end{equation*}
Note that, by the definition of $\widehat \Omega$,
$$w(y) =   \frac{1}{|x-y|^{n}}\:\chi_{\{y:\,
              3|h| \leq |x-y| \leq r(\Omega) \}}(y)  \quad \text{for $y \in \rn$.}
$$ 
Furthermore,
for any measurable set $E \subset \Omega$,
$$v = w \,\, \hbox{in  $E \cap \widehat  E$,}\qquad \min _{\widehat E \setminus
 E} w \geq \max _{E \setminus \widehat  E} v, \quad \text{and} \quad |\widehat  E \setminus
 E| = |E \setminus \widehat  E|. $$ Thus, for any set $E$ of this kind,
 \begin{eqnarray}\label{july30}
\int _E v(y)\dy =   \int _{E \cap {\widehat  E}} v(y)\dy +
\int _{E \setminus {\widehat  E}} v(y)\dy
\leq \int _{E \cap \widehat  E} w(y)\dy + \int _{ \widehat E
\setminus E} w(y)\dy = \int _{ \widehat  E} w(y)\dy\,. 
\end{eqnarray}
The inequality $ \int _E v(y)\dy \leq  \int _{ \widehat  E}
w(y)\dy$
 clearly continues to hold even if
$E\nsubseteq \Omega$, since $v=0$ in $E \setminus \Omega$. Hence, 
\begin{equation}\label{july31}
v^{**}(s) \leq w^{**}(s) \quad \hbox{for $s>0$}
\end{equation}
owing  to property \eqref{c15bis}. By inequality \eqref{july31}, {\color{black} property \eqref{hardy},} and equations \eqref{july28} and \eqref{july29}, there
exists a constant $C=C(n)$ such that
 \begin{align}\label{july32}
& \left\|\frac{\chi_{\Omega}(y)}{|x-y|^{n}}\: \chi _{\{y:\,
               3|h|\leq |x-y|\}}(y)\right\|_{X^{'}(\R^n)}   \leq
              \left\|\frac{1}{|x-y|^{n}}\:\chi_{\{y:\,
               3|h|\leq |x-y| \leq r(\Omega)\}}(y)\right\|_{X^{'}(\R^n)}
 \\ \nonumber & \quad \quad  = \omega _n\left\|r^{-1}\chi_{(\omega_n 3^n|h|^n, \,\,\omega
_n r(\Omega)^n)}(r)\right\|_{{X}'(0,\infty)}
  \leq  \omega _n \left\|r^{-1}\chi_{(\omega_n 3^n|h|^n,\,\, C
|\Omega|)}(r)\right\|_{{X}'(0,\infty)}.
\end{align}
On the other hand,
\begin{equation}\label{july33}
\left\|\frac{1}{|x-y|^{n-1}}\:\chi_{\{y:\, |x-y|\leq
              4|h|\}}(y)\right\|_{X^{'}(\R^n)}
= \omega _n^{\frac{1}{n'}}
\left\|r^{-\frac{1}{n'}}\,\chi_{(0,\,\, \omega_n
4^n|h|^n)}(r)\right\|_{{X}'(0,\infty)}.
\end{equation}
Combining  equations \eqref{july27}, \eqref{july32} and \eqref{july33} tell us that
\begin{eqnarray}\label{july34}
  |\bfu(x+h)-\bfu(x)|& \leq &C\,\|\ep(\bfu)\|_{X(\R^n)}\left\|r^{-\frac{1}{n'}}\,\chi_{(0, \omega_n 4^n|h|^n)}(r)\right\|_{{X}'(0,\infty)}\\ \nonumber
  & & +\, C\,|h|\,\|\ep(\bfu)\|_{X(\R^n)}\left\|r^{-1}\chi_{(\omega_n 3^n|h|^n,
  C
  |\Omega|)}(r)\right\|_{{X}'(0,\infty)}
  \end{eqnarray}
for some constant $C=C(n)$.
Hence, 
\begin{equation}\label{july35}
  |\bfu(x+h)-\bfu(x)|\leq
c\,\|\ep(\bfu)\|_{X(\R^n)}\left(\vartheta_{X}(|h|)+\varrho_{X}(|h|)\right)
\end{equation}
for some constant $c$, independent of $\bfu$, provided that $|h|$ is
sufficiently small, independently of $\bfu$. Inequality \eqref{july35}
clearly continues to hold, for a suitable constant $c$, even if assumption \eqref{july28} is dropped,
for any $h$ such that $x, x+h \in \Omega $. Embedding \eqref{cont1} is thus established.
\\
In order to prove the necessity of condition \eqref{limitesigma} and the optimality of the space $C^{\sigma_X}(\Omega)$ in inequality \eqref{cont1}, assume that $E^1_0X(\Omega ) \to C^\sigma (\Omega)$ for some
modulus of continuity $\sigma$.
Thus,  there
exists a constant $c$ such that
\begin{eqnarray}\label{c18}
\mathop{\sup}\limits_{\tiny{\begin{array}{c}\scriptstyle x,y\in \Omega\\
\scriptstyle x\neq
y\end{array}}}\frac{\left|\bfu(x)-\bfu(y)\right|}{\sigma (|x-y|)}
\leq  c\,\|\ep(\bfu)\|_{X(\Omega)}
\end{eqnarray}
for every $\bfu \in E^1_0X(\Omega)$. We may suppose,  without loss of
generality, that $B\subset\Omega$,  where $B$ is the ball centered at  the origin, with $|B|=1$.
Given a nonnegative function $f \in L^\infty (0,\infty)$,
vanishing outside $(0, 1)$, let $\bfu: \Omega \to \R^n$ be the function
defined as
\begin{equation}\nonumber
\bfu(x) =
\begin{cases}
\Big(\int_{\omega_n |x|^n}^1f(r)
r^{-\frac{1}{n'}}\dr, 0, \dots , 0 \Big)
 & \hbox{if \,$x\in B$}
\\
 (0, \dots , 0) &\hbox{if\, $x\in\Omega \setminus B$.}
\end{cases}
\end{equation}
One can verify that $\bfu \in W^1_0X(\Omega)\subset E^1_0L^1(\Omega)$,   and that $|\nabla \bfu (x)|= c f(\omega_n |x|^n)$ a.e. in $\Omega$ for a suitable constant $c=c(n)$.  Hence, $\ep(\bfu)^*(s) \leq |\nabla \bfu|^* (s) \leq c f^*(s)$ for $s\geq 0$, and
\begin{equation}\label{c7}
\|\ep(\bfu)\|_{X(\Omega)}\leq c \|f\|_{{X}(0,\infty)}.
\end{equation}
Inequalities  \eqref{c18} and \eqref{c7} enable one  to deduce that
\begin{eqnarray} \label{1101}
c&\geq& \mathop{\sup}\limits_{\bfu\in\,
E^1_0X(\Omega)}\frac{\left|\bfu(x)- \bfu(0)\right|}{\sigma
(|x|)\|\ep(\bfu)\|_{X(\Omega)}}
\geq c'\mathop{\sup}\limits_{\begin{tiny}
                      \begin{array}{c}
                       {f\in L^\infty (0, \infty ) }\\
                        f=0 \,{\rm in} \,(\omega_n|x|^n , \infty )
                      \end{array}
                      \end{tiny}}\frac{\int_0^{\omega_n|x|^n}f(r)\,r^{-\frac{1}{n'}}\dr}{\sigma
(|x|)\,\|f\|_{{X} (0,\infty)}}
\\ \nonumber &= &
 c'\mathop{\sup}\limits_{\begin{tiny}
                      \begin{array}{c}
                       {f\in {X} (0, \infty ) }\\
                        f=0 \,{\rm in} \,(\omega_n|x|^n , \infty )
                      \end{array}
                      \end{tiny}}\frac{\int_0^{\omega_n|x|^n}f(r)\,r^{-\frac{1}{n'}}\dr}{\sigma
(|x|)\,\|f\|_{{X}(0,\infty)}}
=c'\frac{\|r^{-\frac{1}{n'}}\,\chi_{(0,
\omega_n\,|x|^n)}(r)\|_{{X} '(0,\infty)}}{ \sigma
(|x|)}
\geq c''\,\frac{\vartheta_{X}(|x|)}{\sigma
(|x|)}
\end{eqnarray}
for positive constants $c, c', c''$, provided that $|x|$ is sufficiently small. 
Note that the first
equality follows via   property  (P3) of the definition of function 
norm, on replacing any nonnegative unbounded function $f \in
{X}(0,\infty)$, vanishing in $(\omega_n|x|^n , \infty )$,
by $\max \{ f , t\}$, and then letting $t \to \infty$.
\par\noindent Next, given any function $f$ as above and any matrix $\bfQ \in \setR^{n\times n}_{\rm skew}$,
consider  the function $\bfu: \Omega \to \R^n$ given by
$$\bfu(x) =
\begin{cases}
\bfQ \,x \int _{\omega _n |x|^n}^1 f(r) r^{-1}  \dr & \hbox{if $x \in B$} \\
0 & \hbox{if $x \in \Omega \setminus B$.}
\end{cases}$$
One has that $\bfu$ is a weakly
differentiable function, and
\begin{align*}
  \ep(\bfu)(x) &=\frac{\bfQ x\otimes ^{\rm sym} x}{|x|^2}n f(\omega _n |x|^n) \quad \text{for a.e. $x \in B$. }
\end{align*}
Hence,
 $ \abs{\ep(\bfu)(x)} \leq c n f(\omega _n |x|^n)$
for some constant $c=c(\bfQ)$ and for a.e. $x \in B$. Thereby,  there
exists a constant  $c=c(n, \bfQ)$ such that
$$ \|\ep(\bfu)\|_{X(\Omega)}\leq c\|f\|_{{X}(0,\infty)}.$$
Assume now, in addition, that the first column of $\bfQ$ agrees with $(0,1,0, \dots, 0)^T$.
Thus,   $$\bfu (x_1, 0, \dots, 0) =   (0,x_1,0, \dots, 0)^T \int_{\omega _n |x_1|^n}^1 f(r) r^{-1}  \dr  \qquad \text{ if $|x_1|\leq \omega_n^{-\frac 1n}$.}
$$
 Moreover,  $\bfu(0)=0$. 
Hence,
$$ |\bfu(x_1,0,\ldots,0)-\bfu(0,0,\ldots,0)|= |x_1|
\int_{\omega_n|x_1|^n}^1 f(r) r^{-1}\dr
\quad \hbox{if $|x_1| \leq \omega _n^{-1/n}$.} $$ 
Therefore, there exist
positive constants $c, c', c'', c'''$ such that
\begin{align}\label{1102}
c&\geq \mathop{\sup}\limits_{u\in\,
E^1_0X(\Omega)}\frac{\left|\bfu(x_1,0\ldots,0)-\bfu(0,0,\ldots,0)\right|}{\sigma(|x_1|)\|\ep(\bfu)\|_{X(\Omega)}}
\geq c'\mathop{\sup}\limits_{\begin{tiny}
                      \begin{array}{c}
                       {f\in L^\infty (0, \infty ) }\\
                        f=0 \,{\rm in} \,(1 , \infty )
                      \end{array}
                      \end{tiny}}\frac{|x_1|\,\int_{\omega_n|x_1|^n}^1f(r)\,r^{-1}\dr}{\sigma(|x_1|)\,\|f\|_{{X}(0,\infty)}}
\\ \nonumber &= 
c'\mathop{\sup}\limits_{\begin{tiny}
                      \begin{array}{c}
                       {f\in {X} (0, \infty ) }\\
                        f=0 \,{\rm in} \,(1 , \infty )
                      \end{array}
                      \end{tiny}}\frac{|x_1|\,\int_{\omega_n|x_1|^n}^1f(r)\,r^{-1}\dr}{\sigma(|x_1|)\,\|f\|_{{X}(0,\infty)}}
= c''\frac{\,|x_1|\,\|r^{-1}\chi_{(\omega_n|x_1|^n,1)}(r)\|_{
{X} '(0,\infty)}}{\sigma (|x_1|)}
\geq c''' \,\frac{\varrho_{X}(|x_1|)}{\sigma (|x_1|)}
 \end{align}
if $|x_1|$ is sufficiently small.
 Note that the first equality holds by the same argument exploited
 in the first equality in \eqref{1101}.
Combining equations \eqref{1101} and \eqref{1102} implies that there exists
a constant $c$ such that
$$\sigma _{X} (r) \leq c \,\sigma (r) \quad \hbox{if $r$ is sufficiently small.}
%
$$ Hence,
\eqref{limitesigma} follows, and the embedding $C^{\sigma
_{X}}(\Omega ) \to C^{\sigma}(\Omega )$ holds.
\\ \emph{Part} II. The proof of inequality \eqref{cont1'} can be accomplished along the same lines as that given for \eqref{cont1}. This is possible since,  by Theorem \ref{C1} there exists a   linear bounded extension operator  $\mathscr{E}_\Omega : E^1X(\Omega) \to E^1X_e(\rn)$. Moreover, since $\Omega$ is bounded, after multiplying $\mathscr{E}_\Omega$ by a function $\rho \in C^\infty_0(\rn)$ such that $\rho=1$ in $\Omega$, one obtains a  linear  bounded extension operator from $E^1X(\Omega)$ into  $ E^1_bX_e(\rn)$.
\end{proof}

\section{Symmetric gradient Orlicz-Sobolev spaces}\label{secorlicz}

Having the general results established in Sections \ref{sobemb} and \ref{cont} at our disposal, 
we are now in a position to prove sharp embeddings   for the Orlicz-Sobolev spaces $E^{1}_0L^A(\Omega)$ and $E^{1}L^A(\Omega)$ on open sets $\Omega$ with finite measure, and for their counterparts on $\rn$.
\par Consider  first the case when $\Omega$ is an open set in $\rn$ with $|\Omega|<\infty$.  Owing to Theorem \ref{reduction}, the optimal Orlicz target space for Sobolev embeddings of the spaces $E^{1}_0L^A(\Omega)$ and $E^{1}L^A(\Omega)$ is the same as that for embeddings of  $W^{1}_0L^A(\Omega)$ and $W^{1}L^A(\Omega)$.  The latter was exhibited in \cite{Cianchi_CPDE} (and in \cite{Ci1} in an equivalent form), and is determined by the Young function $A_n$ defined as follows.
Assume that 
$A$ satisfies the condition
\begin{equation}\label{orlicz1}
\int _0 \bigg(\frac t{A(t)}\bigg)^{\frac {1}{n-1}} \dt < \infty.
\end{equation}
Notice that such a condition   is not a restriction, since  we are assuming that
$|\Omega|< \infty$. Indeed, owing to the characterization \eqref{equivorlicz} of equivalent Orlicz function norms, the Young function $A$ can be replaced, if necessary, by a Young function
equivalent near infinity, which renders \eqref{orlicz1} true, and leaves the spaces $E^{1}_0L^A(\Omega)$ and $E^{1}L^A(\Omega)$ unchanged (up to equivalent
norms).
Define the function $H: [0, \infty ) \to [0, \infty
)$ as
\begin{equation}\label{orlicz2}
H(t) = \bigg(\int _0^t \bigg(\frac s{A(s)}\bigg)^{\frac
{1}{n-1}} \ds\bigg)^{\frac {1}{n'}} \quad \text{for $t\geq 0$.}
\end{equation}
Then  $A_n$ is the Young function given  by
\begin{equation}\label{E:A}
A_n(t) =  A(H^{-1}(t)) \quad
\text{for $t\geq 0$.}
\end{equation}
Here, $H^{-1}$ denotes the classical inverse of $H$ if 
\begin{equation}\label{orlicz1bis}
\int ^\infty \bigg(\frac t{A(t)}\bigg)^{\frac {1}{n-1}} \dt = \infty,
\end{equation}
whereas it has to be understood as the generalized left-continuous inverse of $H$ if 
\begin{equation}\label{orlicz1conv}
\int ^\infty \bigg(\frac t{A(t)}\bigg)^{\frac {1}{n-1}} \dt < \infty.
\end{equation}
In the latter case, $H^{-1}(t)$ diverges to infinity as $t$ tends to $(\int _0^\infty (\frac s{A(s)})^{\frac
{1}{n-1}} \ds)^{\frac {1}{n'}}$,
and $A_n(t)$ has to be interpreted as $\infty$ for $t$ larger than this value. In particular, under the current assumption that $|\Omega|<\infty$,  if condition \eqref{orlicz1conv} is in force, then 
\begin{equation}\label{LAn=inf}
L^{A_n}(\Omega) = L^\infty(\Omega),
\end{equation}
up to equivalent norms.

\begin{theorem}{\rm{\bf [Optimal Orlicz target for Orlicz-Sobolev embeddings on open sets] }}\label{orlicz}
Let $\Omega$ be an open set   in $\R^n$  with $|\Omega|<\infty$.   

\vspace{0.1cm}
\par\noindent 
{\bf Part I [Embeddings for  $E^1_0L^A(\Omega)$]}
 The embedding $E^{1}_0L^A(\Omega) \to L^{A_n}(\Omega)$ holds, and  there exists a constant $c=c(n, A, |\Omega|)$ such that 
\begin{equation}\label{orlicz4}
\|\bfu\|_{L^{A_n}(\Omega)} \leq c \|\ep(\bfu)\|_{L^A(\Omega)}
\end{equation}
for every $\bfu \in E^{1}_0L^A(\Omega)$.  Moreover, the target space in inequality \eqref{orlicz4} 
is optimal among all Orlicz spaces. In particular, if $A$ satisfies condition  \eqref{orlicz1}, then the constant $c$ in inequality \eqref{orlicz4} depends only on $n$.
%
%

\vspace{0.1cm}
\par\noindent 
{\bf Part II [Embeddings for  $E^1L^A(\Omega)$]}
Assume, in addition, that $\Omega$   is an $(\varepsilon , \delta)$-domain. Then  $E^{1}L^A(\Omega) \to L^{A_n}(\Omega)$, and  there
 there exists a constant $c=c(A, \Omega)$ such that 
\begin{equation}\label{orlicz4'}
\|\bfu\|_{L^{A_n}(\Omega)} \leq c \|\bfu\|_{E^1L^A(\Omega)}
\end{equation}
for every $\bfu \in E^{1}L^A(\Omega)$.
Moreover, 
the target space in   inequality \eqref{orlicz4'}
is optimal among all Orlicz spaces.  In particular, if $A$ satisfies condition  \eqref{orlicz1}, then the constant $c$ in inequality  \eqref{orlicz4'} depends only on $\Omega$.
\end{theorem}
\begin{proof} By Theorem \ref{reduction}, inequalities \eqref{orlicz4} and  \eqref{orlicz4'} are   consequences of  the inequality
\begin{equation}\label{orlicz5}
\bigg\|\int_s^{|\Omega|} f (r)r^{-1+\frac 1{n'}}\dr \bigg\|_{L^{A_n}(0, |\Omega|)} \leq  c\|f\|_{L^A(0, |\Omega|)}
\end{equation}
for some constant $c$ and for every $f \in L^A(0, |\Omega|)$.
This inequality follows from \cite[Inequality (2.7)]{Cianchi_CPDE}. An alternative version of inequality \eqref{orlicz5}, with $A_n$ replaced  by an equivalent Young function, is proved in \cite[Lemma 1]{Ci1}.  Hence, via \cite[Lemma 2]{Ci_aniso}, one deduces   that the  space 
$L^{A_n}(0, |\Omega|)$ is optimal in \eqref{orlicz5} among all Orlicz spaces.
By  Theorem \ref{reduction} again, this ensures that the target spaces in \textup{(}i\textup{)} and \textup{(}ii\textup{)} are optimal among all Orlicz spaces. 
\end{proof}

\begin{remark}\label{integralform}
{\rm
Inequalities \eqref{orlicz4} and \eqref{orlicz4'} of Theorem \ref{orlicz} can be equivalently formulated in an integral form, provided that the function $A$ satisfies condition \eqref{orlicz1}. These formulations can be of use in view of applications to the analysis of systems of partial differential equations, and read as follows.
\\
Assume that $\Omega$ is as in Part I. Then,  there exists a constant $c=c(n)$ such that
\begin{equation}\label{mar15}
\int_{\Omega}A _n\Bigg(\frac{ |\bfu|}{c\,
\big(\int_{\Omega}A(|\ep(\bfu)|)\dy\big)^{\frac 1{n}}}\Bigg)\dx \leq
\int_{\Omega}A(|\ep(\bfu)|)\dx
\end{equation}
for every $\bfu \in E^1_0L^A(\Omega)$.
\\ Under the assumptions on $\Omega$ of Part II, 
 there exists a constant $c=c(\Omega)$ such that
\begin{equation}\label{mar16}
\int_{\Omega}A _n\Bigg(\frac{ |\bfu|}{c\,
\big(\int_{\Omega}A(|\bfu|)\dy + \int_{\Omega}A(|\ep(\bfu)|)\dy\big)^{\frac 1{n}}}\Bigg)\dx \leq
\int_{\Omega}A(|\bfu|)\dx + \int_{\Omega}A(|\ep(\bfu)|)\dx
\end{equation}
for every $\bfu \in E^1L^A(\Omega)$.
\\ Inequality \eqref{mar15} can be deduced as follows.   Set  $M=\int_{\Omega}A(|\ep(\bfu)|)\dx$. Of course, we may assume that this quantity is finite, otherwise inequality \eqref{mar15} holds trivially. Define the Young function $A_M$ as $A_M(t) = A(t)/M$ for $t \geq 0$, and let $(A_M)_n$ be the function defined as in \eqref{E:A}, with $A$ replaced by $A_M$. One can verify that 
$(A_M)_n(t)= \frac 1M A\big(\frac t{M^{1/n}}\big)_n $ for $t \geq 0$. 
Since, under assumption \eqref{orlicz1}, the constant $c$ in inequality \eqref{orlicz4} is independent of $A$,   this inequality still holds for $\bfu$, with $A$ and $A_n$ replaced by $A_M$ and $(A_M)_n$, respectively. By the very definition of Luxemburg norm, $\|\ep(\bfu)\|_{L^{A_M}(\Omega)} \leq 1$. Therefore, $\|\bfu\|_{L^{(A_M)_n}(\Omega)} \leq c$, and hence  inequality \eqref{mar15} follows from the expression of $(A_M)_n$ and the definition of Luxemburg norm again.
\\ An analogous argument enables one to deduce inequality \eqref{mar16} from \eqref{orlicz4'}.
}
\end{remark}

\begin{example}\label{ex1}{\rm
Let $\Omega$ be an open set in $\rn$ such that $|\Omega|<\infty$. Consider the special family of Orlicz spaces $L^p(\log L)^\alpha(\Omega)$, where either $p = 1$ and $\alpha\geq 0$, or $p > 1$ and $\alpha\in \R$, also called Zygmund spaces.
%
%
\\ Theorem \ref{orlicz} enables  us to deduce that  
\begin{equation}\label{mar23}
 E^1_0L^p ({\rm log} L)^\alpha (\Omega) \to
\begin{cases}
L^{\frac{np}{n-p}}(\log L)^{\frac{n\alpha}{n-p}}(\Omega)  & \text{if $1 \leq p <n$} \\
 \exp L^{\frac n{n-1-\alpha}} (\Omega)& \text{if $p= n$ and $\alpha < n-1$} \\
 \exp \exp L^{\frac n{n-1}} (\Omega)& \text{if $p= n$ and $\alpha = n-1$}
\\
L^{\infty} (\Omega)& \text{if either $p= n$ and $\alpha >
n-1$, or $p> n$} \,,
\end{cases}
\end{equation}
all the target spaces being optimal in the class of Orlicz spaces. 
\\ The same embeddings hold, with   $E^1_0L^p ({\rm log} L)^\alpha (\Omega)$ replaced by $ E^1L^p ({\rm log} L)^\alpha (\Omega)$, provided that $\Omega$ is an $(\varepsilon, \delta)$-domain.}
 \end{example}

  \begin{example}
\label{ex2}{\rm
Let $\Omega$ be an open set in $\rn$ such that $|\Omega|<\infty$.
Then, one can infer from  Theorem \ref{orlicz} that
\begin{equation*}
E^1_0L^p (\log \log  L)^\alpha (\Omega) \to
\begin{cases}
L^{\frac{np}{n-p}}(\log \,\log L)^{\frac{n\alpha}{n-p}} (\Omega) & \text{if $1 \leq p <n$}   \\
  \exp \big((\log L)^{\frac{\alpha}{n-1}}\big)(\Omega)
  & \text{if $p= n$}
\\
L^{\infty} (\Omega)& \text{if  $p> n$} \,.
\end{cases}
\end{equation*}
 Moreover, the target spaces are optimal among all Orlicz spaces. Analogous conclusions hold when  $\Omega$ is an $(\varepsilon, \delta)$-domain and $E^1_0L^p ({\log\log} L)^\alpha (\Omega)$ is replaced by $ E^1L^p ({\log\log} L)^\alpha (\Omega)$.}
 \end{example}

Assume that the Young function $A$ fulfills condition  \eqref{orlicz1bis}. 
Although the target space in embeddings \eqref{orlicz4} and \eqref{orlicz4'} of Theorem \ref{orlicz} is
optimal in the family of Orlicz spaces,   it can still be improved if the
class of admissible targets is enlarged to include all
rearrangement-invariant spaces. The
optimal rearrangement-invariant target space is an Orlicz-Lorentz space $L(\widehat A, n)(\Omega)$, whose norm is defined as follows. Let $a$ be the function appearing in \eqref{young}, and let $\widehat A$ be the function 
defined  by
\begin{equation}\label{G}
\widehat A(t) = \int _0^t \widehat a (\tau) \,\mathrm{d}\tau\quad \text{for $t\geq 0$},
\end{equation}
where 
where $\widehat a$ is the non-decreasing, left-continuous function in $[0,
\infty )$ obeying
\begin{equation}\label{G1}
 \widehat a^{-1}(t) =  \bigg(\int _{a^{-1}(t)}^{\infty}\bigg(\int _0^s \bigg(\frac{1}{a(r)}\bigg)^{\frac{1}{n-1}}
 \dr\bigg)^{-n}\frac{\ds}{a(s)^{\frac{n}{n-1}}}\bigg)^{\frac{1}{1-n}}\quad{\rm for}\,\,\,
 t> 0.
\end{equation}
Then, according to \eqref{016}, $L(\widehat A, n)(\Omega)$ denotes the space associated with the rearrangement-invariant function norm given by
\begin{equation}\label{016mar}
\|f\|_{L(\widehat  A, n)(\Omega)} = \big\|s^{-\frac 1n} f^* (s)\big\|_{L^{\widehat A}(0, |\Omega|)}
\end{equation}
for $f \in \Mpl (0, |\Omega|)$. In view of Theorem \ref{reduction}, $L(\widehat A, n)(\Omega)$  is the same optimal rearrangement-invariant target space, identified in \cite{Ci3},   for embeddings of the spaces $W^1_0L^A(\Omega)$ and $W^1L^A(\Omega)$.

\begin{theorem}{\rm{\bf [Optimal rearrangement-invariant target for Orlicz-Sobolev embeddings on open sets]}}\label{orliczri}
Let $\Omega$ be an open set   in $\R^n$ with $|\Omega|<\infty$.   Assume that $A$ is a Young
function satisfying condition   \eqref{orlicz1bis}.

\vspace{0.1cm}
\par\noindent 
{\bf Part I [Embeddings for  $E^1_0L^A(\Omega)$]} The embedding $E^{1}_0L^A(\Omega) \to L(\widehat A, n)(\Omega)$ holds, and  there exists a constant $c=c(n, A, |\Omega|)$ such that
\begin{equation}\label{orliczri1}
\|\bfu\|_{L(\widehat A, n)(\Omega)} \leq c \|\ep(\bfu)\|_{L^A(\Omega)}
\end{equation}
for every $\bfu \in E^{1}_0L^A(\Omega)$.
The target space $L(\widehat A, n)(\Omega)$ in inequality \eqref{orliczri1}  
is optimal among all rearrangement-invariant spaces. 
In particular, if $A$ satisfies condition  \eqref{orlicz1}, then the constant $c$ in inequality \eqref{orliczri1} depends only on $n$.

\vspace{0.1cm}
\par\noindent 
{\bf Part II [Embeddings for  $E^1L^A(\Omega)$]} Assume, in addition, that $\Omega$   is an $(\varepsilon , \delta)$-domain. Then $E^{1}L^A(\Omega) \to L(\widehat A, n)(\Omega)$, and
 there exists a constant $c=c(A, \Omega)$ such that 
\begin{equation}\label{orliczri4'}
\|\bfu\|_{L(\widehat A, n)(\Omega)} \leq c \|\bfu\|_{E^1L^A(\Omega)}
\end{equation}
for every $\bfu \in E^{1}L^A(\Omega)$.
The target space $L(\widehat A, n)(\Omega)$ in inequality \eqref{orliczri4'}
is optimal among all rearrangement-invariant spaces. 
In particular, if $A$ satisfies condition  \eqref{orlicz1}, then the constant $c$  in inequality \eqref{orliczri4'} depends only on $\Omega$.
\end{theorem}

\begin{proof} Owing to Theorem \ref{reduction}, the proof of inequalities \eqref{orliczri1} and \eqref{orliczri4'}  is reduced to showing that   the inequality
\begin{equation}\label{orliczri2}
\bigg\|\int_s^{|\Omega|} f (r)r^{-1+\frac 1{n'}}\dr \bigg\|_{L(\widehat A, n)(0, |\Omega|)} \leq  C\|f\|_{L^A(0, |\Omega|)}
\end{equation}
holds for every   $f\in L^A(0, |\Omega|)$, and that  the space 
$L(\widehat A, n)(0, |\Omega|)$ is the optimal target in this inequality among all rearrangement-invariant spaces. 
These facts follow from  \cite[Inequality (3.1) and Proof of Theorem 1.1]{Ci3}.
\end{proof}

\begin{example}\label {ex3} {\rm
Let  $\Omega$ be an open set in $\rn$ with $|\Omega|<\infty$,  and let $p$ and $\alpha$  be as in
Example \ref{ex1}. Then Theorem \ref{orliczri} enables one to show that
\begin{equation}\label{mar23bis}
 E^1_0L^p ({\rm log} L)^\alpha (\Omega) \to
\begin{cases}
L^{\frac{np}{n-p}, p; \frac \alpha p} (\Omega)  & \text{if $1 \leq p < n$}, \\
 L^{\infty, n;-1+\frac{\alpha}{n}} (\Omega)& \text{if $p= n$ and $\alpha < n-1$}, \\
 L^{\infty,n;-\frac{1}{n},-1} (\Omega)& \text{if $p= n$ and $\alpha = n-1$},
\end{cases}
\end{equation}
up to equivalent norms, and that the target spaces are optimal
among all rearrangement-invariant spaces. Notice that the target spaces in embeddings \eqref{mar23bis} are of Lorentz-Zygmund type, whereas the target space in inequality \eqref{orliczri1} is an Orlicz-Lorentz space. The equivalence of the norms in  these two families of spaces, for the specific values of the parameters $p, \alpha, n$ appearing in equation \eqref{mar23bis}, follows via variants of the arguments of \cite[Lemma 6.2, Chapter 4]{BS}.
\\ If, in addition, $\Omega$   is  an $(\varepsilon, \delta)$-domain, then embedding \eqref{mar23bis} also holds with $E^1_0L^p ({\rm log} L)^\alpha (\Omega)$  replaced by $E^1L^p ({\rm log} L)^\alpha (\Omega)$. }
\end{example}

The next couple of theorems is devoted to optimal target spaces for Sobolev embeddings on the whole of $\rn$. The former deals with the space $E^{1}_0L^A(\rn)$, the latter with $E^{1}L^A(\rn)$. Both results provide optimal targets in the class of Orlicz spaces and in that  of rearrangement-invariant spaces.

\begin{theorem}{\rm{\bf [Optimal Orlicz-Sobolev embeddings  on $\rn$, case of $E^{1}_0L^A(\rn)$] }}\label{orliczrn} 
Assume that $A$ is a Young function satisfying condition \eqref{orlicz1}.

\vspace{0.1cm}
\par\noindent 
 {\bf Part I [Optimal Orlicz target space]}
The embedding $E^{1}_0L^A(\rn) \to L^{A_n}(\rn)$ holds, and there exists a constant $c=c(n)$ such that 
\begin{equation}\label{orliczrn0}
\|\bfu\|_{L^{A_n}(\rn)} \leq c \|\ep(\bfu)\|_{L^A(\rn)}
\end{equation}
for every $\bfu \in E^{1}_0L^A(\rn)$. The target space in \eqref{orliczrn0}  
is optimal among all Orlicz spaces.

\vspace{0.1cm}
\par\noindent 
 {\bf Part II [Optimal rearrangement-invariant target space]}
\\ {\rm (i)}  Assume that $A$ satisfies condition \eqref{orlicz1bis}.  Then $E^{1}_0L^A(\rn) \to L(\widehat A, n)(\rn)$, and there exists a constant $c=c(n)$ such that  
%
\begin{equation}\label{orliczrn1}
\|\bfu\|_{L(\widehat A, n)(\rn)} \leq c \|\ep(\bfu)\|_{L^A(\rn)}
\end{equation}
for every $\bfu \in E^{1}_0L^A(\rn)$. The target space in \eqref{orliczrn1}  
is optimal among all rearrangement-invariant spaces. 
\\ {\rm (ii)}
Assume that $A$ satisfies condition \eqref{orlicz1conv}. Let $B$
be a Young function such that
\begin{equation}\label{mar32}
B(t) \,\, \textup{is equivalent to} \,\, \begin{cases} \widehat A(t) &\textup{near}\ 0
\\  \infty &\textup{near infinity},
\end{cases}
\end{equation}
and let $\nu : (0, \infty)\to (0, \infty)$ be the function defined as
\begin{equation}\label{mar33}
\nu(s)= \min\{s^{-\frac 1n}, 1\} \quad \text{for $s >0$.}
%
\end{equation}
Then 
 $E^{1}_0L^A(\rn) \to L^B(\nu)(\rn)$, and there exists a constant $c=c(n,B)$ such that  
%
\begin{equation}\label{orliczrn1conv}
\|\bfu\|_{L^B(\nu)(\rn)} \leq c \|\ep(\bfu)\|_{L^A(\rn)}
\end{equation}
for every $\bfu \in E^{1}_0L^A(\rn)$. Moreover, $L^B(\nu)(\rn)=L(\widehat A, n)(\rn) \cap L^\infty(\rn)$ (up to equivalent norms), and it
is the optimal target space  in \eqref{orliczrn1conv}  
among all rearrangement-invariant spaces. 
\end{theorem}
 
\begin{proof} \emph{Part} I. As a consequence of Theorem \ref{reductionrn}, Part II, inequality \eqref{orliczrn0} is equivalent to the inequality
\begin{equation}\label{mar10}
\bigg\|\int_s^{\infty} f (r)r^{-1+\frac 1{n'}}\dr \bigg\|_{L^{A_n}(0, \infty)} \leq  C\|f\|_{L^A(0, \infty)}
\end{equation}
for every non-increasing function $f: (0, \infty) \to [0, \infty)$. Inequality \eqref{mar10} and the optimality of the Orlicz target space $L^{A_n}(0, \infty)$ follow from the results  mentioned  with regard to inequality \eqref{orlicz5} in the proof of Theorem \ref{orlicz}. The fact that $L^{A_n}(0, \infty)$ is the optimal Orlicz target space in  \eqref{mar10} implies that $L^{A_n}(\rn)$ is the optimal target space in \eqref{orliczrn1}.
\\ \emph{Part} II. Inequalities \eqref{orliczrn1} and \eqref{orliczrn1conv}, and the optimality of the target spaces, are established, with $\ep(\bfu)$ replaced by $\nabla \bfu$, in \cite[Theorem 1.1]{Ci3}. Note that the assertion about the latter inequality also relies upon \cite[Proposition 2.1]{Ci3}.  Owing to Theorem \ref{reductionrn}, Part I, 
the same conclusions continue to hold for inequalities \eqref{orliczrn1} and \eqref{orliczrn1conv} in the present form, namely with $\ep(\bfu)$ on the right-hand side.
\end{proof}

\begin{example}\label{ex4}{\rm
Consider a Young function $A$ such that
\begin{equation}\label{dec250}
A(t) \,\, \text{is equivalent to} \,\, \begin{cases} t^{p_0} (\log \frac 1t)^{\alpha_0} & \quad \text{near zero}
\\
t^p  (\log t)^\alpha & \quad \text{near infinity,}
\end{cases}
\end{equation}
where either $p_0>1$ and $\alpha_0 \in \R$, or $p_0=1$ and $\alpha_0 \leq 0$, and either $p>1$ and $\alpha \in \R$ or $p=1$ and $\alpha \geq 0$.  
\\ The function $A$ satisfies assumption \eqref{orlicz1} if
\begin{equation}\label{dec252}
\text{either $1\leq p_0<n$ and $\alpha_0$ is as above, or $p_0=n$ and $\alpha_0 > n -1$.}
\end{equation}
%
%
%
 Theorem \ref{orliczrn}, Part I, then tells us that  
\begin{equation}\label{mar27}
E^1_0L^A(\rn) \to L^{A_n}(\rn),
\end{equation}
 where
\begin{equation}\label{mar26}
A_{n}(t) \,\, \text{is equivalent to} \,\,  \begin{cases} t^{\frac {n{p_0}}{n-{p_0}}} (\log \frac 1t)^{\frac {n\alpha_0}{n-{p_0}}} & \quad \text{ if $1\leq {p_0}< n$ }
\\
e^{-t^{-\frac{n}{\alpha_0 +1-n}}} & \quad  \text{if ${p_0}=n$ and $\alpha_0 > n -1$}
\end{cases} \quad \quad \text{near zero,}
\end{equation}
and
\begin{equation}\label{dec256}
A_{n}(t) \,\, \text{is equivalent to} \,\, \begin{cases} t^{\frac {np}{n-p}} (\log t)^{\frac {n\alpha}{n-p}} & \quad  \text{ if $1\leq p< n$ }
\\
e^{t^{\frac{n}{n-\alpha -1}}}&  \quad \text{if  $p=n$ and $\alpha < n-1$}
\\
e^{e^{t^{\frac n{n-1}}}} &  \quad \text{if  $p=n$ and $\alpha = n -1$}
\\ \infty &  \quad \text{otherwise}
\end{cases} \quad \quad \text{near infinity.}
\end{equation}
Moreover, the target space in  inequality \eqref{mar27} is optimal among all Orlicz spaces.
\\ 
Next, one can verify that
\begin{equation}\label{dec253}
\widehat A(t)\,\, \text{is equivalent to} \,\,  \begin{cases} t^{p_0}(\log \frac 1t)^{\alpha_0} & \quad \text{ if $1\leq {p_0}< n$ }
\\
t^{n}  (\log \frac 1t)^{\alpha_0 -n} & \quad  \text{if ${p_0}=n$ and $\alpha_0> n -1$}
\end{cases} \quad\quad  \text{near zero,}
\end{equation}
and
\begin{equation}\label{dec254}
\widehat A(t) \,\, \text{is equivalent to} \,\,  \begin{cases} t^p (\log t)^\alpha & \quad  \text{ if $1\leq p< n$ }
\\
t^{n}  (\log t)^{\alpha - n} &  \quad \text{if  $p=n$ and $\alpha <n -1$}
\\
t^{n}  (\log t)^{-1} (\log (\log t))^{-n}&  \quad \text{if  $p=n$ and $\alpha = n -1$}
\end{cases} \quad\quad \text{near infinity.}
\end{equation}
By Theorem \ref{orliczrn}, Part II one thus has that, if either $1\leq p< n$, or $p=n$ and $\alpha \leq n -1$,  then 
\begin{equation}\label{mar30} 
E^{1}_0L^A(\rn) \to L(\widehat A, n)(\rn),
\end{equation}
where $\widehat A$ obeys  \eqref{dec253} and \eqref{dec254}, whereas if either $p>n$,  or $p=n$ and $\alpha > n -1$, then 
\begin{equation}\label{mar31}  E^{1}_0L^A(\rn) \to L^B(\nu)(\rn),
\end{equation}
where 
$$\nu(s)= \min\{s^{-\frac 1n}, 1\} \quad \text{for $s >0$,}$$
and 
\begin{equation*}
B(t)\,\, \text{is equivalent to} \,\,  \begin{cases} t^{p_0}(\log \frac 1t)^{\alpha_0} & \quad \text{ if $1\leq {p_0}< n$ }
\\
t^{n}  (\log \frac 1t)^{\alpha_0 -n} & \quad  \text{if ${p_0}=n$ and $\alpha_0> n -1$}
\end{cases} \quad \text{near zero,}
\end{equation*}
and 
\begin{equation*}
B(t) \,\, \text{is equivalent to} \,\, 
\infty \quad \text{near infinity}\,.
\end{equation*}
Moreover, the target space is optimal among all rearrangement-invariant spaces both in \eqref{mar30} and in \eqref{mar31}.
}
\end{example}

\begin{theorem}{\rm{\bf [Optimal Orlicz-Sobolev embeddings  on $\rn$, case of $E^{1}L^A(\rn)$ ] }}\label{orliczrnfull}
Let  $A$ be any Young function.

\vspace{0.1cm}
\par\noindent 
 {\bf Part I [Optimal Orlicz target space]}
Let $\overline {A}_n$
be a Young function such that
\begin{equation}\label{mar35}
\overline {A}_n(t)  \,\, \textup{is equivalent to} \,\, \begin{cases} A(t) &\textup{near}\ 0
\\
A_{n}(t) &\textup{near infinity}\,.
\end{cases}
\end{equation}
 Then $E^{1}L^A(\rn) \to L^{\overline {A}_n}(\rn)$, and there exists a constant $c=c(n, A)$ such that 
\begin{equation}\label{orliczrn2}
\|\bfu\|_{L^{\overline {A}_n}(\rn)} \leq c \|\bfu\|_{E^1L^A(\rn)}
\end{equation}
for every $\bfu \in E^{1}L^A(\rn)$. The target space in \eqref{orliczrn2}  
is optimal among all Orlicz spaces. 
\\ In particular, if the function $A$ satisfies condition \eqref{orlicz1conv}, then $A_{n}(t)= \infty$ near infinity, and the target space in \eqref{orliczrn2}  
is also optimal among all rearrangement-invariant spaces.

\vspace{0.1cm}
\par\noindent 
 {\bf Part II [Optimal rearrangement-invariant target space]} Assume that $A$ is a Young function satisfying condition \eqref{orlicz1bis}.
 Let $D$
be a Young function such that
\begin{equation}\label{mar36}
D(t) \,\, \textup{is equivalent to} \,\, \begin{cases} A(t) &\textup{near}\ 0
\\  \widehat  A(t) &\textup{near infinity}\,
\end{cases}
\end{equation}
and let $\varpi: (0, \infty)\to (0, \infty)$ be the function defined as
\begin{equation}\label{mar37}
\varpi(s)= \max\{s^{-\frac 1n}, 1\} \quad \text{for $s >0$.}
%
\end{equation}
Then $E^{1}L^A(\rn) \to\Lambda ^D(\varpi)(\rn)$, and there exists a constant $c=c(n,D)$ such that
\begin{equation}\label{orliczrn3}
\|\bfu\|_{\Lambda ^D(\varpi)(\rn)} \leq c \|\bfu\|_{E^1L^A(\rn)}
\end{equation}
for every $\bfu \in E^{1}L^A(\rn)$.
The target space in \eqref{orliczrn3} is  optimal among all
rearrangement-invariant spaces.
\end{theorem}

\begin{proof} \emph{Part} I. As a consequence of Theorem \ref{reductionrn}, Part I,  inequality \eqref{orliczrn2} is equivalent to the same inequality, with the space $E^1L^A(\rn)$ replaced by  $W^1L^A(\rn)$. The fact that the inequality for the latter space holds, and that $L^{\overline {A}_n}(\rn)$ is the optimal Orlicz target  space is the content of \cite[Theorem 3.5]{ACPS}.
\\ \emph{Part} II. 
Inequality \eqref{orliczrn3} is equivalent to the same inequality, with the space $E^1L^A(\rn)$ replaced by  $W^1L^A(\rn)$. The fact that the inequality for the latter space holds, and that $\Lambda ^D(\varpi)(\rn)$ is the optimal rearrangement-invariant target   space is the content of \cite[Theorem 3.9]{ACPS}.
\end{proof}

\begin{example}\label{ex5}{\rm
Let $A$ be a Young function as in \eqref{dec250}, but now with $p_0$ and $\alpha_0$ non-necessarily fulfilling conditions \eqref{dec252}.
Theorem \ref{orliczrnfull}, Part I,  yields
\begin{equation}\label{mar40}
E^1L^A(\rn) \to L^{\overline A_n}(\rn),
\end{equation}
 where
\begin{equation}\label{mar41}
\overline A_{n}(t)\,\, \text{is equivalent to} \quad  t^{p_0} (\log \tfrac 1t)^{\alpha_0}   \quad \text{near zero},
\end{equation}
and 
\begin{equation}\label{mar42}
\overline A_{n}(t)\,\, \text{is equivalent to}
\begin{cases} t^{\frac {np}{n-p}} (\log t)^{\frac {n\alpha}{n-p}} & \quad  \text{ if $1\leq p< n$ }
\\
e^{t^{\frac{n}{n-\alpha -1}}}&  \quad \text{if  $p=n$ and $\alpha < n-1$}
\\
e^{e^{t^{\frac n{n-1}}}} &  \quad \text{if  $p=n$ and $\alpha = n -1$}
\\ \infty &  \quad \text{otherwise}
\end{cases} \quad \text{near infinity.}
\end{equation}
The target space $ L^{\overline A_n}(\rn)$ in  embedding  \eqref{mar40} is optimal among all Orlicz spaces.
\\ 
 In the last case of \eqref{mar42},  the target space $ L^{\overline A_n}(\rn)$  in  embedding  \eqref{mar40}  is also optimal among all rearrangement-invariant spaces. Otherwise,
 Theorem \ref{orliczrnfull}, Part  II, implies that
\begin{equation}\label{mar43}
E^1L^A(\rn) \to \Lambda ^{D}(\varpi)(\rn),
\end{equation}
where 
$$\varpi(s)= \max\{s^{-\frac 1n}, 1\} \quad \text{for $s >0$,}$$
and 
\begin{equation}\label{mar44}
D(t)\,\, \text{is equivalent to} \quad  t^{p_0} (\log \tfrac 1t)^{\alpha_0}  \quad  \text{near zero},
\end{equation}
and 
\begin{equation}\label{mar45}
D(t) \,\, \textup{is equivalent to} \,\, \begin{cases} t^p (\log t)^\alpha & \quad  \text{ if $1\leq p< n$ }
\\
t^{n}  (\log t)^{\alpha - n} &  \quad \text{if  $p=n$ and $\alpha <n -1$}
\\
t^{n}  (\log t)^{-1} (\log (\log t))^{-n}&  \quad \text{if  $p=n$ and $\alpha = n -1$}
\end{cases} \quad  \text{near infinity.}
\end{equation}
Furthermore, the target space $\Lambda ^{D}(\varpi)(\rn)$  in embedding \eqref{mar43} is optimal among all rearrangement-invariant spaces.
}
\end{example}

We conclude by describing symmetric gradient Orlicz-Sobolev embeddings into spaces of uniformly continuous functions. This is the subject of Theorem \ref{orliczcont} below. A couple of interesting features are worth pointing out in this connection. First,  for this class of spaces condition \eqref{limitesigma} is fulfilled whenever \eqref{apr40} is fulfilled. This fact implies that an embedding  of a symmetric gradient Orlicz-Sobolev space into a space $C^\sigma (\Omega)$ holds, for a suitable modulus of continuity $\sigma$, whenever the embedding into $L^\infty (\Omega)$ (or, equivalently, into $C^0(\Omega)$) holds. As mentioned in Remark \ref{notunif}, this need not  be the case for symmetric gradient Sobolev spaces associated with more general rearrangement-invariant function norms. Second, the optimal target space $C^\sigma (\Omega)$ for an embedding of a symmetric gradient Orlicz-Sobolev space  can be essentially larger than that for the full gradient Orlicz-Sobolev space associated with the same Young function -- see Examples \ref{Linf} and \ref{exp}. As already emphasized above, this gap is absent when dealing with rearrangement-invariant target spaces.
\par The optimal modulus of continuity of functions in the spaces $E^1_0X(\Omega)$ and $E^1X(\Omega)$ is defined in terms of the behaviour near infinity of the function 
$\xi_A : (0, \infty) \to [0, \infty
)$, given by
\begin{equation}\label{4.2}
\xi_A (t) = t^{n'} \int _t^\infty \frac {\widetilde A
(\tau )}{\tau ^{1+ n'}}\, \mathrm{d}\tau \quad \hbox{for $t>0$,}
\end{equation}
and  of the function $\eta_A : (0, \infty) \to [0,
\infty )$, given by
\begin{equation}\label{4.3}
\eta_A (t) = t \int _0^t \frac {\widetilde A (\tau
)}{\tau ^{2}}\, \mathrm{d}\tau \quad \hbox{for $t>0$.}
\end{equation}
One can verify that $\xi_A$ and $\eta_A$ are, in fact, Young functions.
\par\noindent
Let us notice that the
 convergence of the integral on the right-hand side of \eqref{4.2}
is equivalent to condition \eqref{4.1} appearing in Theorem \ref{orliczcont} -- see e.g. \cite[Lemma 2.3]{Ci3}. Moreover, since we are dealing with bounded sets $\Omega$,
we may always assume, without loss of generality, that the
integral in \eqref{4.3}   is
convergent. Indeed, $A$ can be replaced, if necessary, by another
Young  function which is equivalent to $A$  near infinity and makes
the relevant integral converge. Such a
replacement results in the same symmetric gradient Orlicz-Sobolev spaces, up to
equivalent norms. The modulus of continuity defined via $\eta_A$ is also not affected (up to equivalence) by this replacement, inasmuch as  it only depends on $\eta_A$ through its  behaviour near infinity.

\begin{theorem}\label{orliczcont}{\rm{\bf [Optimal target space of uniformly continuous functions for Orlicz-Sobolev embeddings]}}
 Let $\Omega$ be an open bounded subset of $\R^n$, and let $A$ be
a Young function.

\vspace{0.1cm}
\par\noindent 
{\bf Part I [Embeddings for  $E^1_0X(\Omega)$]}
\\ (i) The embedding $E^1_0L^A(\Omega) \to C^0(\Omega)$ holds, namely there exists a constant $c$ such that
\begin{equation}\label{jan25}
\|\bfu\|_{C^0(\Omega)} \leq c \|\ep(\bfu)\|_{L^A(\Omega)}  
\end{equation}
for every $\bfu \in E^1_0L^A(\Omega)$
if and only if
\begin{equation}\label{4.1}
\int ^\infty
\bigg(\frac t{A(t)}\bigg)^{\frac 1{n-1}}dt < \infty\,.
\end{equation}
\\(ii) Assume that \eqref{4.1} is in force, and
define
\begin{equation}\label{4.5}
\sigma _{A} (r) =
   \frac { r^{1-n}}{\xi_A ^{-1}(r^{-n})} + \frac { r^{1-n}}{\eta_A
  ^{-1}(r^{-n})}\quad \hbox{for $r>0$.}
\end{equation}
Then $\sigma _{A}$ is a modulus of
continuity, and $E^{1}_0L^A(\Omega ) \to C^{\sigma_ {A}}(\Omega )$, namely there exists a constant $c$ such that
\begin{equation}\label{4.4} 
\|\bfu\|_{C^{\sigma_ {A}}(\Omega )} \leq c \|\ep(\bfu)\|_{L^A(\Omega)}
\end{equation}
for every function $\bfu \in
E^{1}_0L^A(\Omega )$.
 \\ The space $C^{\sigma_ {A}}(\Omega )$ is optimal,
in the sense that, if there exists a modulus of continuity $\sigma
$ such that inequality \eqref{4.4} holds with  $C^{\sigma_ {A}}(\Omega )$ replaced by $C^{\sigma}(\Omega )$,
 then     $C^{\sigma_ {A}}(\Omega ) \to
C^{\sigma}(\Omega )$.

\vspace{0.1cm}
\par\noindent 
{\bf Part II [Embeddings for  $E^1X(\Omega)$]}
Assume that $\Omega$ is a bounded $(\varepsilon, \delta)$-domain. Then   conclusions analogous to those of Part I hold, with $ E^{1}_0L^A(\Omega )$ replaced by $ E^{1}L^A(\Omega )$ and $\|\ep(\bfu)\|_{L^A(\Omega)}$ replaced by $\|\bfu\|_{E^1L^A(\Omega)}$.
\end{theorem}
Theorem \ref{orliczcont} can be deduced as special cases of Theorems \ref{Linfty} and  \ref{continuity}, with $X(\Omega)= L^A(\Omega)$, via  analogous arguments and formulas as in the proofs of \cite[Theorems 6.1 and 6.2]{cianchi-randolfi IUMJ}. The details are omitted, for brevity.

\bigskip
\par
The following examples are stated for spaces of the form $E^1_0L^A(\Omega)$. The same conclusions hold with $E^1_0L^A(\Omega)$ replaced by $E^1L^A(\Omega)$, provided that $\Omega$ is a bounded $(\varepsilon, \delta)$-domain.

\begin{example}\label{Linf}
{\rm Let $\Omega$ be a bounded open subset of $\R^n$. An application of Theorem \ref{orliczcont}, Part I (ii), tells us that
\begin{equation}\label{jan26}
E^1_0L^\infty(\Omega) \to C^{\sigma_{\infty}} (\Omega),
\end{equation}
where 
\begin{equation}\label{aug200}\sigma_{\infty} (r) \approx r\log(1/r) \quad \text{as $r\to 0^+$.}
\end{equation}
Moreover, the modulus of continuity $\sigma_\infty$ is optimal in \eqref{jan26}. 
\\ In this case $A(t)=\chi _{(1, \infty)}(t)\infty$ for $t\geq 0$, whence $\widetilde A(t)=t$ for $t\geq 0$. Thus, $\xi_A(t) \approx t$ near infinity, and $\eta _A(t)\approx t\log t$ near infinity, and equation \eqref{aug200} follows via \eqref{4.5}.
Note that  the behavior of $\sigma_{\infty}(r)$ as $r \to 0^+$ is dictated by the second addend on the right-hand side of equation \eqref{4.5}.
\\ By contrast, the stronger embedding 
\begin{equation}\label{jan26bis}
W^1_0L^\infty(\Omega) \to C^{0,1} (\Omega)
\end{equation}
is well known to hold, where $C^{0,1} (\Omega)$ stands for the space of Lipschitz continuous functions on $\Omega$.
}
\end{example}

\begin{example}\label{exp}
{\rm  Let $\Omega$ be a bounded open subset of $\R^n$.  From Theorem \ref{orliczcont}, Part I (ii), one can deduce  that, given $\beta >0$,
\begin{equation}\label{aug1}
E^{1}_0\exp L^\beta(\Omega) \to C^{\sigma_{\beta}} (\Omega),
\end{equation}
where 
\begin{equation}\label{aug203}\sigma_{\beta} (r) \approx r\big(\log(1/r)\big)^{1+\frac 1 \beta} \quad \text{as $r\to 0^+$.}
\end{equation}
Also, the modulus of continuity $\sigma_\beta$ is optimal in \eqref{aug1}. 
\\ Here, one has that $A(t) \approx e^{t^\beta}$ near infinity, and hence condition \eqref{4.1} is fulfilled for every $\beta >0$. Moreover, $\widetilde A(t) \approx t(\log t)^{\frac 1\beta}$ near infinity, thus $\xi _A(t) \approx  t(\log t)^{\frac 1\beta}$ and $\eta _A(t) \approx  t(\log t)^{1+\frac 1\beta}$  near infinity. Equation \eqref{aug203} follows from \eqref{4.5}.
The behavior of $\sigma_{\beta}(r)$ as $r \to 0^+$ is dictated by the second addend on the right-hand side of equation \eqref{4.5}.
\\ If the full gradient Orlicz-Sobolev space $W^{1}_0\exp L^\beta(\Omega)$ replaces $E^{1}_0\exp L^\beta(\Omega)$ in \eqref{aug1}, then a parallel embedding holds with a smaller target space. Indeed, one has the stronger embedding
\begin{equation}\label{aug1bis}
W^{1}_0\exp L^\beta(\Omega) \to C^{\sigma} (\Omega),
\end{equation}
where $\sigma (r) \approx r\big(\log(1/r)\big)^{\frac 1 \beta}$ as $r\to 0^+$, and it is the optimal modulus of continuity in \eqref{aug1bis} -- see \cite[Example 6.8]{cianchi-randolfi IUMJ}.
}
\end{example}

\begin{example}\label{Lnlog}
{\rm 
 Let $\Omega$ be a bounded open subset of $\R^n$, and let $\alpha \in \R$. An application of Theorem \ref{orliczcont}, Part I (i),  tells us that
\begin{equation}\label{jan27}
E^{1}_0L^n (\log L)^\alpha (\Omega) \to C^0 (\Omega)
\end{equation}
if and only if $\alpha >n-1$. Moreover, Part I (ii) of the same theorem ensures that, if this is the case, then 
\begin{equation}\label{jan28}
E^{1}_0L^n (\log L)^\alpha (\Omega)  \to C^{\sigma_{n,\alpha} }(\Omega),
\end{equation}
where 
\begin{equation}\label{aug201}\sigma_{n, \alpha} (r) \approx( \log(1/r))^{\frac {n-1-\alpha}n} \quad \text{as $r\to 0^+$.}
\end{equation}
Finally, the modulus of continuity $\sigma_{n, \alpha}$ is optimal in \eqref{jan28}. 
\\  In this instance, the function $A(t)\approx t^n (\log t)^\alpha$ near infinity. Hence, condition \eqref{4.1} is fulfilled if and only if $\alpha >n-1$. Moreover, 
 $\widetilde A(t) \approx t^{n'} (\log t)^{-\frac {\alpha}{n-1}}$ near infinity. As a consequence, $\xi_{A}(t) \approx t^{n'}(\log t)^{1-\frac {\alpha}{n-1}}$ and $\eta_A(t) \approx t^{n'}(\log t)^{-\frac {\alpha}{n-1}}$ near infinity. Hence, equation \eqref{aug201} follows from \eqref{4.5}.
Note that  the behavior of $\sigma_{n, \alpha}(r)$ as $r \to 0^+$ is dictated by the first addend on the right-hand side of equation \eqref{4.5}.
\\ Unlike embeddings \eqref{jan26} and \eqref{aug1},
replacing the space $E^{1}_0L^n (\log L)^\alpha (\Omega)$ by $W^{1}_0L^n (\log L)^\alpha (\Omega)$ in \eqref{jan28} does not allow for any improvement of the target space \cite[Example 6.7]{cianchi-randolfi IUMJ}. This is due to the fact that, as a consequence of a version of Korn's inequality in Orlicz spaces \cite{Ckorn,DRS, Fu2},  $E^{1}_0L^n (\log L)^\alpha (\Omega)=W^{1}_0L^n (\log L)^\alpha (\Omega)$, up to equivalent norms.
}
\end{example}

\bigskip
\par\noindent
{\bf Acknowledgements}.  The second author  is
grateful to  M.Bul\'i\v{c}ek and J.M\'alek  for bringing the problem of symmetric gradient Orlicz-Sobolev embeddings  to his attention. He is also obliged to P.Hajlasz for stimulating discussions.
\\ The authors wish to thank the refereee for his valuable comments.

\section*{Compliance with Ethical Standards}\label{conflicts}
\subsection*{Funding}

This research was partly funded by:

\begin{enumerate}
\item Research Project 201758MTR2  of the Italian Ministry of Education, University and
Research (MIUR) Prin 2017 ``Direct and inverse problems for partial differential equations: theoretical aspects and applications'';
\item GNAMPA of the Italian INdAM -- National Institute of High Mathematics
(grant number not available).
\end{enumerate}

\subsection*{Conflict of Interest}

The authors declare that they have no conflict of interest.

\end{document}

\bibitem{Aubin}
\by{\name{T.}{Aubin}} \paper{Probl\`emes isop\'erimetriques et
espaces de Sobolev} \jour{J. Diff. Geom.} \vol{11} \yr{1976}
\pages{573--598}
\endpaper

\bibitem{BS}
\by{\name{C.}{Bennett}\et\name{R.}{Sharpley}}
\book{Interpolation of Operators}
\publ{Pure and Applied Mathematics Vol.\ 129, Academic Press,
Boston 1988}
\endbook

\bibitem{BW}
\by{\name{H.}{Br\'ezis}\et\name{S.}{Wainger}}
\paper{A~note on limiting cases of Sobolev embeddings and
convolution inequalities}
\jour{Comm.\ Partial Diff.\ Eq.}
\vol{5}
\yr{1980}
\pages{773--789}
\endpaper

\bibitem{CPSS}
\by{\name{M.}{Carro}, \name{L.}{Pick}, \name{J.}{Soria}\et\name{V.}{Stepanov}}
\paper{On embeddings between classical Lorentz spaces}
\jour{Math.\ Ineq.\ Appl}
\vol{4}
\yr{2001}
\pages{397--428}
\endpaper

\bibitem{CW}
\by{\name{F.}{Catrina}  \et\name{Z.-Q.}{Wang}}
\paper{ On the Caffarelli-Kohn-Nirenberg inequalities: sharp constants,
existence (and nonexistence), and symmetry of extremal functions}
\jour{Comm. Pure Appl.
Math.}
\vol{54}
\yr{2001}
\pages{229--258}
\endpaper

\bibitem{CM}
\by{\name{P.}{Cavaliere}  \et\name{Z.}{Mihula}}
\paper{Compactness of Sobolev-type trace operators}
\jour{Nonlinear Anal.}
\vol{183}
\yr{2019}
\pages{43--69}
\endpaper

\bibitem{Cianchi-Duke}
\by{\name{A.}{Cianchi}} \paper{Second-order derivatives and
rearrangements} \jour{Duke \ Math.\ J.} \vol{105} \yr{2000}
\pages{355--385}
\endpaper


\bibitem{Cianchi_AMPA}
\by{\name{A.}{Cianchi}} \paper{Symmetrization and second-order Sobolev inequalities}
\jour{Ann. Mat. Pura Appl.} \vol{183} \yr{2004}
\pages{45--77}
\endpaper

%

\bibitem{CKP}
\by{\name{A.}{Cianchi}, \name{R.}{Kerman} \et\name{L.}{Pick}}
\paper{Boundary trace inequalities and rearrangements} \jour{J. Anal. Math.} \vol{105} \yr{2008} \pages{ 241--265}
\endpaper

\bibitem{CP1998}
\by{\name{A.}{Cianchi}\et\name{L.}{Pick}}
\paper{Sobolev embeddings into $BMO$, $VMO$ and $L_\infty$}
\jour{Ark.\ Mat.}
\yr{1998}
\vol{36}
\pages{317--340}
\endpaper

\bibitem{CP2010}
\by{\name{A.}{Cianchi} \et\name{L.}{Pick}}
\paper{An optimal endpoint trace embedding} \jour{Ann. Inst. Fourier (Grenoble)} \vol{60} \yr{2010} \pages{ 939--951}
\endpaper

\bibitem{CP-Trans}
\by{\name{A.}{Cianchi} \et\name{L.}{Pick}}
\paper{Optimal Sobolev trace embeddings} \jour{Trans. Amer. Math. Soc.}
\vol{368} \yr{2016} \pages{8349--8382}
\endpaper

\bibitem{CPS}
\by{\name{A.}{Cianchi},
\name{L.}{Pick}\et\name{L.}{Slav\'{\i}kov\'a}} \paper{Higher-order
Sobolev  embeddings and isoperimetric inequalities} \jour{Adv.\
Math.} \vol{273} \yr{2015} \pages{568--650}
\endpaper

\bibitem{CPS_appl}
\by{\name{A.}{Cianchi},
\name{L.}{Pick}\et\name{L.}{Slav\'{\i}kov\'a}} \paper{Orlicz and Lorentz Sobolev  inequalities with measures} \jour{preprint}
\endprep

\bibitem{CR}
\by{\name{A.}{Cianchi} \et
\name{M.}{Randolfi}} \paper{On the modulus of continuity   of weakly differentiable functions} \jour{Indiana Univ. Math. J.} \vol{60} \yr{2011} \pages{1939--1973}
\endpaper

\bibitem{costea-mazya}
\by {\name{S.}{Costea}\et\name{V.G.}{Maz'ya}} \book {Conductor inequalities and criteria for Sobolev-Lorentz two-weight inequalities}
\publ{(English summary) Sobolev spaces in mathematics. II, 103–121,
Int. Math. Ser. (N. Y.), 9, Springer, New York, 2009}
\endbook

\bibitem{CR-1}
\by {\name{G.P.}{Curbera}\et\name{W.J.}{Ricker}}
\paper{Can optimal rearrangement invariant Sobolev imbeddings be further extended?}
\jour{Indiana Univ. Math. J.}
\vol{56}
\yr{2007}
\pages{1479--1497}
\endpaper

\bibitem{CR-2}
\by {\name{G.P.}{Curbera}\et\name{W.J.}{Ricker}}
\paper{Optimal domains for the kernel operator associated with Sobolev's inequality}
\jour{Studia Math.}
\vol{158}
\yr{2003}
\pages{131--152}
\endpaper

\bibitem{DS}
\by {\name{R.A.}{De Vore}\et\name{K.}{Scherer}}
\paper{Interpolation of linear operators on Sobolev spaces}
\jour{Ann.\ of Math.}
\vol{109}
\yr{1979}
\pages{583--599}
\endpaper

\bibitem{DEFT}
\by {\name{J.}{Dolbeault}, \name{M.J.}{Esteban} \name{S.}{Filippas} \et\name{S.}{Tertikas}}
\paper{Rigidity results with applications to
best constants and symmetry of Caffarelli-Kohn-Nirenberg and logarithmic Hardy inequalities}
\jour{Calc. Var. Partial Differ. Equations}
\vol{54}
\yr{2015}
\pages{2465--2481}
\endpaper

\bibitem{DEL}
\by {\name{J.}{Dolbeault}, \name{M.J.}{Esteban}\et\name{M.}{Loss}}
\paper{Rigidity versus symmetry breaking via nonlinear flows on cylinders and Euclidean spaces}
\jour{Invent. Math.}
\vol{206}
\yr{2016}
\pages{397--440}
\endpaper
%
%

\bibitem{EOP}
\by{\name{W.D.}{Evans}, \name{B.}{Opic}\et\name{L.}{Pick}}
\paper{Interpolation of integral operators on scales of
generalized Lorentz--Zygmund spaces}
\jour{Math.\ Nachr.}
\yr{1996}
\vol{182}
\pages{127--181}
\endpaper

\bibitem{FF}
\by{\name{H.}{Federer}\et \name{W.}{Fleming}} \paper{Normal and
integral currents} \jour{Annals of Math.} \vol{72} \yr{1960}
\pages{458--520}
\endpaper

\bibitem{FS}
\by{\name{V.}{Felli}\et \name{M.}{Schneider}} \paper{Perturbation results of critical elliptic equations of
Caffarelli-Kohn-Nirenberg type} \jour{J. Differ. Equations} \vol{191} \yr{2003}
\pages{121--142}
\endpaper

\bibitem{GMNO}
\by{\name{A.}{Gogatishvili}, \name{S.}{Moura}, \name{J.}{Neves}\et\name{B.}{Opic}}
\paper{Embeddings of Sobolev-type spaces into generalized H\"older spaces involving k-modulus of smoothness}
\jour{Ann. Mat. Pura Appl.}
\vol{194}
\yr{2015}
\pages{425--450}
\endpaper

\bibitem{GNO}
\by{\name{A.}{Gogatishvili}, \name{J.}{Neves}\et\name{B.}{Opic}}
\paper{Characterization of embeddings of
Sobolev-type spaces into generalized H\"older spaces defined by $L^p$-modulus of
smoothness}
\jour{J. Funct. Anal.}
\vol{276}
\yr{2019}
\pages{636--657}
\endpaper

\bibitem{Holik}
\by{\name{M.}{Hol\'ik}}
\paper{Reduction theorems for Sobolev embeddings into the spaces of H\"older, Morrey and Campanato type}
\jour{Math.\ Nachr.}
\vol{289}
\yr{2016}
\pages{1626--1635}
\endpaper

\bibitem{H}
\by{\name{T.}{Holmstedt}}
\paper{Interpolation of quasi-normed spaces}
\jour{Math.\ Scand.}
\vol{26}
\yr{1970}
\pages{177--199}
\endpaper

\bibitem{KP}
\by{{R. }{Kerman} \et\name{L.}{Pick}}
\paper{Optimal Sobolev embeddings} \jour{Forum Math.} \vol{18} \yr{2006} \pages{535--570}
\endpaper

REDUCTION PRINCIPLE FOR A CERTAIN CLASS OF KERNEL-TYPE
OPERATORS
DALIMIL PESA


\bibitem{FuscoLionsSbordone}
\by{\name{N.}{Fusco}, \name{P.-L.}{Lions}\et\name{C.}{Sbordone}}
\paper{Sobolev imbedding theorems in borderline cases}
\jour{Proc. Amer. Math. Soc.}
\vol{124, 2}
\yr{1996}
\pages{561--565}
\endpaper

\bibitem{GOP}
\by{\name{A.}{Gogatishvili}, \name{B.}{Opic}\et\name{L.}{Pick}}
\paper{Weighted inequalities for Hardy-type operators involving suprema}
\jour{Collect. Math.}
\vol{57, 3}
\yr{2006}
\pages{227--255}
\endpaper


\bibitem{H}
\by{\name{T.}{Holmstedt}}
\paper{Interpolation of quasi-normed spaces}
\jour{Math.\ Scand.}
\vol{26}
\yr{1970}
\pages{177--199}
\endpaper

\bibitem{KP}
\by{\name{R.}{Kerman} \et\name{L.}{Pick}}
\paper{Optimal Sobolev embeddings} \jour{Forum Math.} \vol{18} \yr{2006} \pages{535--570}
\endpaper


\bibitem{KK}
\by{\name{J.}{Kristensen} \et\name{M. V.}{Korobkov}}
\paper{The trace theorem, the Luzin N - and Morse-Sard properties for the sharp case of Sobolev-Lorentz mappings}
\jour{J. Geom. Anal.}
\vol{28} \yr{2018} \pages{2834--2856}
\endpaper
%

\bibitem{MaPe}
\by{\name{L.}{Maligranda} \et\name{L.-E.}{Persson}}
\paper{Generalized duality of some Banach function spaces} \jour{Nederl. Akad. Wetensch.
Indag. Math.} \vol{51} \yr{1989} \pages{323--338}
\endpaper


%
\bibitem{Ma1960}
\by {\name{V. G.}{Maz'ya}} \paper {Classes of regions and imbedding
theorems for function spaces} \jour {Dokl. Akad. Nauk. SSSR}
\vol{133} \yr{1960} \pages{527--530 (Russian); English translation:
Soviet Math. Dokl. {\bf 1} (1960), 882--885}
\endpaper

\bibitem{Ma1961}
\by {\name{V. G.}{Maz'ya}} \paper {On p-conductivity and theorems on
embedding certain functional spaces into a C-space} \jour {Dokl.
Akad. Nauk. SSSR} \vol{140} \yr{1961} \pages{299--302 (Russian);
English translation: Soviet Math. Dokl. {\bf 3} (1962)}
\endpaper

\bibitem{Mabook}
\by{\name{V. G.}{Maz'ya}} \book{Sobolev spaces with applications to
elliptic partial differential equations} \publ{Springer, Berlin,
2011}
\endbook

\bibitem{Moser}
\by {\name{J.}{Moser}} \paper {A sharp form of an inequality by
Trudinger} \jour {Indiana Univ. Math. J.} \vol{20} \yr{1971}
\pages{1077--1092}
\endpaper

\bibitem{Musil-dissert}
\by {\name{V.}{Musil}}
\book {Classical operators of harmonic analysis in Orlicz spaces}
\publ{Doctoral Thesis, Prague, 2018}
\endbook

\bibitem{ONeil}
\by{\name{R.}{O'Neil}}
\paper{Convolution operators and L(p,q) spaces}
\jour{Duke Math. J.}
\yr{1963}
\vol{30}
\pages{129--142}
\endpaper

\bibitem{glz}
\by{\name{B.}{Opic}\et\name{L.}{Pick}}
\paper{On generalized Lorentz-Zygmund spaces}
\jour{Math.\ Ineq. Appl.}
\yr{1999}
\vol{2}
\pages{391--467}
\endpaper

\bibitem{Peetre}
\by{\name{J.}{Peetre}}
\paper{Espaces d'interpolation et th\'eor\'eme de Soboleff}
\jour{Ann. Inst. Fourier (Grenoble)}
\yr{1966}
\vol{16, 1}
\pages{279--317}
\endpaper

\bibitem{PKJF}
\by{\name{L.}{Pick}, \name{A.}{Kufner}, \name{O.}{John} \et \name{S.}{Fu\v c\'{\i}k}}
\book{Function Spaces, Volume 1}
\publ{2nd Revised and Extended Edition, De Gruyter Series in Nonlinear Analysis and Applications 14, De Gruyter, Berlin 2013}
\endbook